\documentclass[11pt, notitlepage]{article}
\makeatletter\@addtoreset{section}{part}\makeatother%
\usepackage{amssymb}
\usepackage[margin=2.5cm]{geometry}
\usepackage{amsmath}
\usepackage{amsthm}
\usepackage{bbm}
\usepackage{amsfonts}
\usepackage{mathtools}
\usepackage{graphicx} 
\usepackage{subcaption} 
\usepackage{relsize}
\usepackage[numbers]{natbib}
\usepackage{titling}
\usepackage[]{algorithm2e}
\usepackage{enumitem}
\usepackage{hyperref}

\usepackage{framed}
\usepackage[title]{appendix}
\appendixtitleon
\appendixtitletocon
\newcommand{\subtitle}[1]{%
  \posttitle{%
    \par\end{center}
    \begin{center}\large#1\end{center}
    \vskip0.5em}%
}

\usepackage[dvipsnames]{xcolor}

\newcommand{\xddots}{%
  \raise 4pt \hbox {.}
  \mkern 6mu
  \raise 1pt \hbox {.}
  \mkern 6mu
  \raise -2pt \hbox {.}
}
\usepackage{hyperref}
\hypersetup{
colorlinks=true,
citecolor=OliveGreen, 
linkcolor=blue 
}

\title{Kac's Process with Hard Potentials and a Moderate Angular Singularity}

\author{Daniel Heydecker}
\date{\today}
\newcommand{\EE}{\ensuremath{\mathbb{E}}}

\newcommand{\Nn}{\ensuremath{\mathcal{N}}}
\newcommand{\PP}{\ensuremath{\mathbb{P}}}
\newcommand{\Ss}{\ensuremath{\mathcal{S}}}

\newcommand{\T}{\ensuremath{\mathcal{T}}}
\newcommand{\Q}{\ensuremath{\mathcal{Q}}}

\newcommand{\hs}{\ensuremath{\hspace{1cm}}}
\newcommand{\h}{\ensuremath{\hspace{0.1cm}}}

\definecolor{do}{HTML}{9932cc}

\newcommand{\RRd}{\ensuremath{\mathbb{R}^d}}
\newcommand{\SSd}{\ensuremath{\mathbb{S}^{d-2}}}

\newtheorem{theorem}{Theorem}
\newtheorem{lemma}{Lemma}[section]
\newtheorem{cor}[lemma]{Corollary}
\newtheorem{rmk}[lemma]{Remark}
\newtheorem{proposition}[lemma]{Proposition}
\newtheorem{remark}[theorem]{Remark}
\newtheorem{defn}[lemma]{Definition}

\begin{document}
\maketitle

\begin{abstract}
 We investigate Kac's many-particle stochastic model of gas dynamics in the case of hard potentials with a moderate angular singularity, and show that the noncutoff particle system can be obtained as the limit of cutoff systems, with a rate independent of the number of particles $N$. As consequences, we obtain a wellposedness result for the corresponding Boltzmann equation, and convergence of the particle system in the limit $N\rightarrow \infty$.

\end{abstract}

\section{Introduction \& Main Results} Let us consider Kac's model \cite{kac1956foundations} for the behaviour of a dilute gas.  We consider an ensemble of $N$ indistinguishable particles, with velocities $V^1_t, ..., V^N_t \in \mathbb{R}^d$, which are are encoded in the empirical velocity distribution $\mu^N_t=N^{-1}\sum_{i=1}^N \delta_{V^i_t}$. The rates of each possible collision are governed by a \emph{collision kernel} $B:\mathbb{R}^d\times \mathbb{S}^{d-1}\rightarrow [0,\infty)$, which reflects the physics of the underlying system, and the dynamics can be described informally as follows. For every unordered pair of particles with velocities $v,v_\star$, the velocities change to  \begin{equation}\label{eq: postcollisional}
     v\mapsto v'=\frac{v+v_\star}{2}+\frac{|v-v_\star|}{2}\sigma; \hs v_\star\mapsto v'_\star=\frac{v+v_\star}{2}-\frac{|v-v_\star|}{2}\sigma;
 \end{equation} at a rate $2B(v-v_\star,\sigma)d\sigma/N$. These collisions preserve energy and momentum, so that the total energy $\langle |v|^2, \mu^N_t\rangle$ and momentum $\langle v, \mu^N_t\rangle$ are preserved as time runs. Let us define the \emph{Boltzmann Sphere} $\mathcal{S}$ as those measures on $\mathbb{R}^d$ with normalised mass, momentum and energy:\footnote{here, and throughout, angle brackets $\langle\cdot ,\cdot \rangle$ denote integration against a measure, and $v$ denotes the identity function on $\mathbb{R}^d$.} \begin{equation}\label{eq: normalisation} \langle 1, \mu \rangle =1; \hspace{0.5cm} \langle v, \mu \rangle =0; \hspace{0.5cm} \langle |v|^2, \mu \rangle = 1.\end{equation} We also write $\mathcal{S}_N$ for the subspace of $\mathcal{S}$ consisting of normalised empirical measures on $N$ points; we equip both spaces with the (Monge-Kantorovich-)Wasserstein distance $w_1$ recalled below (\ref{eq: usual WP}, \ref{eq: dual W1}). Due to Gallilean invariance, the collision kernel is assumed to be of the form $B(v,\sigma)=B(|v|, \cos \theta)$, where $\theta$ is  the deflection angle given by $\cos \theta=\sigma\cdot v/|v|$. We will consider the case of \emph{noncutoff hard potentials}, where the kernel has the form \begin{equation}\label{eq: form of B}
     B(v,\sigma)=|v|^\gamma (\sin \theta)^{2-d}
    \beta(\theta); \quad \beta(\theta)\sim \theta^{-1-\nu}\text{ as } \theta\downarrow 0
 \end{equation} with $\gamma \in [0,1], \nu\in [0,1)$. We will also assume that $\beta$ takes the form $\beta(\theta)=b(\cos \theta)$, for a convex function $b:(-1,1)\rightarrow [0,\infty)$; (\ref{eq: form of B}) then rearranges to $b(x)\sim (1-x)^{-1/2-\nu/2}$ as $x\uparrow 1$. Thanks to the symmetry of collisions, we may assume further that $b$ is supported on $[0,1)$; see the discussion in Alexandre  et al. \cite{alexandre2000entropy}. \bigskip \\ In $d=3$, such kernels arise when modelling particles interacting through a repulsive potential $V(r)=r^{-s}, s>5$, with $\gamma=\frac{s-5}{s-s}, \nu=\frac{2}{s-2}$. Importantly, the kernel is not integrable, due to to the abundance of \emph{grazing collisions}, reflected in the non-integrable singularity of $\beta$ as $\theta\downarrow 0$.  The cases we consider have a mild angular singularity, so that \begin{equation} \label{eq: mild singularity}  \int_{\mathbb{S}^{d-1}} B(v,\sigma)d\sigma=\infty; \hs 
     \int_{\mathbb{S}^{d-1}} \theta B(v,\sigma)d\sigma<\infty.
 \end{equation}The divergence of the total rate $\int B d\sigma$ implies that each pair of particles undergoes infinitely many collisions on any nontrivial time interval, and there is work to be done in understanding the informal description of the dynamics above. Formally, a Kac process is a Markov process $\mu^N_t$ in $(\mathcal{S}_N,w_1)$, with c\`adl\`ag paths and generator, defined for Lipschitz $F:(\mathcal{S}_N, w_1)\rightarrow \mathbb{R}$, \begin{equation}\begin{split}\label{eq: generator}
     (\mathcal{G}^NF)(\mu^N)=&N\int_{\mathbb{R}^d\times\mathbb{R}^d\times \mathbb{S}^{d-1}} (F(\mu^{N, v,v_\star,\sigma})-F(\mu^N))\\& \hs \hs \dots\times B(v-v_\star,\sigma)\mu^N(dv)\mu^N(dv_\star)d\sigma. \end{split}
 \end{equation} Here, we have written  $\mu^{N,v,v_\star,\sigma}=\mu^N+N^{-1}(\delta_{v'}+\delta_{v'_\star}-\delta_v-\delta_{v_\star})$ for the measure $\mu^N$ replacing precollisional velocities $v,v_\star\in \text{supp}(\mu^N)$ by postcollisional velocities $v',v_\star'$, given in terms of $v,v',\sigma$ by (\ref{eq: postcollisional}). One can check that $w_1(\mu^{N,v,v_\star,\sigma}, \mu^N)\le 2|v-v_\star|\theta/N$, and so, thanks to (\ref{eq: mild singularity}), the integral written above is convergent for Lipschitz $F$. However, the total rate is still infinite, and it is not \emph{\`a priori} clear that the associated martingale problem is well-posed.

 \paragraph{Labelled vs Unlabelled Dynamics.} Let us briefly mention that it is also possible to work with a \emph{labelled} Kac process $\mathcal{V}^N_t=(V^1_t,...V^N_t) \in (\mathbb{R}^d)^N$, where each particle is assigned a label $i=1,...,N$; in this case, the same normalisation (\ref{eq: normalisation}) now defines a subspace $\mathbb{S}_N\subset (\mathbb{R}^d)^N$. The interchangability of the particles means that such processes have a $\text{Sym}(N)$-symmetry by exchanging the labels; we identify the orbit of $\mathcal{V}^N\in \mathbb{S}_N$ with its empirical measures, and $\mathcal{S}_N\equiv (\mathbb{R}^d)^N/\text{Sym}(N)$. Let us also write $\theta_N:\mathbb{S}_N\rightarrow \mathcal{S}_N$ for the map taking a $N$ velocities $\mathcal{V}^N=(V^1,...,V^N)\in \mathbb{S}_N$ to the associated empirical measure $\mu^N=N^{-1}\sum_i\delta_{V^i}.$ \bigskip \\ Our results are most naturally phrased at the level of empirical measures $\mu^N_t$, especially as this emphasises the important $N$-independence and makes the connections to other results clear. However, many of the intermediate calculations are more naturally phrased in terms of the labelled process; in section \ref{sec: JSDE}, we will introduce a formulation of the labelled process in terms of a stochastic differential equation with jumps, and formalise relationship to the unlabelled process.  \paragraph{Grad's Angular Cutoff.} A convenient simplification of the dynamics above is Grad's angular cutoff, which truncates $B$ near small values of $\theta$; one might hope that this truncation preserves, in some meaningful sense, the physics of the system under consideration. Let us define, for $K>0$, \begin{equation}\label{eq: cutoff kernel}
     B_K(v,\sigma)=B(v,\sigma)1(\theta>\theta_0(K))
 \end{equation} where $\theta_0(K)$ is chosen so that $\int_{\mathbb{S}^{d-1}} B_K(v,\theta)d\sigma=K|v|^\gamma.$ We can now consider the $K$-cutoff Kac processes $\mathcal{V}^{N,K}_t$ on these kernels, with generator defined analagously to (\ref{eq: generator}). In this case, the total rate is finite, and the associated martingale problem has uniqueness in law. The central approach of this paper, summarised in Theorem \ref{thrm: approximate with cutoff}, is that any noncutoff Kac process $\mu^N_t$ on $N$ particles can be obtained as the limit of cutoff process $\mu^{N,K}_t$ of cutoff processes as $K\rightarrow \infty$, in an $N$-uniform way. This will imply a similar result for the cutoff and non-cutoff Boltzmann equations described below, and demonstrates the well-posedness of the martingale problem for (\ref{eq: generator}). 
 \paragraph{Measure Solutions to the Boltzmann Equation.} Kac introduced the stochastic system described above in an effort to justify the spatially homogeneous Boltzmann Equation; following previous works \cite{fournier2009well,mischler2013kac,norris2016consistency}, we will consider measure-valued solutions. For a measure $\mu$ with finite second moment, we define the Boltzmann collision operator by specifying, for all Lipschitz $f:\RRd\rightarrow \mathbb{R}$, \begin{equation} \begin{split} \langle f, Q(\mu)\rangle =\int_{\mathbb{R}^d\times\mathbb{R}^d\times\mathbb{S}^{d-1}}& \left\{f(v')+f(v_\star)-f(v)-f(v_\star)\right\}\\&\hs \hs \dots\times B(v-v_\star,d\sigma)\mu(dv)\mu(dv_\star). \end{split} \end{equation} The same argument as below (\ref{eq: generator}) shows that this integral is well-defined for Lipschitz $f$, but $Q$ must be interpreted as a distribution, for example in the negative Sobolev space $W^{-1,\infty}(\mathbb{R}^d)=W^{1,\infty}(\mathbb{R}^d)^\star$, rather than a signed measure.  We say that a family $(\mu_t)_{t\geq 0}$ of measures in $\mathcal{S}$ satisfies the \emph{Boltzmann equation} if, for any Lipschitz $f$ of compact support, \begin{equation} \tag{BE}\label{BE} \forall t \geq 0 \hspace{1cm} \langle f, \mu_t \rangle =\langle f, \mu_0 \rangle +\int_0^t \langle f, Q(\mu_s)\rangle ds. \end{equation} Replacing $Q$ with the cutoff collision operator $Q_K$, with $B_K$ in place of $B$, we obtain the corresponding $K$-cutoff Boltzmann equations: \begin{equation} \tag{BE$_K$} \label{eq: BEK} \mu_t=\mu_0+\int_0^t Q_K(\mu_s)ds
 \end{equation} which may be interpreted in a similar way to (\ref{BE}) above. \paragraph{Notation regarding Moments.} Our estimates will frequently include moments of the Kac process or Boltzmann flow, and it is convenient to introduce notation to deal with this.  Let us define $\mathcal{S}^k$ as those measures $\mu\in \mathcal{S}$ with a finite $k^\text{th}$ moment $\langle |v|^k, \mu\rangle$, and define \begin{equation}\Lambda_k(\mu):=\langle(1+|v|^2)^{k/2}, \mu\rangle; \hs \Lambda_k(\mu, \nu):=\max(\Lambda_k(\mu), \Lambda_k(\nu)). \end{equation}  With this notation, we define  \begin{equation} \label{eq: definition of SKA} \mathcal{S}^k_a := \{\mu \in \mathcal{S}: \Lambda_k(\mu)\leq a\}. \end{equation}

 \paragraph{A Weighted Wasserstein$_2$ Distance.} When working with labelled dynamics, we will use the following semimetric. We write, for $p\ge 0$, \begin{equation}
     d_p(v,w)=(1+|v|^p+|w|^p)^{1/2}|v-w|.
 \end{equation} At the level of unlabelled dynamics, or the limiting equation, our results are most naturally phrased in terms of the following semimetric of Wasserstein type.  For $\mu, \nu\in \mathcal{S}^{p+2}$, we define $W_p$ as an optimal transportation cost \begin{equation}\label{eq: define Wp} W_{p}(\mu,\nu)=\inf_{\pi \in \Pi(\mu,\nu)} \left(\int_{\RRd\times\RRd} d_p(v,w)^2 \pi(dv,dw)\right)^{1/2} \end{equation} where $\Pi(\mu, \nu)$ denotes the set of couplings of $\mu$ and $\nu$; we emphasise that $p$ here denotes a weighting, and not an exponent in the usual sense of Wasserstein$_p$ metrics. For completeness, we will discuss the properties of these distances, including their relationship to other metrics of Wasserstein type, in Section \ref{sec: metrics}. 
 
\subsection{Main Results} With this notation and terminology fixed, we can now state our main results. The central result, from which our other results will follow, is the following coupling of cutoff and noncutoff Kac processes. \begin{theorem}[Tanaka Coupling of Kac Processes]\label{thrm: approximate with cutoff} Let $B$ be a kernel of the form described above. There exists $p_0=p_0(B,d)$ and, for $p>p_0$, $K_0=K_0(B,p,d)$ such that, whenever $p>p_0$ and $K>K_0$, there exists $C=C(B,p,d)$ such that the following holds. \bigskip \\  
Fix $a_1, a_2, a_3\ge 1$, $N\ge 2$, and let $\mu^N_0, \widetilde{\mu}^{N,K}_0\in \mathcal{S}_N$ be empirical measures satisfing \begin{equation} \Lambda_{p+\gamma}\left(\mu^N_0,\widetilde{\mu}^{N,K}_0\right)\le a_1; \hs  \Lambda_{l}\left(\mu^N_0,\widetilde{\mu}^{N,K}_0\right)\le a_2; \hs  \Lambda_{q}\left(\mu^N_0,\widetilde{\mu}^{N,K}_0\right)\le a_3 \end{equation}where $l=p+2+\gamma, q=2p+4+2\gamma$. Then there exists a coupling of a noncutoff Kac process $\mu^N_t$ starting at $\mu^N_0$ and a $K$-cutoff Kac process $\widetilde{\mu}^{N,K}_t$ starting at $\widetilde{\mu}^{N,K}_0$ such that, for all $t\ge 0$,  \begin{equation}\begin{split}
    \EE\left[W_p^2\left(\mu^N_t,\widetilde{\mu}^{N,K}_t\right)\right] &\le e^{Ca_1(1+t)}\left(W_p^2\left(\mu^N_0,\widetilde{\mu}^{N,K}_0\right)+a_2K^{1-1/\nu}\right)\\& \hspace{5.5cm}+Ca_3^2tN^{-1/2}\end{split}
\end{equation} and, for all $t_\mathrm{fin}\ge 0$,  \begin{equation}\begin{split}\label{eq: statement of LU tanaka} \EE\left[\sup_{t\le t_\mathrm{fin}} \h W_p^2\left(\mu^N_t,\widetilde{\mu}^{N,K}_t\right)\right]& \le e^{Ca_1(1+t_\mathrm{fin})}\bigg(W_p^2\left(\mu^N_0,\widetilde{\mu}^{N,K}_0\right)+a_2tK^{1-1/\nu}\\& \hspace{3.5cm}\dots+\frac{Ca_3^2(1+t_\mathrm{fin})^{2}}{N^{1/2}}\bigg)\end{split}\end{equation}
\end{theorem} We now exhibit two results which arise as a consequence of this coupling. Our first result uses our analysis of the Kac process to study the well-posedness of the limit equation (\ref{BE}).  \begin{theorem}[Wasserstein Stability of the Boltzmann Flow]\label{thrm: wellposedness} Let $B$ be a kernel of the form described above, and let $p>p_0$, $l=p+2+\gamma$ be as in Theorem \ref{thrm: approximate with cutoff}. For any $p'>p+2$ and $\mu_0\in \Ss^{p'}$, there exists a unique energy-conserving solution $(\mu_t)_{t\ge 0}$ to the Boltzmann equation starting at $\mu_0$, which we write as $\mu_t=\phi_t(\mu_0)$. Moreover, for some constant $C=C(B,p,d)$, \begin{enumerate}[label=\roman*).]
    \item Whenever $\mu_0, \nu_0 \in \mathcal{S}^{p'}$ satisfy the moment bound $\Lambda_{p+\gamma}(\mu_0, \nu_0)\le a$, we have the continuity estimate \begin{equation}\label{eq: big wellposedness} W_p\left(\phi_t(\mu_0),\phi_t(\nu_0)\right)\le e^{Ca(1+t)}W_p(\mu_0, \nu_0). \end{equation}  \item Whenever $\mu_0 \in \mathcal{S}^l$, the solution $\phi_t(\mu)$ is the $W_p$-limit of the solutions $\phi^K_t(\mu_0)$ to the $K$-cutoff Boltzmann Equations (\ref{eq: BEK}) starting at $\mu_0$, as the cutoff parameter $K\rightarrow \infty$. More precisely, if $\Lambda_{p+\gamma}(\mu_0)\le a_1, \Lambda_l(\mu_0)\le a_2$, for some $a_1, a_2\ge 1$, then we have \begin{equation}
        W_p(\phi^K_t(\mu_0), \phi_t(\mu_0))\le e^{Ca_1(1+t)}a_2tK^{1-1/\nu}.
    \end{equation}
\end{enumerate}  \end{theorem} We note that this is a much stronger well-posedness estimate than exists in the literature; see the discussion in the literature review below. \bigskip \\ Our final result is to study the convergence of the noncutoff Kac process in the large number limit $N\rightarrow \infty$.\begin{theorem}\label{thrm: Kac Process} Let $B$ be a kernel of the form described above. For all $N$, the $N$-particle Kac process defined by the generator (\ref{eq: generator}) has uniqueness in law. Moreover, if $p,q$ are as in Theorem \ref{thrm: approximate with cutoff},  $a\ge 1$ and $t_\mathrm{fin}\ge 0$, then whenever $\mu_0\in \mathcal{S}^q$ has a moment $\Lambda_q(\mu_0)\le a$ and $\mu^N_t$ is a $N$-particle Kac process with initial data satisfying $\Lambda_q(\mu^N_0)\le a$ almost surely then we have the estimate \begin{equation} \label{eq: convergence of Kac process} \EE\left[\sup_{t\le t_\mathrm{fin}} W_p\left(\phi_t(\mu_0), \mu^N_t\right)\right] \le e^{Ca(1+t_\mathrm{fin})}\left((\log N)^{1/2-1/2\nu}+\EE\left[W_p(\mu^N_0, \mu_0)\right] \right).  \end{equation} \end{theorem}
This estimate may also be understood as proving \emph{propagation of chaos}; this will be discussed in the literature review below.
\subsection{Plan of the Paper} Our programme will be as follows. \begin{enumerate}[label=\roman*).]\item The remainder of this section is a literature review, discussing the existing results and techniques and the relationships to the current work.
\item Section \ref{sec: metrics} is a brief aside on metrics on probability measures. We record some properties of the semimetrics $W_p$ introduced above, and discuss their relationships with the usual Wasserstein$_p$ metrics $w_p$ and with a family of metrics $\mathfrak{w}_\gamma$ which are useful in applying the results of \cite{norris2016consistency}. We also introduce, for future reference, a discrete approximation scheme in $W_p$.
\item In Section \ref{sec: JSDE}, we introduce an alternative parametrisation of possible jumps, and hence formulate the labelled Kac process discussed above as the solution to a stochastic differential equation driven by Poisson random measures. We state results on the well-posedness of the resulting stochastic differential equation and its relationship to the unlabelled Kac process; the proofs are deferred to Appendix \ref{sec: pf_wellposedness} for ease of readability. 
\item Section \ref{sec: moments} collects some results from the literature concerning the moments of the Kac process and the Boltzmann equation. We also prove a novel `concentration of moments' result (Lemma \ref{lemma: concentration of moments}) which shows that the moments of a Kac process remain in a fixed interval with probability converging to $1$ in the large number limit $N\rightarrow \infty.$
\item Section \ref{sec: cutoff case} analyses the convergence of the cutoff Kac process to the cutoff Boltzmann equation, with slight modification of the techniques of \cite{norris2016consistency}. 
\item Section \ref{sec: TK} introduces the key coupling of cutoff and noncutoff Kac processes. We state the key result (Lemma \ref{lemma: big calculation}) on the divergence of the couplings in the $W_p$ distance, which leads to Theorem \ref{thrm: approximate with cutoff}. The proof of the lemma is deferred until Section \ref{sec: main calculations} for ease of readability.
\item In Section \ref{sec: T2}, we carefully propagate the previous coupling to the Boltzmann equation, use the convergence of the cutoff Kac process proven in Section \ref{sec: cutoff case}. This proves Theorem \ref{thrm: wellposedness}.
\item In Section \ref{sec: proof_of_poc}, we combine the previous results to deduce Theorem \ref{thrm: Kac Process}.
\item Section \ref{sec: main calculations} presents the main calculations on our Tanaka coupling, deferred from Section \ref{sec: TK}.
\item Finally, Appendix \ref{sec: pf_wellposedness} deals with some technical issues concerning the well-posedness for labelled and unlabelled Kac processes, and the relationships between these.  \end{enumerate}

\subsection{Literature Review} We will now briefly discuss related works and their relationship to our work.
\paragraph{1. Tanaka's Coupling.} The key idea in our analysis is a coupling pioneered by Tanaka \cite{tanaka1978probabilistic} in the case of cutoff Maxwell molecules $(\gamma=0, \nu<0$), who interpreted the Boltzmann equation in terms of a stochastic differential equation describing a `typical' particle; this was generalised by Fournier and M\'el\'eard \cite{fournier2001markov} to include the cases without cutoff, and for non-Maxwellian molecules and used to show uniqueness for the Boltzmann equation with Maxwell molecules \cite{tanaka2002uniqueness}. Since then, this coupling and analytic counterparts have been a popular idea in the treatment of the Boltzmann equation \cite{bolley2007tanaka,fournier2008uniqueness,fournier2009well,fournier2016rate,pulvirenti1996theory,toscani1999probability}. \bigskip \\ Let us mention two particular works to which our approach can be compared. The main calculations in Sections \ref{sec: TK}, \ref{sec: main calculations} were inspired by Fournier and Mischler \cite{fournier2016rate} on the Nanbu particle system, in which only one particle jumps at a time. In our notation, the cited paper produces estimates in $W_0$; the major novelty of this work is that, by working in $W_p$ for $p$ large enough, and working with the symmetric Kac process, we are able to obtain a desirable cancellation of `bad' terms. Let us also remark that that the main result of Rousset \cite{rousset2014n} is very similar to Theorem \ref{thrm: approximate with cutoff} in obtaining a coupling of Kac processes with error uniform in $N$ in the case of Maxwell molecules, although the ends to which this is used are orthogonal to our programme. In this case, strong results concerning existence and uniqueness are already known and the coupling is used to investigate relaxation to equilibrium in the limit $t\rightarrow \infty$. In our case, we are unable to make the coupling uniform in time to investigate relaxation to equilibrium, but instead derive results concerning the Boltzmann equation (as discussed below) by taking the limits $K\rightarrow \infty, N\rightarrow \infty$ on fixed time intervals.
\paragraph{2. Well-Posedness of the Boltzmann Equation.} The results expressed in Theorem \ref{thrm: wellposedness} add to a long list of results concerning the well-posedness and stability of the Boltzmann equation. In the case of Maxwell molecules, including the noncutoff case, we refer the reader to \cite{toscani1999probability,tanaka2002uniqueness}; in the case of hard spheres ($\gamma=1, \nu<0$), let us mention the works \cite{arkeryd1972boltzmann,escobedo2010scalings,lu1999conservation,lu2012measure,mischler1999spatially}. Most recently, Mischler and Mouhot \cite{mischler2013kac} prove very strong `twice-differentiability' and exponential stability of the Boltzmann equation in the hard-spheres case measured in total variation distance, and the author obtained a uniform-in-time Wasserstein stability result in a previous work \cite{heydecker2019pathwise}. \bigskip \\ For the case of noncutoff hard potentials, the theory is substantially less complete. Fournier \cite{fournier2006uniqueness} examined the case where $|v-v_\star|^\gamma$ is replaced by a bounded function $\Phi$, and results for the case of full hard potentials have been found by Desvilettes and Mouhout \cite{desvillettes2009stability} and in the case of measure solutions by Fournier and Mouhout \cite{fournier2009well}. Let us note that the uniqueness and stability statement  in Theorem \ref{thrm: wellposedness} assume sonly a finite number of moments, rather than a finite exponential moment $\langle e^{\epsilon |v|^\gamma}, \mu_0\rangle<\infty$ as does the result of \cite{fournier2009well}, which is recalled in Proposition \ref{prop: weak wellposedness} below; correspondingly, the quantitative stability result is stronger. The result of \cite{desvillettes2009stability} requires the initial data $\mu_0$ to have a density $\frac{d\mu_0}{dv}\in W^{1,1}(\mathbb{R}^d,dv)$, and so requires fewer moments than our results but much more regularity.
\paragraph{3. Propagation of Chaos for the Kac Process.} The sense in which Kac first proposed to relate his stochastic process to Boltzmann's equation is through the \emph{propagation of chaos}: he proposed that, if $\mathcal{V}^N_t$ is a labelled Kac process, with symmetric initial conditions and $\mu_0\in \mathcal{S}$ is such that marginal distribution of $(V^1_0,...,V^k_0)$ is approximately $\mu_0^{\otimes k}$, then this approximation is propagated through time: the law of $(V^1_t,...V^k_t)$ is approximately $\phi_t(\mu_0)^{\otimes k}$, for any fixed $k, t$ in the regime where $N$ is large, and where the approximation is understood as the weak topology of measures on $(\RRd)^k$. This chaoticity property is equivalent to the convergence of the empirical measures \cite{sznitman1991topics}; quantitatively, the same arguments as in \cite{mischler2013kac,heydecker2019pathwise} show how the conclusion of Theorem \ref{thrm: Kac Process} can be viewed as a quantitative estimate of this approximation. We now mention some existing works in this direction: \begin{enumerate}[label=\roman*).] \item For Maxwell molecules, results in this direction were obtained by McKean \cite{mckean1975fluctuations}, Graham and M\'el\'eard \cite{graham1997stochastic} and Desvilettes, Graham and M\'el\'eard \cite{desvillettes1999probabilistic}. Strong results were obtained by Mischler and Mouhot \cite{mischler2013kac}, and close to optimal results were found by Cortez and Fontbona \cite{cortez2018quantitative}. \item Regarding the case of hard spheres, Mischler and Mouhot \cite{mischler2013kac} obtained results decaying as $(\log N)^{-r}$ for some $r>0$. Norris \cite{norris2016consistency} obtained results with optimal $N$ dependence, replacing the right-hand side in Theorem \ref{thrm: Kac Process} with the optimal $N$ dependence $N^{-1/d}$ but which are not uniform in time, and the author \cite{heydecker2019pathwise} obtained results with close-to-optimal $N$ dependence and which are uniform in time.
\item For the case of non-cutoff hard potentials, we are not aware of any results on the Kac process, but mention some works on related models. Fournier and Guilin \cite{fournier2015kac} consider a related particle system which approximates the Landau equation for hard potentials, and Fournier \cite{fournier2016propagation} deals with this model for soft potentials. The work \cite{fournier2016rate} which we have already mentioned considers the asymmetric \emph{Nanbu process} in which only one particle jumps at a time, and shows propagation of chaos for this system for Maxwell molecules and hard potentials; a recent work of Salem \cite{salem2019propagation} extends this to the case of soft potentials with a moderate angular singularity. \end{enumerate}

Our work combines ideas of Fournier and Mischler \cite{fournier2016rate} and Norris \cite{norris2016consistency}. Theorem \ref{thrm: approximate with cutoff} is based on the calculations of \cite{fournier2016rate}; in order to deduce Theorems \ref{thrm: wellposedness}, \ref{thrm: Kac Process}, we show how the arguments of \cite{norris2016consistency} apply to the cutoff Kac processes. Our result on the cutoff Kac process, presented in Lemma \ref{lemma: convergence of cutoff KP} is marginally stronger than the corresponding result, replacing results with `high probability' with $L^1$ estimates thanks to the concentration of moments result Lemma \ref{lemma: concentration of moments}. We also remark that the rate obtained in Theorem \ref{thrm: Kac Process} is equivalent to that of \cite{desvillettes1999probabilistic}, and is likely very far from optimal; it may be possible to improve on this by using the \emph{regularising effect of grazing collisions} \cite{alexandre1999remarks,alexandre2000entropy,desvillettes2005smoothness} to improve the estimates in Section \ref{sec: cutoff case}, but we will not explore this here.\\ 

Let us remark on the two philosophies of approach throughout these works. Kac \cite{kac1956foundations} initially introduced the stochastic process as a proxy with which to study the Boltzmann equation in the belief that a high-dimensional stochastic evolution may be easier to study than the limiting partial differential equation. To some extent, the opposite has been true; Mischler and Mouhout \cite{mischler2013kac} adopted a `top-down' approach, in which a detailed analyses of the Boltzmann equation, in the cases of (noncutoff) Maxwell molecules and hard spheres, were used to deduce results on Kac's process. By contrast, in the case we are interested in, comparatively few results are known for the Boltzmann equation (see above), and our approach is `bottom up': a careful analysis of Kac's process is used to study the Boltzmann equation.

\paragraph{Acknowledgements.} I am very grateful to my supervisor, James Norris, for the suggestion of this project and for several useful conversations, and to Nicolas Fournier for his comments and corrections. This work was supported by the UK Engineering and Physical Sciences Research Council (EPSRC) grant EP/L016516/1 for the University of Cambridge Centre for Doctoral Training, the Cambridge Centre for Analysis

\section{Metrics on Probability Measures}\label{sec: metrics}  We will briefly mention some properties of the semimetrics $W_p$ defined in (\ref{eq: define Wp}), and how they relate to the more common Wasserstein$_p$ metrics $w_p$. We will also introduce a class of metrics $\mathfrak{w}_\alpha, 0<\alpha\le 1$ which will be useful in applying the results of Norris \cite{norris2016consistency} to the cutoff Kac processes. For completeness, we will sketch the relationships between these classes of metrics; the proofs involved are at lease one of standard and elementary, and are therefore omitted. We also sketch a discrete approximation scheme in the metric $W_p$, details of which can be found in Norris \cite{norris2016consistency}. \bigskip \\  First, let us remark that $d_p$ does not satisfy the triangle inequality for $p>0$, but it is straightforward to check that the function $\delta_p(v,w)=|(1+|v|^p)^{1/2}v-(1+|w|^p)^{1/2}w|$ defines a metric, and that $d_p/\delta_p$ is bounded above, and away from $0$. It follows that $W_p$ is a semimetric, in the sense that one replaces the triangle inequality with a $C$-relaxed triangle inequality, for some new constant $C=C(p)$: for all $\mu, \nu, \xi\in \mathcal{S}^{p+2}$, we have  \begin{equation}\label{eq: triangle inequality}
     W_p(\mu, \nu)\le C_p\left(W_p(\mu, \xi)+W_p(\xi, \nu)\right).
 \end{equation} For completeness, let us recall usual Wasserstein$_p$ metrics $w_p$, defined for $p\ge 1$, $\mu, \nu\in \mathcal{S}^p$ by \begin{equation} \label{eq: usual WP}
     w_p(\mu, \nu)=\inf_{\pi \in \Pi(\mu, \nu)}\left(\int_{\mathbb{R}^d\times\mathbb{R}^d}|v-w|^p\pi(dv,dw)\right)^{1/p}.
 \end{equation} We also define a family of metrics $\mathfrak{w}_\gamma, 0<\gamma\le 1$ on $\mathcal{S}$ as follows. For $f:\RRd\rightarrow \mathbb{R}$, define $\widehat{f}(v)=f(v)/(1+|v|^2)$, and the $\gamma$-H\"older norm
 \begin{equation} \|f\|_{0,\gamma}:=\max\left(\sup_{v}|f|(v), \hspace{0.1cm} \sup_{v\neq w} \frac{|f(v)-f(w)|}{|v-w|^\gamma}\right). \end{equation} We write $\mathcal{A}_\gamma$ for the space of weighted $\gamma$-H\"older functions:\begin{equation} \label{eq: defn of script A}\mathcal{A}_\gamma:=\left\{f: \mathbb{R}^d \rightarrow \mathbb{R}: \|\hat{f}\|_{0,\gamma}\le 1\right\}\end{equation} and define the weighted Wasserstein metric of type $\gamma$ by the duality\begin{equation}\label{eq: definition of W} \mathfrak{w}_\gamma(\mu, \nu):= \sup_{f\in\mathcal{A}_\gamma}|\langle f, \mu-\nu\rangle|. \end{equation} \subsection{Relationships between Metrics} Let us now review some relationships between these metrics. \paragraph{1. $w_1$ and $\mathfrak{w}_\gamma$: Kantorovich-Wasserstein Duality.}In the case $p=1$, the metric $w_1$ is known as the Monge-Kantorovich-Wasserstein distance, and the well-known Kantorovich-Wasserstein duality can be written \begin{equation}\begin{split}\label{eq: dual W1}
 w_1\left(\mu, \nu\right)=\mathfrak{w}_1\left(\frac{\mu}{1+|v|^2},\frac{\nu}{1+|v|^2}\right)=\sup\left\{\langle f, \mu-\nu\rangle: \|f\|_{0,1}\le 1\right\}.\end{split} \end{equation} Further, $\mathfrak{w}_\gamma, 0<\gamma\le 1$ and $w_1, w_2$ all generate the topology of weak convergence on $\mathcal{S}$. \bigskip \\ Let us mention some quantitative comparisons between and within these classes. For all $f$, we have the bound $\|f\|_{0,\gamma}\le 2^{1-\gamma}\|f\|_{0,1}$, which leads to the comparison $\mathfrak{w}_\gamma\le 2^{1-\gamma}\mathfrak{w}_1$, while approximating $f\in \mathcal{A}_\gamma$ by $f^\epsilon\in c_\epsilon \mathcal{A}_1$ leads to the bound $\mathfrak{w}_\gamma\le \mathfrak{w}_1^\gamma$. \bigskip \\ We now compare $w_1$ and $\mathfrak{w}_1$. On the one hand, if $\|f\|_{0,1}\le 1$, then it is straightforward to see that $\|\widehat{f}\|_{0,1}\le c$ for some absolute constant $c$, so the duality (\ref{eq: dual W1}) implies that $w_1\le c\mathfrak{w}_1$. In the other direction, if $\|f\|_{0,1}\le 1$ and $\pi \in \Pi(\mu, \nu)$, then explicit calculations show that \begin{equation}\begin{split} |\langle f(1+|v|^2), \mu-\nu\rangle|\le & \int_{\RRd\times \RRd}(|v-w|)(|v|+|w|)\pi(dv,dw)\\&+2\int_{\RRd\times\RRd}(1\land |v-w|)(1+|w|^2)\pi(dv,dw).\end{split}\end{equation} If $\mu, \nu \in \mathcal{S}^p$ for some $p>2$, then one can interpolate and optimise over $f, \pi$ to obtain, for some $C=C(p), \alpha=\alpha(p)>0$, \begin{equation} \mathfrak{w}_1(\mu, \nu)\le C\Lambda_p(\mu,\nu)w_1(\mu, \nu)^\alpha. \end{equation} \paragraph{2. $W_p$ and $w_{p+2}$.} It is elementary to show that, for some constant $C=C(p)$ and all $\mu, \nu\in \mathcal{S}^{p+2}$, we have \begin{equation}\label{eq: compare WP WP}
     C^{-1}w_{p+2}(\mu, \nu)^{(p+2)/2}\le W_p(\mu, \nu)\le Cw_p(\mu, \nu)\Lambda_{p+2}(\mu, \nu).
 \end{equation} Further, for any sequence $\mu^n, \mu \in \mathcal{S}^{p+2}$, the convergence $w_{p+2}(\mu^n, \mu)\rightarrow 0$ implies convergence of the moments $\Lambda_{p+2}(\mu^n)\rightarrow \Lambda_{p+2}(\mu)$, and so $W_p, w_{p+2}$ have the same convergent sequences and generate the same topology on $\mathcal{S}^{p+2}$. \medskip \\ Let us also record the elementary interpolation estimate, for any $p'>p+2$ and any $\mu, \nu \in \mathcal{S}^{p'}$, \begin{equation} \label{eq: elementary interpolation} w_1(\mu, \nu)\le W_p(\mu, \nu)\le w_1(\mu,\nu)^\alpha \Lambda_{p'}(\mu,\nu) \end{equation} for some $\alpha=\alpha(p,p')>0$. 
\paragraph{3. $W_p$ and $\mathfrak{w}_\gamma$.} 
Combining the estimates above, if $p\ge 0, p'>p+2$ and $0<\gamma\le 1$, we have the equivalence \begin{equation}\label{eq: mathfrak wgamma and Wp}\mathfrak{w}_\gamma(\mu,\nu)\le C\Lambda_{p'}(\mu, \nu)W_p(\mu,\nu)^\gamma;\hs W_p(\mu,\nu)\le C\Lambda_{p'}(\mu,\nu)\mathfrak{w}_\gamma(\mu,\nu)^\alpha \end{equation} for some $C=C(p,p'), \alpha=\alpha(p,p')$. 
\subsection{Discrete Approximation Scheme} We record, for future use, the following result on approximating measures $\mu$ by discrete measures $\mu^N$ in the distance $W_p$. \begin{proposition}\label{prop: discrete scheme}[Discrete Approximation Scheme in $W_p$] Fix $p\ge 0$, and let $\mu\in \mathcal{S}^{q}$ for $q>p+2$. Then there exists a sequence $\mu^N\in \mathcal{S}_N$ of discrete approximations to $\mu$ such that \begin{equation} \label{eq: conclusion of DAS} W_p(\mu^N,\mu)\rightarrow 0; \hs \Lambda_{p'}(\mu^N)\rightarrow \Lambda_{p'}(\mu)\text{ for all }p'\le q. \end{equation} \end{proposition} \begin{proof}[Sketch Proof] The following is based on an argument by Norris \cite[Proposition 9.3]{norris2016consistency}, and we will sketch only the main points. \bigskip \\ Let $(\Omega, \mathcal{F},\mathbb{P})$ be a probability space on which are defined   an infinite sequence of independent draws $u_1,u_2,....$ from $\mu$, and define for each $N$ \begin{equation} v^N_i=\frac{u_i-\overline{u}_N}{\sqrt{s_N}}; \hs 1\le i\le N\end{equation} where \begin{equation} \overline{u}_N=\frac{1}{N}\sum_{i=1}^N u_i; \hs s_N:=\sum_{i=1}^N|u_i-\overline{u}_N|^2. \end{equation} This construction guarantees that the empirical measures \begin{equation} \mu^N:=\frac{1}{N}\sum_{i=1}^N \delta_{v^N_i}\end{equation} lie in $\mathcal{S}_N$. One can verify that \begin{equation} w_1(\mu^N,\mu)\rightarrow 0;\hs \Lambda_{q}(\mu^N)\rightarrow \Lambda_{q}(\mu) \end{equation} almost surely, which implies the same for $W_p$ thanks to (\ref{eq: mathfrak wgamma and Wp}). In particular, the set of $\omega\in \Omega$ where $\mu^N(\omega), N\ge 1$ satisfy the conclusion (\ref{eq: conclusion of DAS}), restricted to the special case $p'=q$, has probability $1$, which implies the existence of such sequences. We conclude by noting that the special case $p'=q$ of (\ref{eq: conclusion of DAS}) implies all cases $p'\le q$ by a straightforward interpolation. \end{proof}

\section{A Jump Stochastic Differential Equation Associated to the Kac Process}\label{sec: JSDE} We begin our analysis of the Kac process by formulating a stochastic differential equation in $(\RRd)^N$ driven by Poissonian jump measures, which will correspond to the labelled Kac process discussed above.  Introduce first measurable maps $\iota=(\iota_1,...\iota_{d-1}):\RRd\rightarrow (\RRd)^{d-1}$ such that, for all $v\neq 0$, the set \begin{equation}\left\{\frac{v}{|v|},\frac{\iota_1(v)}{|v|},...,\frac{\iota_{d-1}(v)}{|v|}\right\}\end{equation}is an orthonormal basis of $\RRd$, and $\iota(-v)=-\iota(v)$. With this choice of $\iota$, define $\Gamma: \RRd\times \mathbb{S}^{d-2}\rightarrow \mathbb{R}^d$ by \begin{equation}\Gamma(v,\varphi)=\sum_{j=1}^{d-1}\varphi_j \iota_j(v). \end{equation} Let us also define \begin{equation} H(\theta)=\int_\theta^{\pi/2} b(\cos x)dx, \hs \theta\in \left(0,\frac{\pi}{2}\right). \end{equation} Thanks to (\ref{eq: form of B}), $H$ is now a bijection from $(0,\pi/2)$ to the ray $(0, \infty)$; let us write $G$ for its inverse. We finally define, for distinct $v,v_\star \in \RRd$ and $\varphi \in \mathbb{S}^{d-2}, z>0$\begin{equation} \theta(v, v_\star,z)=G\left(\frac{z}{|v-v_\star|^\gamma}\right); \end{equation} \begin{equation} a(v,v_\star,z,\varphi)=-\frac{1-\cos(\theta(v,v_\star,z))}{2}(v-v_\star)+\frac{\sin(\theta(v,v_\star,z))}{2}\Gamma(v-v_\star,\varphi).\end{equation}  In the case $v=v_\star$, we set $a(v,v_\star,z,\varphi)=0$; we note that, by construction, $a$ is antisymmetric in $v,v_\star$. Some estimates for the function $G$ are established in Section \ref{subsec: estimates for G}. \bigskip \\  With this parametrisation, we define a labelled Kac process to be the solution to an SDE with Poisson noise. For unordered pairs $\{i j\}=\{j i\}$ of distinct indeces $i,j=1,...,N$, let $\mathcal{N}^{\{ij\}}$ be independent Poisson random measures on $(0,\infty)\times \SSd\times (0,\infty)$, with intensity $2N^{-1}dsd\varphi dz$. A labelled Kac process is then exactly a solution $\mathcal{V}^N_t=(V^1_t,...,V^N_t)$ to the system of stochastic differential equations \begin{equation}\label{eq: SDE form}\tag{LK} V^i_t=V^i_0+\sum_{j\neq i}\int_{(0,t]\times \SSd\times(0,\infty)} a(V^i_{s-},V^j_{s-},z,\varphi)\h\Nn^{\{ij\}}(ds,d\varphi,dz)\end{equation} where the index $i$ runs over $1,...,N$. The factor of $2$ in the rate corresponds to working with unlabelled, rather than labelled, pairs of particles. Moreover, thanks to the antisymmetry of $a$ in the first two arguments, and recalling that $\Nn^{\{ij\}}=\Nn^{\{ji\}}$, we see that a jump in the $i^\text{th}$ particle $V^i_t\neq V^i_{t-}$ matches a jump in some $j^\text{th}$ particle, $j\neq i$. \bigskip \\ Clasically \cite{kurtz2011equivalence}, weak solutions to the stochastic differential equation (\ref{eq: SDE form}) are Markov processes $\mathcal{V}^N_t$ with the generator \begin{equation}
    \label{eq: labelled generator1}\begin{split} (\mathcal{G}^\mathrm{L}\widehat{F})(\mathcal{V}^N)&\\&\hspace{-1cm}=\frac{2}{N}\sum_{\{ij\}}\int_0^\infty dz\int_{\SSd}d\varphi \h \left(\widehat{F}\left(\mathcal{V}^N+a(V^i,V^j,z,\varphi)(\mathbf{e}_i-\mathbf{e}_j)\right)-\widehat{F}(\mathcal{V}^{N})\right) \end{split}
\end{equation} for Lipschitz $\widehat{F}:\mathbb{S}_N\rightarrow \mathbb{R}$, where the sum is over unordered pairs $\{ij\}=\{ji\}$, and we use the notation, for $1\le i\le N$ and $h\in \mathbb{R}^d$, $h\mathbf{e}_i$ is the vector $(0,...h,...0)$ in $(\mathbb{R}^d)^N$ with $h$ in the $i^\text{th}$ place. Thanks to the construction of $G$ and $a$, it is straightforward to check that the integral can be rewritten \begin{equation} \label{eq: labelled generator2} (\mathcal{G}^\mathrm{L}\widehat{F})(\mathcal{V}^N)=\frac{1}{N}\sum_{i=1}^N\sum_{j=1}^N\int_{\mathbb{S}^{d-1}}\left(\widehat{F}(\mathcal{V}^N_{i,j,\sigma})-\widehat{F}(\mathcal{V}^{N})\right)d\sigma \end{equation} where $\mathcal{V}^N_{i,j,\sigma}$ where $\mathcal{V}^N_{i,j,\sigma}$ denotes the vector in $(\mathbb{R}^d)^N$ where the $i^\text{th}$ and $j^\text{th}$ coordinates have been updated according to (\ref{eq: postcollisional}); we note the strong similarity with (\ref{eq: generator}). In this way, we can think of (\ref{eq: SDE form}) as corresponding to the Kac process where each particle is assigned a label, and we call (weak) solutions to (\ref{eq: SDE form}) a \emph{labelled Kac process}. We formalise the connection, and justify moving between the labelled and unlabelled dynamics, with the following proposition. \begin{proposition}\label{prop: labelled and unlabelled dynamics} \begin{enumerate}[label=\roman{*}).]
    \item Suppose $\mathcal{V}^N_t$ is a solution to to the stochastic differential equation (\ref{eq: SDE form}). Then the empirical measures $\mu^N_t=\theta_N(\mathcal{V}^N_t)$ are unlabelled Kac process. \item Every Kac process arises in this way: if $(\widetilde{\mu}^N_t, t\ge 0)$ is a Kac process starting at $\mu^N_0$, pick $\mathcal{V}^N_0\in \theta_N^{-1}(\mu^N_0)$ uniformly at random. Then there exists a weak solution to the stochastic differential equation (\ref{eq: SDE form}), starting at $\mathcal{V}^N_0$, such that $(\mu^N_t, t\ge 0)=(\theta_N(\mathcal{V}^N_t), t\ge 0)$ has the same law as $(\widetilde{\mu}^N_t, t\ge 0)$.
\end{enumerate}\end{proposition}  For ease of readability, the proof is deferred to Appendix \ref{sec: pf_wellposedness}. The first item is elementary, and uses the $\text{Sym}(N)$-symmetry of the labelled dynamics; the second item amounts to the careful application of a theorem due to Kurtz \cite{kurtz1998martingale,kurtz2011equivalence}. Since this proof does not rely on the intermediate results of this paper, we can use this equivalence without further comment without concern for circular arguments. We will also prove the following result concerning the well-posedness of the two formulations. \begin{proposition}\label{prop: wellposedness/existence}  For all $\mathcal{V}^N_0\in \mathbb{S}_N$, there exists a labelled Kac process $\mathcal{V}^N_t, t\ge 0$, that is, a weak solution to (\ref{eq: SDE form}), starting at $\mathcal{V}^N_0$. Therefore, if  $\mu^N_0\in \mathcal{S}_N$, then there exists an $N$-particle (unlabelled) Kac process starting at $\mu^N_0.$  \end{proposition} \begin{proposition}\label{prop: wellposedness/uniqueness}  For all $\mathcal{V}^N_0\in \mathbb{S}_N$, the labelled Kac process $\mathcal{V}^N_t, t\ge 0$ starting at $\mathcal{V}^N_0$ has uniqueness in law. At the level of unlabelled dynamics if  $\mu^N_0\in \mathcal{S}_N$, then there exists a unique, in law, $N$-particle (unlabelled) Kac process starting at $\mu^N_0.$  \end{proposition} Both proofs are again deferred until Appendix \ref{sec: pf_wellposedness}. For the existence result, we  use the combination of tightness and consistency, which is standard for martingale problems; in particular, no other results of this paper are necessary for this argument. For uniqueness, we use the construction of a coupling and estimates in Lemma \ref{lemma: Tanaka Coupling of KP}, which itself depends only on the elementary calculations in  Section \ref{sec: main calculations}. We emphasise that the analyses of Sections \ref{sec: moments}, \ref{sec: TK} may be understood as \emph{\`a priori} estimates and do not require this uniqueness. In both cases, the item on unlabelled dynamics follows from the statement on labelled dynamics using Proposition \ref{prop: labelled and unlabelled dynamics}.  \bigskip \\   We also construct a cutoff version $\mathcal{V}^{N,K}_t=(V^{1,K}_t,...V^{N,K}_t)$ of these processes as follows.  In analogy to the definition above, set \begin{equation} a_K(v,v_\star,z,\varphi)=a(v,v_\star,z,\varphi)1\left(z\le K|v-v_\star|^\gamma\right).\end{equation} The $K$-cutoff version of (\ref{eq: SDE form}),  corresponding to the cutoff kernel $B_K$ defined in (\ref{eq: cutoff kernel}) is now \begin{equation} \tag{cLK}\label{eq: cutoff SDE} V^{i,K}_t=V^{i,K}_0+\sum_{j\neq i}\int_{(0,t]\times \SSd\times(0,\infty)} a_K(V^{i,K}_{s-},V^{j,K}_{s-},z,\varphi)\h\Nn^{\{ij\}}(ds,d\varphi,dz).\end{equation} In the notation above, $\theta_0(K)=H(K)\rightarrow 0$ as $K\rightarrow \infty.$ Let us remark that the statements equivalent to Propositions \ref{prop: labelled and unlabelled dynamics}, \ref{prop: wellposedness/existence}, \ref{prop: wellposedness/uniqueness} for the cutoff differential equation (\ref{eq: cutoff SDE}) and the corresponding cutoff Kac process $\mu^{N,K}_t$ are elementary, as in both cases the overall jump rates are uniformly bounded.

\section{Moment Estimates for the Kac Process and Boltzmann Equation}
\label{sec: moments} We now present some results concerning the moment evolution for both the Kac process and the Boltzmann equation. Subsections \ref{sec:cutoff moment estimates},\ref{subsec: noncutoff moments} collect some moment estimates for the Kac process and Boltzmann equation in the cutoff and noncutoff cases respectively, uniformly in both the number of particles $N$ and cutoff parameter $K$; these results are mostly classical, and we present results from the literature. The arguments are well-known for the Boltzmann equation, and have been proven for measure-valued solutions of the Boltzmann equation by Lu and Mouhot \cite{lu2012measure}, and for the Kac process by Mischler and Mouhot \cite{mischler2013kac} and Norris \cite{norris2016consistency}. In Subsection \ref{subsec: concentration of moments}, we prove a novel `concentration of moments' result.

\subsection{Moment Inequalities in the Cutoff Case}\label{sec:cutoff moment estimates}

We next turn to some moment inequalities for the cutoff process and associated limit equation.  The only novelty here is some care to ensure that the estimates are uniform in the cutoff parameter $K$ as soon as $K$ is bounded away from $0$.
 
 \begin{proposition} [Moment Inequalities for the Cutoff Kac Process and Boltzmann Equation] \label{thrm:momentinequalities} We have the following moment bounds for polynomial velocity moments: \\		
\begin{enumerate}[label={\roman*).},ref={\roman*.}] \item \label{lemma:momentboundpt1}  Let $(\mu^{N,K}_t)_{t\geq 0}$ be a $K$-cutoff Kac process on $N\ge 2$ particles, $K\ge 1$, started from $\mu^{N,K}_0$, and fix $q\ge p\ge 4$. Then there exists a constant $C(p,q)<\infty$, which does not depend on $K$, such that, for all $t\ge 0$, \begin{equation} \label{eq: pointwise moment bound}  \mathbb{E}\left[ \Lambda_q(\mu^{N,K}_t) \right] \leq C(1+t^{p-q})\Lambda_p(\mu^{N,K}_0).\end{equation} \item In the notation of the previous point, there exists a constant $C=C(p)$, also independent of $K$, such that for all $t_\mathrm{fin}\ge 0$,\begin{equation} \label{eq: local uniform moment bound} \mathbb{E}\left(\sup_{0\leq t \leq t_\text{fin}} \Lambda_p(\mu^{N,K}_t)\right) \leq (1+C(p)t_\text{fin})\Lambda_p(\mu^{N,K}_0).\end{equation} 
\item Let $k\ge 2$. In the notation of point i), we have the almost sure relation \begin{equation}\label{eq: geometric growth}
    \PP\left(\Lambda_k(\mu^{N,K}_t)\le 2^{\frac{k}{2}+1}\Lambda_k(\mu^{N,K}_{t-})\h\text{ for all }t\ge 0\right)=1.
\end{equation}
\item \label{lemma:momentboundpt2} Let $p, q$ be as above, and let, and $\mu_0\in \cup_{k>2}\mathcal{S}^k$. Then there exists a constant $C=C(p,q)$, which does not depend on $K\ge 1$, such that the solution $\phi^K_t(\mu_0)$ to the cutoff Boltzmann Equation satisfies \begin{equation}\label{eq: BE moment bound}  \Lambda_q(\phi^K_t(\mu_0)) \le C(1+t^{p-q})\Lambda_p(\mu_0).\end{equation}
\end{enumerate} \end{proposition} 
\begin{proof} Let us sketch the arguments leading to point items i)-ii)., which are similar to those of \cite[Proposition 3.1]{norris2016consistency} but require some modification. It is convenient to work with a labelled, cutoff Kac process $\mathcal{V}^{N,K}_t$ driven by Poisson random measures $\Nn^{\{ij\}}$, and whose empirical measures are $\mu^{N,K}_t$. We start from the bound, for all $p\ge 4$ and any $v,v_\star,z,\varphi$,\begin{equation}\begin{split} |v+a|^p &\le \left(\frac{1+\cos G(z/x^\gamma)}{2}\right)^{p/2} |v|^p+\left(\frac{\sin G(z/x^\gamma)}{2}\right)^{p/2}|v_\star|^p \\[1ex] & \hs \hs +C_p(|v|^{p-1}|v_\star|+|v_\star|^{p-1}|v|)\sin G(z/x^\gamma) \end{split} \end{equation} where we write $a=a(v,v_\star,z,\varphi)$ and $x:=|v-v_\star|$.  From this, and a similar inequality for $|v_\star-a|^p$, we obtain \begin{equation} \begin{split}\label{eq: change of moment at collision} |v+a|^p+|v_\star-a|^p-|v|^p-|v_\star|^p &\le -\beta(p, G(z/x^\gamma))\left(|v|^p+|v_\star|^p\right) \\[1ex] &   +C_p(|v|^{p-1}|v_\star|+|v_\star|^{p-1}|v|)\sin G(z/x^\gamma)\end{split} \end{equation} where \begin{equation} \beta(p, \theta)=\left(1-\left(\frac{1+\cos \theta}{2}\right)^{p/2}-\left(\frac{\sin \theta}{2}\right)^{p/2}\right).\end{equation} Since we consider only $p\ge 4$, we see that $\beta(p,\theta)>0$ for all $\theta\in (0,\pi/2]$. We conclude that, for any $v,v_\star$, \begin{equation}\label{eq: prototype pozvner}\begin{split}& \int_{(0,\infty)\times\SSd} (|v+a|^p+|v_\star-a|^p-|v|^p-|v_\star|^p)1(z\le K|v-v_\star|^\gamma)dzd\varphi  \\ & \hs \hs \hs    \le - |v-v_\star|^{\gamma}\left(\int_0^K \beta(p,G(z)) dz\right)(|v|^p+|v_\star|^p) \\ & \hs \hs \hs + C|v-v_\star|^{\gamma}\left(\int_0^K \sin G(z) dz\right)(|v|^{p-1}|v_\star|+|v_\star|^{p-1}|v|).\end{split}\end{equation} The coefficient multiplying the first term is bounded bounded away from $0$ for $K\ge 1$, and, in particular, we can bound the right-hand side above by replacing this coefficient by a $K$-independent constant. For the second term, $\int_0^K \sin G(z) dz$ is bounded, uniformly in $K$, since $\int_0^\infty \sin G(z) dz<\infty$.  With this modification, the same arguments as in \cite[Proposition 3.1]{norris2016consistency} lead to the first point of i). with $\Lambda_p$ replaced by $\langle |v|^p, \mu^{N,K}_t\rangle$. The conclusion follows on noting that, for some $C=C(p)$ and all $\mu \in \mathcal{S}$, \begin{equation} C^{-1}\Lambda_p(\mu)\le \langle |v|^p, \mu\rangle \le C\Lambda_p(\mu). \end{equation}   For the second point, we return to (\ref{eq: change of moment at collision}) to bound the jumps of $\Lambda_p(\mu^{N,K}_t)$ by \begin{equation} \frac{C_p}{N}(|V^{i,K}_{t-}|^{p-1}|V^{j,K}_{t-}|+|V^{j,K}_{t-}|^{p-1}|V^{i,K}_{t-}|)\h\sin G\left(z/|V^{i,K}_{t-}-V^{j,K}_{t-}|^\gamma\right)\end{equation}  at points of $\Nn^{\{ij\}}$, since the first term is always negative. We now consider the process $A_t$ whose jumps are exactly the right-hand side, so that $A_t$ is increasing and \begin{equation} \sup_{s\le t}\langle |v|^p, \mu^{N,K}_s\rangle \le \langle |v|^p, \mu^{N,K}_0\rangle + A_t. \end{equation} We now estimate \begin{equation}\begin{split} \EE[A_t] & \le C\EE\left[\int_0^t \int_{\RRd\times\RRd} (|v|^{p-1}|v_\star|+|v_\star|^{p-1}|v|)|v-v_\star|^\gamma \mu^{N,K}_s(dv)\mu^{N,K}_s(dv_\star)\right]\end{split}\end{equation} where the constant is independent of $K$, due to the inclusion of the factor $\sin G$ in the definition of $A_t$. Simplifying, we see that \begin{equation} \EE A_t\le C_p\EE \int_0^t \langle |v|^{p+\gamma-1},\mu^{N,K}_s\rangle ds \le C_p\EE \int_0^t \Lambda_p(\mu^{N,K}_s)ds \end{equation} and the conclusion now follows, using the previous point to bound $\EE \Lambda_p(\mu^{N,K}_t)$. Item iii) can be straightforwardly checked following the same calculations as \cite[Lemma 2.1]{heydecker2019pathwise}.
\bigskip \\  Regarding the Boltzmann Equation, items iv). and v). are standard in the literature for the cutoff Boltzmann equation, going back as far as Povzner and Bobylev, and these results can be checked by following the proofs in the quoted papers \cite{lu2012measure,mischler2013kac,norris2016consistency}. The key point here is that the constants arising in the proofs are independent of the cutoff $K$, which follows a similar pattern to the analysis of (\ref{eq: prototype pozvner}) above.\end{proof} 

\subsection{Moment Inequalities in the Noncutoff Case}\label{subsec: noncutoff moments}
We will also use similar results for the noncutoff processes. In this case, where the available statement on well-posedness of the Boltzmann equation, or uniqueness in law for the Kac process, are weaker than in the cutoff case, we are careful that all the results cited are \emph{\`a priori} bounds, which are valid for any solution to the Boltzmann equation and any unlabelled Kac process respectively. 

\begin{proposition} [Moment Inequalities for the Noncutoff Kac Process and Boltzmann Equation] \label{thrm:noncutoffmomentinequalities} We have the following moment bounds for polynomial velocity moments: \\		
\begin{enumerate}[label={\roman*).},ref=\roman*.] \item Let $\mu^{N}_t$ be a noncutoff Kac process on $N\geq 2$ particles, and let $q\ge p\ge 4$. Then there exists a constant $C(p,q)<\infty$, such that, for all $t\ge 0$, \begin{equation} \label{eq: noncutoff pointwise moment bound}  \mathbb{E}\left[ \Lambda_q(\mu^{N}_t) \right] \leq C(1+t^{p-q})\Lambda_p(\mu^{N}_0).\end{equation} \item In the notation of the previous point, there exists a constant $C=C(p)$ such that for all $t_\mathrm{fin}\ge 0$,\begin{equation} \label{eq: noncutoff local uniform moment bound} \mathbb{E}\left(\sup_{0\leq t \leq t_\text{fin}} \Lambda_p(\mu^{N}_t)\right) \leq (1+C(p)t_\text{fin})\Lambda_p(\mu^{N}_0).\end{equation} 

\item Let $\mu^N_t$ be as in item i)., and let $k\ge 2$. Then the bound (\ref{eq: geometric growth}) holds with $\mu^N_t$ in place of $\mu^{N,K}_t$.
\item Let $(\mu_t)_{t\ge 0}\subset \mathcal{S}$ be a solution to the noncutoff Boltzmann equation (\ref{BE}), and $p>2$. Then there exists a constant $C=C(p)<\infty$ such that \begin{equation} \sup_{t\ge 0}\Lambda_p(\mu_t)\le C\Lambda_p(\mu_0).\end{equation}

\item In the notation of the previous point, for all $t>0$, there exists $\epsilon=\epsilon(t)>0$ such that \begin{equation}\label{eq: exponential moment appears}
    \left\langle e^{\epsilon(t)|v|^\gamma}, \mu_t\right\rangle<\infty.
\end{equation} 

\end{enumerate} \end{proposition} 

\begin{proof}
Concerning the Kac process, Items i-ii) follow as in the previous proposition, and  item iii). follows the same argument as in Proposition \ref{thrm:momentinequalities} and \cite{heydecker2019pathwise}, although one must now be careful that $\mu^N_t$ is a jump process of infinite activity. \bigskip \\  Regarding the Boltzmann equation, item iv) can be proved in the same way as the results for the cutoff case by carfeully justifying a differential inequality for the moment $m_p(t)=\Lambda_p(\mu_t)$; see, for example, \cite[Equation 4.8]{fournier2008uniqueness} with $s=1$. For item v), we refer to \cite{fournier2008uniqueness} and \cite{lu2012measure}. \emph{\`a priori} estimates \cite[Lemma 4.1]{fournier2008uniqueness} show that, for all $t>0, p>2$, $\sup_{s\ge t}\Lambda_p(\mu_s)<\infty$. We now apply \cite[Theorem 1.3d]{lu2012measure}, noting that the hypothesis {\bf (H3).} required is is satisfied for our choice of kernel, although we warn the reader that the usage of $b$ in the cited paper differs from ours by a factor of $\sin^{d-2}\theta$; the cited result is precisely that, with these \emph{\`a priori} bounds, the claimed result (\ref{eq: exponential moment appears}) holds.
\end{proof}

\subsection{Concentration of Moments}
\label{subsec: concentration of moments}
We also prove the following result concerning the concentration of moments for the Kac process. The results above show uniform bounds on the expectation of moments, but to attain a bound of the form $\PP(\Lambda_p(\mu^N_t)\le b_N)\rightarrow 1$, we would need to take some sequence $b_N\rightarrow \infty$. The following result allows us to obtain such bounds with a single $b_N=b$ independent of $N$.

\begin{lemma}[Concentration of Moments]\label{lemma: concentration of moments} Fix $p\ge 2$, and let $q\ge 2p+\gamma$. Then there exist constants $C_1(p), C_2(p)$ such that, whenever $\mu^N_t$ is a (cutoff or non-cutoff) Kac process on $N$ particles satisfying an initial moment bound $\Lambda_q(\mu^N_0) \le a_3$, $a\ge 1$, then for all $t_\mathrm{fin}\ge 0$ and $\epsilon>0$, we have the bound \begin{equation} \PP\left(\sup_{t\le t_\mathrm{fin}} \langle|v|^p, \mu^N_t\rangle \ge \max(\langle|v|^p, \mu^N_t\rangle, C_1)+\epsilon\right) \le C_2t_\mathrm{fin}\h a_3\h N^{-1} \h \epsilon^{-2}. \end{equation} Define, for $b\ge 1$, \begin{equation}\label{eq: defn of TB} T^N_b=\inf\left\{t\ge 0: \Lambda_p(\mu^N_t)>\frac{b}{2^{\frac{p}{2}+1}}\right\}. \end{equation} As a consequnce of the estimate above, there exists $C=C(p)$ such that, if the initial data has the moment estimates $\Lambda_p(\mu^N_0)\le a_1$, $\Lambda_q(\mu^N_0)\le a_3$, then  \begin{equation}
    \PP(T^N_{Ca_1}\le t_\mathrm{fin}) \le Ct_\mathrm{fin}a_3 N^{-1}.
\end{equation} \end{lemma}  The first statement here is somewhat sharper, and may be of independent interest; however, for applications later in the paper, it is useful to apply the second form, which absorbs some constants, without further comment.
\begin{proof} Thanks to Proposition \ref{prop: labelled and unlabelled dynamics}, it is sufficient to consider the case where $\mathcal{V}^N_t=(V^1_t,...,V^N_t)$ is a labelled Kac process, and $\mu^N_t$ are the associated empirical measures. Let us define \begin{equation}\begin{split} \label{eq: noisy moment ODE}M^N_t&:=\langle |v|^p, \mu^N_t\rangle - \langle |v|^p, \mu^N_0\rangle - \int_0^t \langle |v|^p, Q(\mu^N_s)\rangle ds \\&=\frac{1}{N}\int_{(0,t]\times\SSd\times(0,\infty)}\sum_{1\le i<j\le N} H^p_{ij}(s,\varphi,z)\overline{\Nn}^{\{ij\}}(ds,d\varphi,dz) \end{split}\end{equation}where we write \begin{equation}\begin{split} H^p_{ij}(t,\varphi,z):=&|V^i_{t-}+a(V^i_{t-},V^j_{t-},z,\varphi)|^p-|V^i_{t-}|^p\\& \hspace{1cm}+|V^j_{t-}-a(V^i_{t-},V^j_{t-},z,\varphi)|^p-|V^j_s|^p\end{split} \end{equation} and \begin{equation}\overline{\Nn}^{\{ij\}}(ds,d\varphi,dz)={\Nn}^{\{ij\}}(ds,d\varphi,dz)-\frac{2}{N}dsd\varphi dz. \end{equation} From the results of \cite{darling2008differential}, $M^N$ is a total variation martingale, and thanks to Pozvner estimates in the spirit of (\ref{eq: prototype pozvner}), for some $\beta=\beta(p)>0$ and all $\mu \in \mathcal{S}^p$, \begin{equation}\label{eq: moment drift} \begin{split}\langle |v|^p, Q(\mu)\rangle &\le -\beta \langle |v|^{p+\gamma}, \mu\rangle+\beta^{-1}\langle |v|^p, \mu\rangle \\&\le -\beta \langle |v|^p, \mu\rangle^{1+\gamma/p}+\beta^{-1}\langle |v|^p, \mu\rangle.  \end{split}\end{equation} Set $C_1=\beta^{-2p/\gamma}$, so that the right-hand side of (\ref{eq: moment drift}) is nonpositive as soon as $\langle |v|^p, \mu \rangle \ge C_1$. Define $T$ to be the stopping time \begin{equation}T=\inf\left\{t\ge 0: \langle|v|^p, \mu^N_t\rangle >\max(C_1,\langle |v|^p, \mu^N_0\rangle)+\epsilon\right\} \end{equation} and on the event $T\le t_\mathrm{fin}$,  define \begin{equation}
    T'=\sup\left\{t<T: \langle |v|^p, \mu^N_t\rangle \le \max(C_1,\langle |v|^p, \mu^N_0\rangle)\right\}.
\end{equation} This set is always nonempty, as it includes $0$, and we have \begin{equation} \limsup_{t\uparrow T'} \langle |v|^p, \mu^N_t\rangle \le \max(C_1,\langle |v|^p, \mu^N_0\rangle); \end{equation} \begin{equation}
    \langle |v|^p, \mu^N_t\rangle > \max(C_1,\langle |v|^p, \mu^N_0\rangle) \text{ for all }t\in (T', T].
\end{equation} By the choice of $C_1$, it follows that \begin{equation} \int_{(T',T]} \langle |v|^p, Q(\mu^N_s)\rangle ds \le 0 \end{equation} and so, from (\ref{eq: noisy moment ODE}), we must have $M^N_T-M^N_{T'-}\ge \epsilon.$ Therefore, on the event $\{T\le t_\mathrm{fin}\}$, we have the lower bound $\sup_{t\le t_\mathrm{fin}}|M^N_t|\ge \frac{\epsilon}{2}.$ \bigskip \\ Let us now estimate $M^N.$ From the analysis in \cite{darling2008differential}, we have \begin{equation}\label{eq: second moment of M} \begin{split}\EE\left[|M^N_t|^2\right]=\EE\left[\int_{(0,t]\times\SSd\times(0,\infty)}\sum_{1\le i<j\le N}\frac{1}{N^2}\left(H^p_{ij}(s,\varphi,z)\right)^2\hspace{0.1cm} \frac{2}{N}dsd\varphi dz\right].\end{split} \end{equation} To bound the integrand, we observe that $|a|^2=|V^i_s-V^j_s|^2(1-\cos G(z/x^\gamma))$, where $x=|V^i_s-V^j_s|$. We therefore obtain \begin{equation}\begin{split} (|V^i_s+a|^p-|V^i_s|^p)^2 &\le C(p)(1+|V^i_s+a|^{p-1}+|V^i_s|^p)^2|a|^2 \\& \le C(p)(1+|V^i_s|^{2p-2}+|V^j_s|^{2p-2})|V^i_s-V^j_s|^2(1-\cos G(z/x^\gamma)) \\ & \le C(p)(1+|V^i_s|^{2p}+|V^j_s|^{2p}) G(z/x^\gamma)^2.\end{split}\end{equation} Thanks to the estimates in Section \ref{subsec: estimates for G}, it follows that $\int_0^\infty G^2dz<\infty$, and in particular, the integral on the right-hand side of (\ref{eq: second moment of M}) is finite. Using a similar computation for $(|V^j_s-a|^p-|V^j_s|^p)^2$, we obtain \begin{equation} \begin{split}\int_{\SSd\times(0,\infty)}\left(H^p_{ij}(s,\varphi,z)\right)^2\hspace{0.1cm}d\varphi dz &\le C(p)(1+|V^i_s|^{2p}+|V^j_s|^{2p})x^\gamma\\ & \le C(p)(1+|V^i_s|^{2p+\gamma}+|V^j_s|^{2p+\gamma}).\end{split}\end{equation} Returning to (\ref{eq: second moment of M}), we sum over pairs $i,j$ to obtain, for some $C_2=C_2(p),$ \begin{equation} \EE\left[|M^N_t|^2\right]\le \frac{C_2}{16N}\EE\left[\int_0^t\Lambda_{2p+\gamma}(\mu^N_s)ds\right]. \end{equation} By the choice of $\mu^N_0$ and moment propagation results above, the right-hand side is at most $C_2t_\mathrm{fin}a_3/16N.$ The first item now follows by using Doob's $L^2$ inequality to bound $\EE[\sup_{t\le t_\mathrm{fin}} |M^N_t|^2]$, and Chebychev's inequality to bound the probability \begin{equation} \PP\left(\sup_{t\le t_\mathrm{fin}}|M^N_t|\ge \frac{\epsilon}{2}\right)\le 16\epsilon^{-2}\EE\left[|M^N_{t_\mathrm{fin}}|^2\right]. \end{equation}  The second item is a largely trivial reformulation of the first, noting that $\Lambda_p(\mu^N_t)/\langle |v|^p, \mu^N_t\rangle$ is bounded, and since $a_1\ge 1$ and $a_1\ge \langle |v|^p, \mu^N_t\rangle$, we can choose $C=C(p)$  so that, on the event $\{\Lambda_p(\mu^N_t)>{Ca_1}/{2^{\frac{p}{2}+1}}\}$, we also have $\langle |v|^p, \mu^N_t\rangle >\max(\langle |v|^p, \mu^N_t\rangle, C_1)+1.$ \end{proof} 

\newpage \section{Analysis of the Cutoff Kac Process \& Boltzmann Equation} \label{sec: cutoff case}
In this section, we will collect some results concerning the \emph{cutoff} Kac process and Boltzmann equation. Much of what we need exists in the literature already, and our emphasis is on tracking the dependence of the results on the cutoff parameter $K$, in preparation for taking $K\rightarrow \infty$ later.  Our techniques build primarily on the ideas of Norris \cite{norris2016consistency}, which ensure that the number of moments required for our convergence result does not depend on the cutoff parameter $K$. We are aware of other techniques which can show convergence of the $N$-particle cutoff system in the limit $N\rightarrow\infty$ \cite{mischler2013kac,heydecker2019pathwise}, but which would lead to results where the number of moments required grows with $K$. Let us mention, to begin with, that existence and uniqueness is well-established in the literature for the cutoff case; see, for instance, \cite{lu2012measure} and the references therein. In our terminology, \cite[Theorem 1.4]{lu2012measure} shows existence and uniqueness for $\mathcal{S}^p$-locally bounded solutions, for any $p>2$, which justifies the use of the semigroup notation $\phi^K_t$. The following result builds on \cite[Theorem 1.1]{norris2016consistency}, and quantifies the rate of convergence. \begin{lemma}[Convergence of the Cutoff Kac Process]\label{lemma: convergence of cutoff KP}  Let $p\ge 0$ and  $q>\max(4+3\gamma, p+2)$. Then there exists $C=C(G,q,d), \alpha=\alpha(d,p,q)\ge 0$ such that, whenever $a\ge 1$, $\mu_0\in \mathcal{S}$ and $\mu^{N,K}_t$ is a $K$-cutoff Kac with $K\ge 1$ and initial moment estimates \begin{equation} \Lambda_q(\mu_0)\le a,\hs \PP\left(\Lambda_q(\mu^{N,K}_0)\le a\right)=1\end{equation} then we have the convergence estimate, for all $t_\mathrm{fin}\ge 0$, \begin{equation} \EE\left[\h \sup_{t\le t_\mathrm{fin}} \h W_p\left(\mu^{N,K}_t, \phi^K_t(\mu_0)\right)\right]\le \left(N^{-\alpha}+\EE\left[W_p\left(\mu^{N,K}_0, \mu_0\right)\right]^\alpha \right)e^{CaK(1+t_\mathrm{fin})}. \end{equation} \end{lemma} The following lemma, which may be of interest in its own right, arises as a step towards the above result. \begin{lemma}[Wasserstein Stability for Cutoff Boltzmann Equation]\label{lemma: W-W stability for cutoff}  For all $p\ge 0, p'>p+2$, there exists a constant $C=C(p,p')<\infty$ and $\alpha=\alpha(p,p')>0$ such that, for all $\mu, \nu\in \mathcal{S}^{2+\gamma}$ and all $K<\infty$, we have the stability, for $t\ge 0$, \begin{equation}
     W_p\left(\phi^K_t(\mu),\phi^K_t(\nu)\right)\le \exp\left(CK(1+t)\Lambda_{2+\gamma}(\mu, \nu)\right)W_p(\mu,\nu)^\alpha \Lambda^2_{p'}(\mu,\nu).
 \end{equation} \end{lemma} 
\subsection{Random Measures Associated to the Cutoff Process}
We begin by first introducing the \emph{jump measure} and compensator associated to the cutoff Kac process $(\mu^{N,K}_t)_{t\ge 0}$. 
\begin{defn}\label{def: JM+C}[Jump Measure and Compensator] Let $(\mu^{N,K}_t)_{t\ge 0}$ be a cutoff Kac process on $N$ particles.\begin{enumerate}[label=(\roman*.)]
    \item The \emph{jump measure} $m^{N,K}$ is the unnormalised empirical measure on $(0,\infty)\times \mathcal{S}_N$ on all pairs $(t, \mu^{N,K}_t)$ such that $\mu^{N,K}_t\neq \mu^{N,K}_{t-}$. \item Let $\mathcal{Q}_{N,K}$ be the kernel on $\mathcal{S}_N$ given by \begin{equation}
        \mathcal{Q}_{N,K}(\mu^N,A)=\int_{\RRd\times\RRd\times\mathbb{S}^{d-1}} 1(\mu^{N,v,v_\star,\sigma}\in A)B_K(v-v_\star,d\sigma)\mu^N(dv)\mu^N(dv_\star).
    \end{equation}
The \emph{compensator} $\overline{m}^{N,K}$ of the jump measure is the measure on $(0,\infty)\times \mathcal{S}_N$ given by \begin{equation} \overline{m}^{N,K}(dt,A)=\mathcal{Q}_{N,K}(A,d\mu^N)dt. \end{equation} \end{enumerate} \end{defn} Since we are working with the cutoff process, both of these measures are almost surely finite on compact subsets $(0,t]\times\mathcal{S}_N$, for any $t<\infty$.

\subsection{The Linearised Kac Process for the Cutoff Case} \label{subsection: LKP}
We next sketch the proof of Lemma \ref{lemma: W-W stability for cutoff}. In order to apply the ideas of \cite{norris2016consistency}, we check the following continuity property of the kernel. \begin{lemma}\label{lemma: holder in TV} There exists $C<\infty, \alpha>0$, depending only on $b$ such that, for all $v,v' \in \RRd$, we have the estimate \begin{equation} \sup_{v_\star\in \RRd} \left\|B_K(v-v_\star,\cdot)-B_K(v'-v_\star,\cdot)\right\|_{L^1(d\sigma)}\le CK^\alpha |v-v'|.\end{equation} \end{lemma} \begin{proof}  Using the convexity of $b$, it follows that $b$ is differentiable almost everywhere on $(0,1)$, and from (\ref{eq: form of B}), one can check the bound \begin{equation} b'(x)\le C(1-x)^{-2-\nu/2}\end{equation} for some constant $C$, which depends only on the singularity of $b$ itself. Let us recall that the cutoff for $B_K$ is at $\theta_0(K)=G(K)\sim K^{-1/\nu}$, and so there exists $\alpha>0, C$ such that \begin{equation} \sup_{x\le \cos \theta_0(K)} (|b(x)+|b'(x)|) \le CK^\alpha.\end{equation} Let us fix $u, u' \in \mathbb{S}^{d-1}$. By splitting the integral into three regions, we find \begin{equation}  \begin{split} \int_{\mathbb{S}^{d-1}} |B_K(u,\sigma)-B_K(u',\sigma)|d\sigma &\le |u'-u|\h\sup_{\theta \ge \theta_0(K)} \left(|b'(\theta)|+|b(\theta)|\right) \\ & \le C|u'-u|\h K^{\alpha}. \end{split} \end{equation}  This extends to general $v,w$ of norm at least $1$: \begin{equation} \label{eq: lipschitz in TV} \|B_K(v,\cdot)-B_K(w,\cdot)\|_{L^1(d\sigma)} \le C|v-w|K^\alpha \hs \forall v, w: |v|, |w|\ge 1. \end{equation} We now consider the total variation distance \begin{equation} \Delta B_K(v_\star)=\|B_K(e_1-v_\star,\cdot)-B_K(-e_1-v_\star,\cdot)\|_{L^1(d\sigma)}.\end{equation} If $|v_\star|\le 2$, then we use the bound $\|B_K(\pm e_1-v_\star,\cdot)\|_{L^1(d\sigma)}\le CK$ to see that $\Delta B_K \le CK$. On the other hand, in the region $|v_\star|\ge 2$, it follows from (\ref{eq: lipschitz in TV}) that $\Delta B_K\le CK^\alpha$ and, combining, we conclude that \begin{equation} \sup_{v_\star \in \RRd} \|B_K(e_1-v_\star,\cdot)-B_K(-e_1-v_\star,\cdot)\|_{L^1(d\sigma)} \le CK^\alpha. \end{equation} For general $v\neq v'$, there exists a rigid transformation of $\RRd$ taking $2e_1$ to $v-v'$. Using the scaling and rotational properties of the form (\ref{eq: form of B}) of $B_K$, we conclude that \begin{equation} \sup_{v_\star\in \RRd} \|B_K(v-v_\star,\cdot)-B_K(v'-v_\star,\cdot)\|_{L^1(d\sigma)} \le CK^\alpha |v-v'|^\gamma \end{equation} as desired. \end{proof} 
 Following the ideas of \cite[Section 4]{norris2016consistency}, we introduce the following \emph{linearised Kac process}. Throughout, we fix $K\ge 1$, and omit it from the notation. \begin{defn}\label{def: LKP} Let us write  $V=\mathbb{R}^d$ and $V^*$ for the signed space $V^*=V\times\{\pm 1\}=V^+\sqcup V^-$. We write $\pi: V^*\rightarrow V$ as the projection onto the first factor, and $\pi_\pm: V^\pm\rightarrow V$ for the obvious bijections. Let also $\rho=(\rho_t)_{t\ge 0}$ be family of measures on $V=\mathbb{R}^d$ such that \begin{equation} \langle 1, \rho_t \rangle =1;\hspace{1cm} \langle |v|^2, \rho_t\rangle =1 \hs \text{for all }t\ge 0;\end{equation}  \begin{equation} \label{eq: integrability for environment}  \int_0^t \Lambda_{2+\gamma}(\rho_s)ds <\infty \hspace{1cm} \text{for all }t<\infty.  \end{equation} The \emph{Linearised Kac Process} \emph{in environment $(\rho_t)_{t\ge 0}$} is the branching process on $V^*$ where each particle of type $(v,1)$, at rate $2B_K(v-v_\star,d\sigma)\rho_t(dv_\star)$, dies, and is replaced by three particles, of types \begin{equation} (v'(v,v_\star,\sigma),1);\hspace{0.5cm}(v_\star'(v,v_\star, \sigma),1);\hspace{0.5cm}(v_\star,-1) \end{equation} where $v', v_\star'$ are the post-collisional velocities. The dynamics are identical for particles of type $(v,-1)$, with the signs exchanged. \medskip \\ We write $\Xi^*_t$ for the associated process of unnormalised empirical measures on $V^*$, and define a signed measure $\Xi_t$ on $V$ by including the sign at each particle: \begin{equation} \Xi_t=\Xi^+_t-\Xi^-_t ; \hspace{1cm} \Xi^\pm_t=\Xi^\star_t\circ \pi_\pm^{-1}.\end{equation} We can also consider the same branching process, started from a time $s\ge 0$ instead. We write $E$ for the expectation over the branching process, which is not the full expectation in the case where $\rho$ is itself random. When we wish to emphasise the initial velocity $v$ and starting time $s$, we will write $E_{(s,v)}$ when the process is started from $\Xi^*_0=\delta_{(v,1)}$ at time $s$, and $E_v$ in the case $s=0$. \end{defn} Provided that the initial data $\Xi^\star_0$ satisfies $E\langle 1+|v|^2, |\Xi_0|\rangle <\infty$, one can show that the branching process $(\Xi_t)_{t\ge 0}$ is non-explosive, and therefore defined for all time $t\ge 0$. Moreover, the bound is propagated: \begin{equation} \label{eq: propagation of bound} E\langle 1+|v|^2, \Xi^*_t\rangle \le \exp\left(CK\int_0^t\Lambda_{2+\gamma}(\rho_s) ds\right)\h E\langle 1+|v|^2, \Xi^*_0\rangle.\end{equation} We can therefore define, for functions $f$ of quadratic growth, \begin{equation}\label{eq: fst} f_{st}(v)=E_{(s,v)} \langle f, \Xi_t\rangle.\end{equation} We will write $f_{st}[\rho]$ when we wish to emphasise the dependence on the environment $\rho$. The proof of Lemma \ref{lemma: W-W stability for cutoff} is based on the following representation formula, which can be proved with only slight modifications of \cite[Proposition 4.2]{norris2016consistency}. \begin{proposition}\label{lemma: representation formula}[Representation formula for Cutoff Cases] Let us fix $\mu, \nu\in \mathcal{S}^{2+\gamma}$, and consider the environment \begin{equation} \rho_t=\frac{\phi^K_t(\mu)+\phi^K_t(\nu)}{2}.\end{equation} Then, for all $t\ge 0$ and all $f$ of quadratic growth, \begin{equation} \langle f, \phi^K_t(\mu)-\phi^K_t(\nu)\rangle = \langle f_{0t}[\rho],\mu-\nu\rangle. \end{equation} Further, let $\mu^{N,K}_t$ be a cutoff Kac process on $N$ particles, and let $m^{N,K}, \overline{m}^{N,K}$ be its jump measure and compensator, as in Definition \ref{def: JM+C}. In this case, consider propagation $f_{st}=f_{st}[\rho^N]$ in the random environment \begin{equation} \rho^N_t=\frac{\mu^{N,K}_t+\phi^K(\mu)}{2}.\end{equation} Then, for all $t\ge 0$, and all functions $f$ of quadratic growth, we have \begin{equation} \langle f, \mu^{N,K}_t-\phi^K(\mu)\rangle = \langle f_{0t}[\rho^N],\mu^{N,K}_t - \mu\rangle + M^{N,K,f}_t \end{equation} where \begin{equation} M^{N,K,f}_t=\int_{(0,t]\times \mathcal{S}_N} \langle f_{st}[\rho^N], \mu^N-\mu^{N,K}_{s-}\rangle (m^{N,K}-\overline{m}^{N,K})(ds,d\mu^N).\end{equation} \end{proposition}   Therefore, recalling the definition of $\mathfrak{w}_\gamma$ from Section (\ref{sec: metrics}), the proof of Lemma \ref{lemma: W-W stability for cutoff} reduces to estimating the growth and regularity of $f_{0t}$ when we start with a function $f\in \mathcal{A}_\gamma$. The following result adapts \cite[Proposition 4.3]{norris2016consistency} to our case. \begin{lemma}\label{lemma: fst}[Growth and Regularity of $f_{0t}$.] Fix $f\in \mathcal{A}_\gamma$ and an environment $\rho_t, t\ge 0$. Then $f_{0t}\in z_t \mathcal{A}_\gamma$, where \begin{equation} z_t=\exp\left(CK\left(1+\int_0^t \Lambda_{2+\gamma}(\rho_s)ds \right)\right) \end{equation}  for some constant $C$ independent of $K$. \end{lemma} \begin{proof} This lemma may be proven by adapting the arguments leading to \cite[Propositions 4.3, 4.5]{norris2016consistency}, which we outline here. Firstly, the same estimate (\ref{eq: propagation of bound}) already proves the claimed growth condition. \bigskip \\ In order to estimate the difference $f_{0t}(v_0)-f_{0t}(w_0)$, we introduce a coupling of the processes $\Xi^\star_t$ started at the initial data $(v_0, 1), (w_0,1)$. We consider a branching process on $W=(\RRd \times\RRd)\sqcup \RRd \sqcup \RRd =W_0\sqcup W_1\sqcup W_2$, where particles can either be coupled pairs $(v,w)\in (\RRd\times\RRd)$, or uncoupled particles in one of two disjoint copies of $\RRd$, and where each particle is assigned a sign $\pm 1$. The branching rules for uncoupled particles are the same as in Definition \ref{def: LKP}, while coupled particles of type $(v,w)$ scatter to remain coupled as far as possible, but undergo decoupling transitions at rate \begin{equation} \|B_K(v-v_\star,\cdot)-B_K(w-v_\star,\cdot)\|_{L^1(d\sigma)}\h \rho_t(dv_\star).\end{equation} Let $\Gamma^{0\star}, \Gamma^{1\star}, \Gamma^{2\star}$ be the empirical measures on $W_i \times\{\pm 1\}$, $i=0,1,2$, and consider the projection maps $p_i: W_0\times\{\pm 1\} \rightarrow \RRd\times{\pm 1}$ by projecting onto the $i^\text{th}$ marginal, $i=1,2$. The empirical measures \begin{equation} \Xi^{i\star}_t=\Gamma^{0\star}_t\circ p_i^{-1}+\Gamma^{i\star}_t, \hs i=1,2 \end{equation} are now a coupling of Linearised Kac processes. Using the bound (\ref{eq: propagation of bound}) on each marginal, we have estimates starting from a coupled pair \begin{equation}\label{eq: growth of coupled} E_{(0,(v_0,w_0) \in V_0)}\langle 1+|v|^2+|w|^2, \Gamma^{0\star}_t\rangle \le z_t\h(1+|v_0|^2+|w_0|^2)\end{equation} or from decoupled particles, for all $0\le s\le t$, \begin{equation} E_{(s,v_0 \in V_1)}\langle 1+|v|^2, \Gamma^{1\star}_t\rangle \le z_t\h(1+|v_0|^2)\end{equation} and similarly for $V_2$. Let us now run this process starting from a particle of type $(v_0, w_0)\in V_0$. Using the triangle inequality inductively, $\Gamma^{0\star}$ is supported only on coupled pairs $(u,u')\in V_0$ with $|u-u'|\le |v_0-w_0|$, and thanks to Lemma \ref{lemma: holder in TV}, the rate of decoupling of such a pair is at most $CK^\alpha |v_0-w_0|^\gamma$. With this modification, the proof of \cite[Lemma 4.5]{norris2016consistency} now gives the estimate \begin{equation} \label{eq: growth of decoupled} E_{(0,(v_0,w_0)\in V_0)} \langle 1+|v|^2, \Gamma^{1,\star}_t+\Gamma^{2,\star}_t\rangle \le  CK^\alpha |v_0-w_0|^\gamma (1+|v_0|^2+|w_0|^2)\h z_t.\end{equation} Let us fix $f\in \mathcal{A}_\gamma$. Since the processes $\Xi^{i,\star}_t$ give a coupling of the linearised Kac processes started at $(v_0,1), (w_0,1)$ respectively, we have \begin{equation} f_{0t}(v_0)-f_{0t}(w_0)= E_{(0,(v_0,w_0)\in V_0)} \left\{\langle f\circ p_1-f\circ p_2, \Gamma^0_t\rangle +  \langle f, \Gamma^1_t\rangle -\langle f, \Gamma^2_t\rangle \right\}.\end{equation} On the support of $\Gamma^0_t$, the difference $f\circ p_1-f\circ p_2$ is at most $3(1+|v|^2+|w|^2)|v_0-w_0|^\gamma$, and we can estimate the integral using (\ref{eq: growth of coupled}). The other terms only gain contributions from decoupled particles, and we can estimate both such terms using (\ref{eq: growth of decoupled}) and recalling that $|f|\le 1+|v|^2$. We therefore put everything together to conclude that \begin{equation}\begin{split}  |f_{0t}(v_0)-f_{0t}(w_0)|&\le CK^\alpha|v_0-w_0|^\gamma (1+|v_0|^2+|w_0|^2)\h z_t.\end{split}   \end{equation}  which is the regularity desired. Finally, since $K^\alpha \le  \exp(CK)$ only appears in the decoupling rate and appears only as a multiplicative factor, rather than in the exponent, it can be absorbed into $z_t$, by changing the value of $C$ if necessary. \end{proof} \subsection{Proof of Lemmas \ref{lemma: convergence of cutoff KP}, \ref{lemma: W-W stability for cutoff}}\label{subsec: proof of cutoff estimates}Combining the previous lemmas, we prove the two Lemmata \ref{lemma: convergence of cutoff KP}, \ref{lemma: W-W stability for cutoff} which are the main focus of this section. We begin with Lemma \ref{lemma: W-W stability for cutoff}, which is a simple application of Proposition \ref{lemma: representation formula} and the estimates in Lemma \ref{lemma: fst}. \begin{proof}[Proof of Lemma \ref{lemma: W-W stability for cutoff}]  Let us fix $\mu, \nu \in \mathcal{S}^p$, for $p$ to be chosen later. Let $\rho$ be the environment \begin{equation} \rho_t=\frac{1}{2}(\phi^K_t(\mu)+\phi^K_t(\nu))\end{equation} and let $f_{st}$ denote the functions given by (\ref{eq: fst}) in this environment. For any $f\in \mathcal{A}_\gamma$, we have \begin{equation} \langle f, \phi^K_t(\mu)-\phi^K_t(\nu)\rangle =\langle f_{0t},\mu-\nu\rangle \le z_t \h \mathfrak{w}_\gamma(\mu,\nu)\end{equation} where $z_t$ is as in Lemma \ref{lemma: fst}; by Proposition \ref{thrm:momentinequalities}, we bound \begin{equation} z_t\le \exp\left(cK(1+t)\Lambda_{2+\gamma}(\mu,\nu)\right)\end{equation} and so, optimising over $f$, \begin{equation} \mathfrak{w}_\gamma(\phi^K_t(\mu),\phi^K_t(\nu))\le \exp\left(cK(1+t)\Lambda_{2+\gamma}(\mu,\nu)\right) \h \mathfrak{w}_\gamma(\mu,\nu).\end{equation} Finally, we use (\ref{eq: mathfrak wgamma and Wp}) twice to convert both sides from $\mathfrak{w}_\gamma$ to $W_p$: for some $C=C(G,p,q,d), \alpha=\alpha(p,p')$, \begin{equation}\begin{split}\label{eq: convert everything to WP} W_p\left(\phi^K_t(\mu),\phi^K_t(\nu)\right)&\le C\Lambda_{p'}\left(\phi^K_t(\mu),\phi^K_t(\nu)\right)\mathfrak{w}_\gamma\left(\phi^K_t(\mu),\phi^K_t(\nu)\right)^\alpha\\ & \le C\Lambda_{p'}(\mu,\nu)\mathfrak{w}_\gamma(\mu,\nu)^\alpha \exp\left(cK(1+t)\Lambda_{2+\gamma}(\mu,\nu)\right)\\& \le C\Lambda_{p'}(\mu, \nu)^2\exp\left(cK(1+t)\Lambda_{2+\gamma}(\mu,\nu)\right) W_p(\mu,\nu)^{\alpha\gamma}\end{split}\end{equation} which proves the claim for a new choice of $\alpha$. \end{proof} For the case with an $N$-particle Kac process, we will need to control the stochastic integral term, uniformly over $f$ belonging to the class of test functions $\mathcal{A}_\gamma$. This is achieved with the following proposition. \begin{proposition} \label{prop: short time mg estimate} Let $\mu^{N,K}_t, t\ge 0$ be a $N$-particle, $K$-cutoff Kac process, and let $\mathcal{F}^N_t$ be the natural filtration. Let $\rho=(\rho_t)_{t\ge 0}$ be a potentially random environment, adapted to $\mathcal{F}^N_t$, such that \begin{equation} \label{eq: moment condition for environment} \lambda^\star=\left\|\hspace{0.1cm}\sup_{t\ge 0} \Lambda_{2+\gamma}(\rho_t)\hspace{0.1cm}\right\|_{L^\infty(\mathbb{P})}<\infty. \end{equation} For $f\in \mathcal{A}_\gamma$ and $t\ge s\ge 0$, let $f_{st}[\rho]$ denote the propagation in this environment, as described in Definition \ref{def: LKP}. Let $q\ge 2+\gamma$ and $a\ge 1$, and suppose that $\mu^{N,K}_t$ has an initial moment $\Lambda_q(\mu^{N,K}_0)\le a$. Let $m^{N,K}, \overline{m}^{N,K}$ be as in Definition \ref{def: JM+C}, and write \begin{equation} \label{eq: stopped stochastic integral}\widetilde{M}^{N,K,f}_t[\rho]=\int_{(0,t]\times \mathcal{S}_N} \langle f_{st}[\rho], \mu^N-\mu^{N,K}_{s-}\rangle (m^{N,K}-\overline{m}^{N,K})(ds, d\mu^N).\end{equation} In this notation, we have the bound \begin{equation} \label{eq: sup over matilde} \EE\left[\hspace{0.1cm} \sup_{t\le t_\mathrm{fin}}\hspace{0.1cm} \sup_{f\in \mathcal{A}_\gamma} \hspace{0.1cm} \widetilde{M}^{N,K,f}_t[\rho]\h \right] \le CaN^{-\eta}\h \exp\left(C\lambda^\star K(1+t_\mathrm{fin})\right)\end{equation} for some $C=C(d,p,\beta)$ and $\eta=\eta(d,p)>0$. Here, we emphasise that $\EE$ refers to the $L^1$ norm with simultaneous expectation over $\mu^{N,K}_t$ and the environment $\rho$. \end{proposition} This largely follows the same arguments as the proof of \cite[Theorem 1.1]{norris2016consistency}. The key difference is that the hypotheses on the environment $\rho$ guarantee an $L^\infty(\mathbb{P})$ control on the quantities \begin{equation} z_t=\exp\left(cK\int_0^t \Lambda_{2+\gamma}(\rho_u)du\right) \end{equation} \begin{equation} y_\beta(t)=z_t\hspace{0.1cm}\sup_{0\le s\le s'\le t}\left[(s'-s)^{-1}\int_s^{s'}\Lambda_{2+\gamma}(\rho_u)du\right], \hs 0<\beta\le 1\end{equation} in terms of the multiplicative factor $e^{C\lambda^\star K(1+t)}$, which describe the continuity of $f_{st}(v)$ in $v$ and $s$ respectively. Finally, we indicate how these results may be used to prove Lemma \ref{lemma: convergence of cutoff KP}. \begin{proof}[Sketch Proof of Lemma \ref{lemma: convergence of cutoff KP}] Let us consider the linearised Kac process in the random environment \begin{equation} \rho^N_t=\frac{\mu^{N,K}_t+\phi^K_t(\mu^{N,K}_0)}{2} \end{equation} as in Lemma \ref{lemma: representation formula}, and for $b\ge 1$, consider the stopping times $T^N_b$ defined in (\ref{eq: defn of TB}) for the $(2+\gamma)^\text{th}$ moment. Let us write $\widetilde{M}^{N,K,f,b}_t$ for the stochastic integrals in (\ref{eq: stopped stochastic integral}) in the environment $\rho^{T^N_b}$. \bigskip \\ We consider the events $\{T_b\le t_\mathrm{fin}\}, \{T_b> t_\mathrm{fin}\}$ separately. On the event $\{T_b >  t_\mathrm{fin}\}$, we have the equalities \begin{equation} M^{N,K,f}_t=\widetilde{M}^{N,K,f,b}_t\text{ for all }f\in \mathcal{A}_\gamma\text{ and all }t\le  t_\mathrm{fin}\end{equation} while on $\{T_b\le t_\mathrm{fin}\}$ we have the trivial bound \begin{equation}\sup_{t\le  t_\mathrm{fin}}\mathfrak{w}_\gamma(\mu^{N,K}_t,\phi^K_t(\mu))\le 4.\end{equation} Combining, we have the bound \begin{equation} \sup_{t\le t_\mathrm{fin}} \h  \mathfrak{w}_\gamma\left(\mu^{N,K}_t,\phi^K_t(\mu^{N,K}_0)\right)\le \sup_{f\in \mathcal{A}_\gamma, t\le  t_\mathrm{fin}}\left\{ \widetilde{M}^{N,K,f,b}_t\right\}+4 \cdot 1(T^N_b\le  t_\mathrm{fin}).\end{equation} Since $q>2+\gamma$, the moment hypothesis on $\mu^{N,K}_0$ implies $\Lambda_{2+\gamma}(\mu^{N,K}_0)\le a$ almost surely, which is propagated to $\phi^K_t(\mu_0)$ by Proposition \ref{thrm:momentinequalities}. The first term is therefore controlled by Proposition \ref{prop: short time mg estimate}, with $\lambda^\star \le b+Ca$ for some constant $C$. We now take $b=Ca$, for some large $C$; by Lemma \ref{lemma: concentration of moments}, $C$ can be chosen so that $\PP(T^N_b\le  t_\mathrm{fin})\le CaN^{-1} t_\mathrm{fin}$. Combining, we obtain\begin{equation} \EE\left[\h\sup_{t\le  t_\mathrm{fin}}\h  \mathfrak{w}_\gamma\left(\mu^{N,K}_t,\phi^K_t(\mu^{N,K}_0)\right)\h\right] \le Ca N^{-\eta}\exp(C aK(1+ t_\mathrm{fin}))+Ca t_\mathrm{fin}N^{-1}\end{equation} and keeping the worse term\begin{equation} \EE\left[\h\sup_{t\le  t_\mathrm{fin}}\h  \mathfrak{w}_\gamma\left(\mu^{N,K}_t,\phi^K_t(\mu^{N,K}_0)\right)\h\right]\le  Ca N^{-\eta}\exp(C aK(1+ t_\mathrm{fin})). \end{equation} To convert this approximation into $W_p$, we argue as in (\ref{eq: convert everything to WP}). Fix $p'\in (p+2, q)$; thanks to the comparisons in Section \ref{sec: metrics}, for some $\alpha>0$, \begin{equation}\begin{split}&\sup_{t\le t_\mathrm{fin}} W_p\left(\mu^{N,K}_t,\phi^K_t(\mu^{N,K}_0)\right)\\& \hs \hs \le \left(\sup_{t\le t_\mathrm{fin}} \mathfrak{w}_\gamma\left(\mu^{N,K}_t,\phi^K_t(\mu^{N,K}_0)\right)\right)^\alpha\left(\sup_{t\le t_\mathrm{fin}} \Lambda_{p'}\left(\mu^{N,K}_t,\phi_t(\mu^{N,K}_0)\right)\right). \end{split}\end{equation} We now use H\"older's inequality with indexes $\frac{q}{p'}, \frac{q}{q-p'}$ and control the moment term with Proposition \ref{thrm:momentinequalities} to find that, for some new $\alpha>0$,\begin{equation}\begin{split} &\EE\left[\sup_{t\le t_\mathrm{fin}} W_p\left(\mu^{N,K}_t,\phi^K_t(\mu)\right)\right]\\&\hspace{1.5cm}\le C\h\EE\left[\sup_{t\le t_\mathrm{fin}} \mathfrak{w}_\gamma\left(\mu^{N,K}_t,\phi^K_t(\mu^{N,K}_0)\right)\right]^\alpha\h\EE\left[\sup_{t\le t_\mathrm{fin}}\Lambda_q(\mu^{N,K}_t,\phi^K_t(\mu^{N,K}_0))\right] \\[1ex] & \hspace{1.5cm}  \le CaN^{-\alpha \eta}\exp(CaK(1+t_\mathrm{fin})) \cdot Ca(1+t_\mathrm{fin}).\end{split} \end{equation} Absorbing constants and the moment factors into the exponent, we have shown that, for some $\alpha=\alpha(p,q,d)>0$, \begin{equation}  \EE\left[\sup_{t\le t_\mathrm{fin}} W_p\left(\mu^{N,K}_t,\phi^K_t(\mu^{N,K}_0)\right)\right] \le N^{-\alpha}\exp(CaK(1+t_\mathrm{fin}). \end{equation} The conclusion now follows by comparing $\phi^K_t(\mu^{N,K}_0)$ and $\phi^K_t(\mu_0)$ using Lemma \ref{lemma: W-W stability for cutoff}. \end{proof}

\section{Tanaka Coupling of the Kac Processes} \label{sec: TK}
In this section, we will exhibit the key coupling of Kac processes, and a family of lemmas which control how fast the distance between the coupling can grow.  \subsection{Accurate Tanaka's Trick} We begin with the following `accurate Tanaka Lemma', which generalises that of \cite{fournier2008uniqueness}. Our result is slightly more general, in that we allow any $d\ge 3$, while the result cited applies for only $d=3$. \begin{lemma}\label{lemma: ATT}[Accurate Tanaka's Trick] There exists a measurable function $R: \RRd\times\RRd\rightarrow \text{Isom}(\SSd)$ such that, for all $X, Y \in \RRd$ and $\varphi \in \SSd$, we have  \begin{equation} \Gamma(X,\varphi)\cdot \Gamma(Y, R\varphi)= \varphi_1^2(X\cdot Y)+(1-\varphi_1^2)|X||Y|\ge X\cdot Y. \end{equation}  Here, $\varphi_1$ denotes the first coordinate of $\varphi\in \SSd\subset \mathbb{R}^{d-1}$.\end{lemma} \begin{proof} First, the case where either $X, Y=0$ is vacuous and can be omitted. Let us write, throughout, $S_X$ for the set \begin{equation} S_X=\{u\in \RRd: |u|=|X|, u\cdot X=0\}.\end{equation} By considering separately the cases where $X, Y$ are and are not colinear, we observe that we may choose $j^1_X, j^1_Y$ such that \begin{equation}\text{dim Span}(X,Y, j^1_X, j^1_Y)=2,\hs j^1_X\in S_X, j^1_Y\in S_Y;\hs j^1_X\cdot j^1_Y=X\cdot Y.\end{equation} With some care, the map $(X,Y)\mapsto (j^1_X, j^1_Y)$ can further be constructed to be measurable. We now construct, in a measurable way, $u_2, ...u_{d-1}$ as an orthonormal basis for $\text{Span}(X,j^1_X)^\perp=\text{Span}(Y, j^1_Y)^\perp$, and set \begin{equation} j^2_X=|X|u_2, \h j^3_X=|X|u_3,\h ....,\h j^{d-1}_X=|X|u_{d-1};\end{equation}
\begin{equation} j^2_Y=|Y|u_2, \h j^3_Y=|Y|u_3,\h ....,\h j^{d-1}_Y=|Y|u_{d-1}.\end{equation} Now, $\{j^1_X,...j^{d-1}_X\}$ are orthonogal, and lie in $S_X$, so there is a unique isometry $P_X\in \text{Isom}(\SSd)$ such that \begin{equation} \Gamma(X, P_Xe_k)=j^k_X,\hs k=1,...,d-1\end{equation} and similarly for $P_Y$. We now observe, for all $\varphi\in \SSd$, \begin{equation}\begin{split}  \Gamma(X,P_X\varphi)\cdot\Gamma(Y,P_Y\varphi)=\sum_{k=1}^{d-1} \phi_k^2 \h j_X\cdot j_Y &=\phi_1^2\h(X\cdot Y)+(1-\varphi_1^2)|X||Y| \\ &\ge \phi_1^2\h(X\cdot Y)+(1-\varphi_1^2)(X\cdot Y)\\ & =X\cdot Y\end{split}\end{equation} which implies the result when we define $R(X,Y)=P_YP_X^{-1}$.\end{proof} 
\subsection{Tanaka-Povzner Lemmata}
The key tool at the heart of our results is the following variant of some calculations in \cite[Lemmas 3.1, 3.3]{fournier2016rate}. The key point is the appearance of a large negative term, similar to that arising in the Povzner inequalities, which ensures the cancellation of `bad' terms and leads to a Gr\"onwall inequality. \begin{lemma}\label{lemma: big calculation} Let us write, for $v,\widetilde{v},v_\star, \widetilde{v}_\star\in \RRd$, $z\in (0,\infty), \varphi \in \SSd$ and $K<\infty$, \begin{align} a&=a(v,v_\star,z,\varphi); \hs v'=v+a; \hs v'_K=v+a_K(v,v_\star,z,\varphi); \\ \widetilde{a}_K&={a}_K(\widetilde{v},\widetilde{v}_\star,z,R(v-v_\star,\widetilde{v}-\widetilde{v}_\star)\varphi); \hs \widetilde{v}'_K=\widetilde{v}+\widetilde{a}_K \end{align} for the isometry $R(v-v_\star,\widetilde{v}-\widetilde{v}_\star): \SSd\rightarrow\SSd$ constructed in Lemma \ref{lemma: ATT}. Define \begin{equation} \begin{split}  &\mathcal{E}^1_K(v,\widetilde{v}, v_\star,\widetilde{v}_\star)=\int_0^\infty dz \int_{\SSd} d\varphi \left(d_p^2(v', \widetilde{v}'_K)-d_p^2(v,\widetilde{v})\right).\end{split}\end{equation} Let us define, for $p\ge 2$, \begin{equation} \lambda_p=|\mathbb{S}^{d-2}|\int_0^{\pi/2} \left(1-\left(\frac{1+\cos \theta}{2}\right)^{p/2}\right)\beta(\theta)d\theta.\end{equation} Then there exists $K_0(p)$,  constants $c=c(G,d)$ and $C=C(G,d,p)$, such that, whenever $K\ge K_0(p)$, we have \begin{equation} \begin{split} \label{eq: EK real}  \mathcal{E}_K(v,\widetilde{v},v_\star,\widetilde{v}_\star)& \le \left(c+\left(c-\frac{\lambda_p}{2}\right)|v|^{p+\gamma}+\left(c-\frac{\lambda_p}{2}\right)|\widetilde{v}|^{p+\gamma}\right)|v-\widetilde{v}|^2\\&\hs +c\left(|v_\star|^{p+\gamma}+|\widetilde{v}_\star|^{p+\gamma}\right)|v_\star-\widetilde{v}_\star|^2 \\&\hs +C\left(|v_\star|^{p+\gamma}+|\widetilde{v}_\star|^{p+\gamma}\right)d_p^2\left(v,\widetilde{v}\right)\\&\hs +C\left(|v|^{p+\gamma}+|\widetilde{v}|^{p+\gamma}\right)d_p^2\left(v_\star,\widetilde{v}_\star\right)\\ & \hs +C K^{1-1/\nu}(1+|v|^l+|v_\star|^l+|\widetilde{v}|^l+|\widetilde{v}_\star|^l)\end{split} \end{equation} where $l=p+2+\gamma$. \end{lemma}  \begin{rmk}Let us motivate this lemma, which is not necessarily transparent. We obtain, in expanding $\mathcal{E}_K$, the noise term in the final line, and terms proportional to $|v-\widetilde{v}|^2, |v_\star-\widetilde{v}_\star|^2$ with all possible polynomial weightings of order $p+\gamma$. The difficult terms are those like $|v|^{p+\gamma}|v-\widetilde{v}|^2$, which prevent a Gr\"onwall estimate. However, we ensure that the coefficients of such terms are independent of $p$, which allows us to cancel all such terms by the negative `Pozvner term' appearing in the first line by making $p$ large. \end{rmk} We will also use the following variants, which will be used to prove a local uniform estimate on our coupling. \begin{lemma}\label{lemma: quadratic pozvner} In the notation of the previous lemma, define also \begin{equation} \mathcal{Q}_K(v,\widetilde{v},v_\star,\widetilde{v}_\star)=\int_0^\infty dz \int_{\SSd}d\varphi \left(d^2_p(v',\widetilde{v}'_K)+d^2_p(v_\star',\widetilde{v}'_{\star K})-d^2_p(v,\widetilde{v})-d_p^2(v_\star,\widetilde{v}_\star)\right)^2.\end{equation} 
Then, for some C=C(G,d,p), we have 
\begin{equation}\begin{split}\label{eq: conclusion of QP}
            \mathcal{Q}_K(v,\widetilde{v},v_\star,\widetilde{v}_\star) & \le C(1+|v|^{2l}+|v_\star|^{2l}+|\widetilde{v}|^{2l}+|\widetilde{v}_\star|^{2l}).
        \end{split} \end{equation}  where $l=p+2+\gamma$ is as in Lemma \ref{lemma: big calculation}. \end{lemma}
\subsection{Coupling of the Kac Process} We now present our coupling of the Kac processes. Let $\mathcal{V}^N_t=(V^1_t,...,V^N_t)$ be a noncutoff Kac process, and let $\Nn^{\{ij\}}$, $1\le i \neq j \le N$, be the Poisson random measures of intensity $2N^{-1}dtd\varphi dz$ driving $\mathcal{V}^N_t$, so that\begin{equation}\label{eq: SDE form for coupling} V^i_t=V^i_0+\sum_{j\neq i}\int_{(0,t]\times \SSd\times(0,\infty)} a(V^i_{s-},V^j_{s-},z,\varphi)\h\Nn^{\{ij\}}(ds,d\varphi,dz), \hs i=1,...,N. \end{equation} Let us fix $\widetilde{\mathcal{V}}^{N,K}_0=(\widetilde{V}^{1,K}_0,....,\widetilde{V}^{N,K}_0)$ and define $\widetilde{\mathcal{V}}^{N,K}_t=(\widetilde{V}^{1,K}_t,...\widetilde{V}^{N,K}_t)$ by  \begin{equation} \label{eq: coupled cutoff} \widetilde{V}^{i,K}_t=\widetilde{V}^{i,K}_0+\sum_{j\neq i}\int_{(0,t]\times \SSd\times(0,\infty)} a_K(\widetilde{V}^{i,K}_{s-},\widetilde{V}^{j,K}_{s-},z,R^{i,j}_{s-}\varphi)\h\Nn^{\{ij\}}(ds,d\varphi,dz); \end{equation} \begin{equation} \label{eq: coupled cutoff2} R^{i,j}_t:=R(V^i_{t}-V^j_{t},\widetilde{V}^{i,K}_{t}-\widetilde{V}^{j,K}_{t})\end{equation} where $R:\RRd\times\RRd\rightarrow \text{Isom}(\SSd)$ is the isometry constructed in Lemma \ref{lemma: ATT}. We remark first that the rates of $\Nn^{\{ij\}}$ are all finite on the support of $a_K$, so that the stochastic differential equation (\ref{eq: coupled cutoff}, \ref{eq: coupled cutoff2}) is really a recurrence relation; in particular, $\widetilde{\mathcal{V}}^{N,K}_t$ is uniquely defined by the above equations. Next, we claim that $\widetilde{\mathcal{V}}^{N,K}_t$ is a $K$-cutoff Kac process on $N$ particles; this is the content of the following lemma, which is adapted from a similar claim \cite[Proposition 4.4]{fournier2016rate}. \begin{lemma} Let $\mathcal{V}^{N}_t$ be a noncutoff Kac process, and fix $\widetilde{\mathcal{V}}^{N,K}_0$. Then the process $(\widetilde{\mathcal{V}}^{N,K}_t)_{t\ge 0}$ constructed by (\ref{eq: coupled cutoff}) is a cutoff Kac process starting at $\widetilde{\mathcal{V}}^{N,K}_0.$ \end{lemma} \begin{proof} It is sufficient to show that $\widetilde{\mathcal{V}}^{N,K}_t$ solves (\ref{eq: cutoff SDE}) for some choice of Poisson random measures. Let $\Nn^{\{ij\}}, 1\le i\neq j\le N$ be the random measures driving $\mathcal{V}^N_t$ and $\mathcal{F}_t$ their natural filtration; define random measures $\widetilde{\Nn}^{\{ij\}}$ by specifying, for bounded and compactly supported $f: (0,\infty)\times \SSd\times(0,\infty)\rightarrow \mathbb{R}$, \begin{equation} \begin{split} \int_{(0,\infty)\times \SSd\times(0,\infty)}f(t,\varphi,z)\widetilde{\Nn}^{\{ij\}}(dt,d\varphi,dz)&\\& \hspace{-4cm}=\int_{(0,\infty)\times\SSd\times(0,\infty)}f(t,R^{ij}_{t-}\varphi,z)\Nn^{\{ij\}}(dt,d\varphi,dz).\end{split}\end{equation} The construction (\ref{eq: coupled cutoff}) is now exactly that (\ref{eq: cutoff SDE}) holds for these measures. Moreover, since $R^{ij}_{t-}$ preserves the uniform measure $d\varphi$ and is previsible for the filtration $\mathcal{F}_t$, it follows that, if $H$ is a bounded and compactly supported previsible function on $(0,\infty)\times \SSd\times(0,\infty)$, the process \begin{equation} \begin{split} & \int_{(0,t]\times\SSd\times(0,\infty)}H(s,\varphi,dz)(\widetilde{\Nn}^{\{ij\}}(dt,d\varphi,dz)-2N^{-1}dsd\varphi dz) \\[1ex] & \hs  =\int_{(0,t]\times \SSd\times(0,\infty)} H(s,R^{ij}_{s-}\varphi,z)(\Nn^{\{ij\}}(ds,d\varphi,dz)-2N^{-1}dsd\varphi dz) \end{split}\end{equation} is a martingale, which implies that $\widetilde{\Nn}^{\{ij\}}$ is a Poisson random measure of rate $2N^{-1}dtd\varphi dz$, as desired.  \end{proof} Our first result on the coupling is the following, which proves Theorem \ref{thrm: approximate with cutoff}.

\begin{lemma}[Convergence of the Tanaka Coupling]\label{lemma: Tanaka Coupling of KP} There exists $p_0=p_0(G,d)$ and, for $p>p_0$, there exists $K_0=K_0(G,p,d)$ such that, whenever $p>p_0$ and $K>K_0$, we have the following estimates. \bigskip \\   Let $\mathcal{V}^N_t$ be a noncutoff labelled Kac process and $\widetilde{\mathcal{V}}^{N,K}_0\in \mathbb{S}_N$. Let $\widetilde{\mathcal{V}}^{N,K}_t$ be the cutoff Kac process constructed in (\ref{eq: coupled cutoff}), and define \begin{equation}
    \overline{d}_p^2(t):=\frac{1}{N}\sum_{i=1}^N d_p^2\left(V^i_t, \widetilde{V}^{i,K}_t\right).
\end{equation} Suppose the initial data $\mathcal{V}^N_0, \widetilde{\mathcal{V}}^{N,K}_0$ are such that the associated empirical measures $\mu^N_0,\widetilde{\mu}^{N,K}_0$ satisfy moment bounds 
\begin{equation} \max\left(\Lambda_{l}(\mu^N_0),\Lambda_{l}(\widetilde{\mu}^{N,K}_0)\right)\le a_2; \end{equation}
\begin{equation} \max\left(\Lambda_{q}(\mu^N_0),\Lambda_{q}(\widetilde{\mu}^{N,K}_0)\right)\le a_3; \end{equation}with $l$ as in Lemma \ref{lemma: big calculation} and $q=2l$, and for some $a_2, a_3> 1$. Fix $b>1$, and let $T^N_b$ be the stopping time (\ref{eq: defn of TB}) for the empirical measures $\mu^N_t$ of $\mathcal{V}^N_t$, with $p+\gamma$ in place of $p$, and similarly $T^{N,K}_b$ for $\mathcal{V}^{N,K}_t$. Then there exists $C=C(p,G,d)$ such that, for all $t\ge 0$, \begin{equation}\begin{split}\label{eq: PW tanaka} \EE\left[\overline{d}_p^2(t) \right] &\le e^{Cb(1+t)}\left(\overline{d}_p^2(0) +a_2tK^{1-1/\nu}\right)+ a_3C\PP(T^{N,K}_b\land T^N_b\le t)^{1/2}\end{split}\end{equation}  and, for all $t_\mathrm{fin}\ge 0$, \begin{equation}\begin{split}\label{eq: LU tanaka result} \EE\left[\sup_{t\le t_\mathrm{fin}}\overline{d}_p^2(t) \right] &\le e^{Cb(1+t_\mathrm{fin})}\left(\overline{d}_p^2(0) +a_2tK^{1-1/\nu}+\frac{Ca_3t_\mathrm{fin}^{1/2}}{N^{1/2}}\right) \\& \hs\hs\hs+ a_3C(1+t_\mathrm{fin})\PP(T^{N,K}_b\land T^N_b\le t)^{1/2}.\end{split}\end{equation} 
  \end{lemma}This is the key result from which our results follow. Let us make the following remarks. \begin{remark}\begin{enumerate}[label=\roman*).]
      \item We will show below that this essentially establishes Theorem \ref{thrm: approximate with cutoff}. The form presented here, where we are free to choose $b$, is useful for dealing with the well-posedness issues deferred from Proposition \ref{prop: labelled and unlabelled dynamics}. \item In principle, one could perform a finer analysis for $\mathcal{Q}_K$, to replace the third term in (\ref{eq: LU tanaka result}) with an error in terms of $\overline{d}_p^2(0), K^{-\alpha}$, for some $\alpha>0$. In this way, we would obtain an estimate for the uniform convergence on compacts in probability of $\widetilde{\mathcal{V}}^{N,K}$, as $K\rightarrow \infty$ with $N$ fixed, and which is uniform in $N$. Since we are mostly interested in a limit where $N,K \rightarrow \infty$ simultaneously, we will not explore this.
  \end{enumerate}  \end{remark} Before turning to the proofs of these bounds, let us show some applications. We will now deduce Theorem \ref{thrm: approximate with cutoff}, and record a corollary at the level of unlabelled cutoff processes which is natural for later analyses.
  \begin{proof}[Proof of Theorem \ref{thrm: approximate with cutoff}] Let $p_0(G,d), K_0(G,p,d)$ be as above, and fix $p>p_0, K>K_0$. Let us fix $\mu^N_0$, $\widetilde{\mu}^{N,K}_0, a_1, a_2, a_3$ as in the statement of the Theorem, and choose  $\mathcal{V}^{N}_0 \in \theta_N^{-1}(\mu^N_0), \widetilde{\mathcal{V}}^{N,K'}_0\in \theta_N^{-1}(\widetilde{\mu}^{N,K}_0)$ corresponding to $\mu^{N,K}_0, \widetilde{\mu}^{N,K'}_0$ which achieve the optimal coupling \begin{equation} W^2_p(\mu^{N,K}_0,\widetilde{\mu}^{N,K'}_0)=\frac{1}{N}\sum_{i=1}^N d_p^2(V^i_0, \widetilde{V}^{i,K'}_0). \end{equation} Now, let $\mathcal{V}^N_t$ be a noncutoff labelled Kac process starting at $\mathcal{V}^N_0$, and write $\mu^N_t=\theta_N(\mathcal{V}^N_t)$ for the process of empirical measures. Let $\mathcal{V}^{N,K}_t$ be the Kac processes constructed by Lemma \ref{lemma: Tanaka Coupling of KP} for the initial data $\mathcal{V}^N_0$ respectively with cutoff parameter $K$, and let $\mu^{N,K}_t$ be the associated empirical measures. We observe that \begin{equation}
      W_p^2\left(\mu^N_t,\widetilde{\mu}^{N,K}_t\right)\le \frac{1}{N}\sum_{i=1}^N d_p^2\left(V^i_t,\widetilde{V}^{i,K}_t\right)=\overline{d}^2_p(t)
  \end{equation} which we control by the previous lemma to obtain, for some $C$ and all $t\ge 0, b>1$ \begin{equation}\begin{split}
      \EE\left[W_p^2\left(\mu^N_t,\widetilde{\mu}^{N,K}_0\right)\right]&\le e^{Cb(1+t)}\left(W_p^2\left(\mu^N_0,\widetilde{\mu}^{N,K}_0\right)+a_2K^{1-1/\nu}\right)\\&\hs\hs\hs\h\h+Ca_3\PP\left(T^N_b\land T^{N,K}_b\le t\right)^{1/2}.\end{split}
  \end{equation} where $T^N_b, T^{N,K}_b$ are as above. Now, taking $b=Ca_1$ for some large $C=C(p)$, we use Lemma \ref{lemma: concentration of moments} to control the final term and obtain, for some $C$, \begin{equation}\begin{split}
      \PP\left(T^N_{Ca_1}\land T^{N,K}_{Ca_1}\le t\right)& \le   \PP\left(T^N_{Ca_1}\le t\right)+  \PP\left(T^{N,K}_{Ca_1}\le t\right)
  \\& \le Cta_3N^{-1}\end{split} \end{equation} as desired. We obtain (\ref{eq: statement of LU tanaka}) from (\ref{eq: LU tanaka result}) for the same processes $\mu^N_t,\widetilde{\mu}^{N,K}_t$ in exactly the same way.  \end{proof} We obtain, in a very similar way, the following corollary for coupling cutoff processes.
\begin{cor}\label{cor: couple cutoff}[Coupling of Cutoff Kac Processes] Let $p, q,l, K_0, C$ be as in Lemma \ref{lemma: Tanaka Coupling of KP}, and let $K'\ge K>K_0(G,p,d)$. Let $\mu^{N,K}_0,\widetilde{\mu}^{N,K'}_0\in \mathcal{S}_N$, with moments   \begin{equation} \Lambda_{p+\gamma}\left(\mu^N_0,\widetilde{\mu}^{N,K}_0\right)\le a_1; \hs  \Lambda_{l}\left(\mu^N_0,\widetilde{\mu}^{N,K}_0\right)\le a_2; \hs  \Lambda_{q}\left(\mu^N_0,\widetilde{\mu}^{N,K}_0\right)\le a_3. \end{equation}  Then there exists a coupling of cutoff Kac processes $\mu^{N,K}_t, \widetilde{\mu}^{N,K'}_t$ with cutoff parameters $K, K'$ respectively, such that \begin{equation} \EE\left[W^2_p(\mu^{N,K}_t, \widetilde{\mu}^{N,K'}_t)\right]\le e^{Ca_1(1+t)}(W^2_p(\mu^{N,K}_0,\widetilde{\mu}^{N,K}_0)+a_2tK^{1-1/\nu})+a_3^2CN^{-1/2}t. \end{equation}   \end{cor} \begin{proof}[Sketch Proof] The proof is very similar to the above, and we will sketch the main points. Let us construct $\mathcal{V}^N_0 \in \theta_N^{-1}(\mu^N_0)$ and $ \widetilde{\mathcal{V}}^{N,K'}_0 \in \theta_N^{-1}(\widetilde{\mu}^{N,K'}_0)$ as in the previous proof. Following the previous proof, construct a \emph{non}cutoff labelled process $\mathcal{V}^N_t$ starting at $\mathcal{V}^N_0$ and a $K'$-cutoff $\widetilde{\mathcal{V}}^{N,K'}_t$ starting at $\widetilde{\mathcal{V}}^{N,K'}_0$. We take $\mu^N_t, \widetilde{\mu}^{N,K'}_t$ to be the associated empirical measures, which are (unlabelled) Kac processes. \bigskip \\ We now repeat this argument to construct a $K$-cutoff process $\mathcal{V}^{N,K}_t$ starting at the same point $\mathcal{V}^N_0=\mathcal{V}^{N,K}_0$, and let $\mu^{N,K}_t$ be the associated empirical measures. The same argument as the previous proof establishes controls on \begin{equation}
    \EE\left[W_p^2\left(\mu^N_t, \widetilde{\mu}^{N,K'}_t\right)\right]; \hs    \EE\left[W_p^2\left(\mu^N_t, {\mu}^{N,K}_t\right)\right].
\end{equation} Recalling the relaxed triangle inequality (\ref{eq: triangle inequality}), we combine these to find the desired estimate. \end{proof}
The estimate in Lemma \ref{lemma: Tanaka Coupling of KP} also allows us to establish Proposition \ref{prop: wellposedness/uniqueness}, the proof of which is deferred to the appendix.  
\subsection{Proof of Lemma \ref{lemma: Tanaka Coupling of KP}}\begin{proof}[Proof of Lemma \ref{lemma: Tanaka Coupling of KP}]Let $p\ge 2$ to be decided later, and consider the processes \begin{equation} M^i_t=d_p^2(V^i_t,V^{i,K}_t)-d_p^2(V^i_0,\widetilde{V}^{i,K}_0)-\frac{2}{N}\int_0^t \sum_{j=1}^N \mathcal{E}_K(V^i_s, V^{i,K}_s,V^j_s,V^{j,K}_s)ds \end{equation}for $1\le i\le N$, and their average\begin{equation} \label{eq: overline M}\begin{split}\overline{M}_t=\frac{1}{N}\sum_{i=1}^NM^i_t=\overline{d}_p^2(t)-\int_0^t \overline{\mathcal{E}}_K(s)ds\end{split}\end{equation} where we define \begin{equation} \overline{\mathcal{E}}_K(t):=\frac{2}{N^2}\sum_{i,j=1}^N \mathcal{E}_K(V^i_t,\widetilde{V}^{i,K}_t,V^j_t,\widetilde{V}^{j,K}_t).
\end{equation} By classical results in the theory of Markov chains \cite{darling2008differential}, each $M^i_t$ is a martingale, and hence so is $\overline{M}$. By Lemma \ref{lemma: big calculation}, provided $K$ is large enough, depending on $G,p,d$, we have, for some $c=c(G,d), C=C(G,d,p)$, \begin{equation}\begin{split}  \mathcal{E}_K(V^i_t,\widetilde{V}^{i,K}_t,V^j_t,\widetilde{V}^{j,K}_t)& \le \left(c+\left(c-\frac{\lambda_p}{2}\right)|V^i_t|^{p+\gamma}+\left(c-\frac{\lambda_p}{2}\right)|\widetilde{V}^{i,K}_t|^{p+\gamma}\right)|V^i_t-\widetilde{V}^{i,K}_t|^2\\&\hs +c\left(|V^j_t|^{p+\gamma}+|\widetilde{V}^{j,K}_t|^{p+\gamma}\right)|V^j_t-\widetilde{V}^{j,K}_t|^2 \\&\hs +C\left(|V^j_t|^{p+\gamma}+|\widetilde{V}^{j,K}_t|^{p+\gamma}\right)d_p^2\left(V^i_t,\widetilde{V}^{i,K}_t\right)\\&\hs +C\left(|V^i_t|^{p+\gamma}+|\widetilde{V}^{i,K}_t|^{p+\gamma}\right)d_p^2\left(V^j_t,\widetilde{V}^{j,K}_t\right)\\ & \hs +C K^{1-1/\nu}(1+|V^i_t|^l+|V^j_t|^l+|\widetilde{V}^{j,K}_t|^l+|\widetilde{V}^{i,K}_t|^l).\end{split} \end{equation} Let us now take the average over all $i,j$. The first two lines can be absorbed together, as can the third and the fourth; for some new constants $c,C$ with the same dependence as above,  \begin{equation}\label{eq: final sum}\begin{split}  \overline{\mathcal{E}}_K(t)& \le  \frac{1}{N}\sum_{i=1}^N \left(c+\left(c-\frac{\lambda_p}{2}\right)|V^i_t|^{p+\gamma}+\left(c-\frac{\lambda_p}{2}\right)|\widetilde{V}^{i,K}_t|^{p+\gamma}\right)|V^i_t-\widetilde{V}^{i,K}_t|^2 \\&\hs \hs \hs +C\left(\Lambda_{p+\gamma}(\mu^N_t)+\Lambda_{p+\gamma}(\widetilde{\mu}^{N,K}_t)\right)\overline{d}_p^2(t)\\[1ex] & \hs\hs\hs +C K^{1-1/\nu}\left(\Lambda_l(\mu^N_t)+\Lambda_l(\widetilde{\mu}^{N,K}_t)\right).\end{split} \end{equation} Let us now choose $p$. We recall that $c$ does \emph{not} depend on $p$, and return to the definition \begin{equation} \lambda_p:=|\mathbb{S}^{d-2}|\int_0^{\pi/2}\left(1-\left(\frac{1+\cos \theta}{2}\right)^{p/2}\right)\beta(\theta)d\theta. \end{equation} As $p\rightarrow \infty$, the term in parentheses converges up to $1$ for any $\theta\neq 0$, and so $\lambda_p$ converges to $|\mathbb{S}^{d-2}|\int_0^{\pi/2} \beta(\theta)d\theta=\infty$ by monotone convergence. In particular, there exists some $p_0$, depending only on $G,d$ such that, for all $p> p_0$, $\lambda_p\ge 2c$, and for such $p$, the first line of (\ref{eq: final sum}) can be absorbed into the second: \begin{equation} \label{eq: final final sum}\begin{split} \overline{ \mathcal{E}}_K(t) &\le C\left(\Lambda_{p+\gamma}(\mu^N_t)+\Lambda_{p+\gamma}(\widetilde{\mu}^{N,K}_t)\right)\overline{d}_p^2(t) + CK^{1-1/\nu}\left(\Lambda_l(\mu^N_t)+\Lambda_l(\widetilde{\mu}^{N,K}_t)\right)\end{split} \end{equation} whence \begin{equation}\begin{split} \label{eq: pregronwall dbar}
    \overline{d}_p^2(t)& \le \overline{d}_p^2(0)+C\int_0^t \left(\Lambda_{p+\gamma}(\mu^N_s)+\Lambda_{p+\gamma}(\widetilde{\mu}^{N,K}_s)\right)\overline{d}_p^2(s)ds \\[1ex] & \hs \hs  +CK^{1-1/\nu}\int_0^t\left(\Lambda_l(\mu^N_s)+\Lambda_l(\widetilde{\mu}^{N,K}_s)\right)ds +\overline{M}_t.\end{split}
\end{equation} Let us now write $T:=T^N_b\land T^{N,K}_b$ for the stopping times $T^N_b, T^{N,K}_b$ defined in the statement, and consider the moment prefactor in (\ref{eq: final final sum}, \ref{eq: pregronwall dbar}). We recall from Propositions \ref{thrm:momentinequalities}, \ref{thrm:noncutoffmomentinequalities} that, almost surely, for all $t\ge 0$, \begin{equation}\Lambda_{p+\gamma}(\mu^N_t)\le 2^{\frac{p+\gamma}{2}+1}\Lambda_{p+\gamma}(\mu^N_{t-})\end{equation} and similarly for $\widetilde{\mu}^{N,K}_t$. The moment factor is therefore at most $2b$ for all $s\le T$, and so we obtain, for all $t\ge 0$, \begin{equation} \int_0^{t\land T} \left(\Lambda_{p+\gamma}(\mu^N_s)+\Lambda_{p+\gamma}(\widetilde{\mu}^{N,K}_s)\right)\overline{d}_p^2(s)ds\le 2b\int_0^t \overline{d}_p^2(s\land T)ds. \end{equation} Stopping (\ref{eq: pregronwall dbar}) at $T$, we therefore obtain, for all $t\ge 0$, \begin{equation}\label{eq: gronwall dbar}\begin{split}
    \overline{d}_p^2(t\land T)&\le \overline{d}_p^2(0)+Cb\int_0^t d_p^2(s\land T) ds\\& \hs \hs  +CK^{1-1/\nu}\int_0^t\left(\Lambda_l(\mu^N_s)+\Lambda_l(\widetilde{\mu}^{N,K}_s)\right)ds+\overline{M}_{t\land T}. \end{split}
\end{equation} For the first item, we fix $t\ge 0$, and take expectations of (\ref{eq: gronwall dbar}). By optional stopping, $\EE[\overline{M}_{t\land T}]=0$, and we use the moment estimates in Propositions \ref{thrm:momentinequalities}, \ref{thrm:noncutoffmomentinequalities} to control the first term on the second line: \begin{equation}
    \EE\left[\int_0^t(\Lambda_l(\mu^N_s)+\Lambda_l(\widetilde{\mu}^{N,K}_s)ds\right] \le Cta_2.
\end{equation} We therefore use Gr\"onwall's Lemma to obtain \begin{equation}\label{eq: the previous term} \EE\left[\overline{d}_p^2(t\land T)\right]\le e^{Cbt}\left(\overline{d}_p^2(0)+Cta_2K^{1-1/\nu}\right).
\end{equation}  Next, we observe that \begin{equation} \overline{d}_p^2(t)\le \overline{d}_p^2(t\land T)+\overline{d}_p^2(t)1[T\le t].\end{equation} We now estimate the second term. From the bound $d_p^2(v,w)\le c(1+|v|^{p+2}+|w|^{p+2})$ we see that \begin{equation} \overline{d}_p^2(t)\le c\left(\Lambda_{p+2}(\mu^N_t)+\Lambda_{p+2}(\widetilde{\mu}^{N,K}_t)\right) \end{equation} We use H\"older's inequality with indexes $\frac{q}{p+2}$ and $\frac{q}{p+2+\gamma}\le 2$, to obtain \begin{equation}\begin{split} \label{eq: T small term}\EE\left[\overline{d}_p^2(t)1[T\le t]\right]&\le c\h \PP(T\le t)^{(p+2+\gamma)/q}\h\EE\left[\Lambda_q(\mu^N_t)+\Lambda_q(\widetilde{\mu}^{N,K}_t)\right]^{(p+2)/q} \\ & \le C\h\PP(T\le t)^{1/2} a_3\end{split}\end{equation} thanks to the moment bounds in Lemma \ref{thrm:noncutoffmomentinequalities} and the choice of initial data. Combining with the previous term (\ref{eq: the previous term}) now proves the first claim. \bigskip \\ For the second item, we return to the martingale $\overline{M}_t$ constructed above. From  \cite[Lemma 8.7]{darling2008differential}, the process \begin{equation}
    L_t=\overline{M}_t^2-\frac{2}{N^3}\sum_{\{ij\}}\int_0^t \mathcal{Q}_K(V^i_s, \widetilde{V}^{i,K}_s,V^j_s,\widetilde{V}^{j,K}_s)ds
\end{equation} is also a martingale, where the sum now runs over unordered pairs $\{ij\}$ of indexes. Thanks to the bound computed in Lemma \ref{lemma: quadratic pozvner}, we find \begin{equation} \begin{split}
    \EE\left[\overline{M}_{t_\mathrm{fin}}^2\right]& \le \frac{C}{N^3}\EE\left[\sum_{\{ij\}}\int_0^{t_\mathrm{fin}}(1+|V^i_s|^{q}+|V^j_s|^{q}+|\widetilde{V}^{i,K}_s|^{q}+|\widetilde{V}^{j,K}_s|^{q})ds\right]\\ & \hs \hs  \le \frac{C}{N}\int_0^{t_\mathrm{fin}}\EE(\Lambda_{2l}(\mu^N_s)+\Lambda_{2l}(\widetilde{\mu}^{N,K}_s))ds. \end{split}
\end{equation} Using the moment propagation estimate in Propositions  \ref{thrm:momentinequalities}, \ref{thrm:noncutoffmomentinequalities} and Doob's $L^2$ inequality, we conclude that \begin{equation}
    \EE\left[\sup_{s\le t_\mathrm{fin}} |\overline{M}_t|\right] \le \frac{Ca_3t_\mathrm{fin}^{1/2}}{N^{1/2}}.
\end{equation} With this estimate, we return to the argument above. Applying Gr\"onwall to (\ref{eq: gronwall dbar}), we obtain a pathwise estimate  \begin{equation}
    \begin{split}
        \sup_{t\le t_\mathrm{fin}}\h\overline{d}_p^2(t\land T)\le e^{Cbt}\left(\overline{d}_p^2(0)+K^{1-1/\nu}\int_0^{t_\mathrm{fin}} \left(\Lambda_l(\mu^N_s)+\Lambda_l(\widetilde{\mu}^{N,K}_s)\right)ds+\sup_{t\le t_\mathrm{fin}}|\overline{M}_t|\right).
    \end{split}
\end{equation} Taking expectations, we conclude that \begin{equation}\label{eq: LU stopped term}
    \EE\left[\sup_{t\le t_\mathrm{fin}}\h\overline{d}_p^2(t\land T)\right]\le e^{Cbt}\left(\overline{d}_p^2(0)+K^{1-1/\nu}Ct_\mathrm{fin}a_2+\frac{Ca_3t_\mathrm{fin}^{1/2}}{N^{1/2}}\right).
\end{equation}Following the same argument as in (\ref{eq: T small term}) we also bound \begin{equation}\begin{split}\label{eq: LU T small term}
    \EE\left[\left(\sup_{t\le t_\mathrm{fin}}\h\overline{d}_p^2(t)\right)1[T\le t_\mathrm{fin}]\right]&\le C\PP\left(T\le t_\mathrm{fin}\right)^{1/2}\EE\left[\sup_{t\le t_\mathrm{fin}}(\Lambda_q(\mu^N_t)+\Lambda_q(\widetilde{\mu}^{N,K}_t))\right]\\[1ex] & \le C\PP(T\le t_\mathrm{fin})^{1/2}t_\mathrm{fin}a_3.\end{split}
\end{equation} Combining (\ref{eq: LU stopped term}, \ref{eq: LU T small term}), we obtain the desired result.  \end{proof}

\section{Proof of Theorem \ref{thrm: wellposedness}}\label{sec: T2} We will now prove Theorem \ref{thrm: wellposedness}, based on the Tanaka coupling presented in Lemma \ref{lemma: Tanaka Coupling of KP} and Corollary \ref{cor: couple cutoff}. The proof is broken down into a series of lemmas; in order to give an overview of the strategy, we will state all the intermediate steps before turning to the proofs. Our first result transfers the coupling achieved in Corollary \ref{cor: couple cutoff} to solutions to the cutoff Boltzmann equation, potentially with different cutoff parameters and different initial data. \begin{lemma}\label{lemma: T1.1} Let $p>p_0(G,d)$ and $l=p+2+\gamma$. Then there exist a constant $C=C(G,p,d)$ such that, whenever $K'\ge K>K_0(G,p,d), a_1, a_2\ge 1$ and $\mu_0, \nu_0 \in \mathcal{S}$ satisfy moment bounds \begin{equation}
    \Lambda_{p+\gamma}(\mu_0, \nu_0)\le a_1; \hs \Lambda_l(\mu_0, \nu_0)\le a_2
\end{equation} then the solution maps $\phi^K_t$ to the cutoff Boltzmann equation (\ref{eq: BEK}) satisfy, for all $t\ge 0,$ \begin{equation}   W_p\left(\phi^K_t(\mu_0),\phi^{K'}_t(\nu_0)\right)\le e^{Ca_1(1+t)}\left(W_p\left(\mu_0, \nu_0\right)+a_2tK^{1/2-1/2\nu}\right).\end{equation}    \end{lemma} 

As a next step, we show that the solutions $\phi^K_t(\mu_0)$ to the cutoff Boltzmann equations converge, as $K\rightarrow \infty$, to \emph{a} solution of the {noncutoff} equation (\ref{BE}).

\begin{lemma}\label{lemma: T1.2} Let $p, l$ be as above, and let $\mu_0\in \mathcal{S}$ satisfy moment assumptions \begin{equation} \Lambda_{p+\gamma}(\mu_0)\le a_1, \hs \Lambda_{l}(\mu_0)\le a_2\end{equation} for some $a_1, a_2\ge 1$. Then, for some $(\phi_t(\mu))_{t\ge 0}\subset \mathcal{S}$ and some $C=C(G,p,d)$, \begin{equation} \label{eq: convergence of cutoff BE} W_p(\phi^K_t(\mu), \phi_t(\mu))\le e^{Ca_1(1+t)/2}ta_2K^{1/2-1/2\nu} \end{equation} for all $K>K_0(G,p,d)$.  Moreover, if $\nu_0 \in \mathcal{S}$ is another measure with the same moment estimates, we have the continuity \begin{equation}\label{eq: continuity of special solution} W_p\left(\phi_t(\mu_0),\phi_t(\nu_0)\right)\le e^{Ca_1(1+t)}W_p(\mu_0, \nu_0).\end{equation}   Finally, $(\phi_t(\mu_0): t\ge 0)$ is a solution to the noncutoff Boltzmann equation (\ref{BE}), and satisfies the moment estimates in Proposition \ref{thrm:noncutoffmomentinequalities}.   \end{lemma} 

We next extend the maps $\phi_t$ defined above to all of $\mathcal{S}^{p+1}$, and obtain the claimed continuity estimate in this context

\begin{lemma}\label{lemma: T1.3} Let $p,l$ be as above. The solution maps $\phi_t: \mathcal{S}^l\rightarrow \mathcal{S}$ defined above can be extended to $\cup_{p'>p+2}\mathcal{S}^{p'}$, such that, for all $\mu_0\in \cup_{p'>p+2}S^{p'}$, $(\phi_t(\mu_0): t\ge 0)$ is a solution to the Boltzmann Equation (\ref{BE}), and so that (\ref{eq: continuity of special solution}) holds whenever $\mu_0, \nu_0 \in \cup_{p'>p+2}\mathcal{S}^{p'}$ satisfy a moment estimate $\Lambda_{p+\gamma}(\mu_0,\nu_0)\le a$, for some $a\ge 1$.\end{lemma} To conclude Theorem \ref{thrm: wellposedness}, we must show that the solutions obtained in this way are the unique solutions to (\ref{BE}) as soon as $\mu_0$ has $p'$ moments, for any $p'>p+2$. We will use the following auxiliary result, which appears as \cite[Corollary 2.3iii)]{fournier2009well}. \begin{proposition}\label{prop: weak wellposedness} Suppose $\mu_0\in \mathcal{S}$ satisfies, for some $\epsilon>0$, $\langle e^{\epsilon|v|^\gamma}, \mu_0\rangle<\infty$. Then there exists at most one solution to the Boltzmann Equation (\ref{BE}) taking values in $\mathcal{S}$ and starting at at $\mu_0$. \end{proposition} Let us now show how these results imply the claimed result. \begin{proof}[Proof of Theorem \ref{thrm: wellposedness}] In light of Lemma \ref{lemma: T1.3} above, it remains only to prove that the solutions constructed above are unique. Let us fix $p'>p+2$ and $\mu_0\in \mathcal{S}^{p'}$ and such that $\Lambda_{p+2}(\mu_0)\le a$ for some $a\ge 1$. Let $(\mu_t)_{t\ge 0}\subset \mathcal{S}$ be any solution to (\ref{BE}) starting at $\mu_0$; we will now show that $\mu_t=\phi_t(\mu_0)$ for all $t\ge 0$. \bigskip \\ Fix $s>0, t\ge 0$. Thanks to the appearance of exponential moments in Proposition \ref{thrm:noncutoffmomentinequalities}, there exists $\epsilon=\epsilon_s>0$ such that $\langle e^{\epsilon |v|^\gamma}, \mu_s\rangle<\infty$, and by Proposition \ref{prop: weak wellposedness}, there exists at most one energy-conserving solution starting at $\mu_s$. Since both $(\phi_u(\mu_s))_{u\ge 0}$ and $(\mu_{u+s})_{u\ge 0}$ are such solutions, we conclude that $\phi_t(\mu_s)=\mu_{t+s}$ for all such $t,s$. \bigskip \\ Let us now take the limit $s\downarrow 0$. Using (\ref{BE}) and the duality (\ref{eq: dual W1}), it is straightforward to see that $w_1(\mu_s, \mu_0)\rightarrow 0$. By (\ref{eq: elementary interpolation}),  \begin{equation}\label{eq: W continuity at 0}
    W_p(\mu_{s}, \mu_0)\le w_1(\mu_{s},\mu_0)^{\alpha}\Lambda_{p'}(\mu_{s},\mu_0)
\end{equation} for some $\alpha=\alpha(p')>0$; the moment factor is bounded uniformly in $s>0$ by Lemma \ref{thrm:noncutoffmomentinequalities}, and so the right-hand side converges to $0$ as $s\downarrow 0$. Lemma \ref{lemma: T1.3} now shows that, up to a new choice of $C$, \begin{equation} W_p(\phi_t(\mu_s), \phi_t(\mu_0))\le e^{Ca(1+t)}W_p(\mu_s, \mu_0)\rightarrow 0.\end{equation} Using the same argument as (\ref{eq: W continuity at 0}), $W_p(\mu_{t+s},\mu_t)\rightarrow 0$, and we conclude that \begin{equation} W_p\left(\mu_t,\phi_t(\mu_0)\right)\le \limsup_{s\downarrow 0}\left[CW_p\left(\phi_{t}(\mu_s), \phi_t(\mu_0)\right)+CW_p\left(\mu_{t+s}, \mu_t\right)\right] =0 \end{equation} and so we have the desired uniqueness.\end{proof} 

\subsection{Proof of Lemmas} \label{subsec: actual proof of T1} \begin{proof}[Proof of Lemma \ref{lemma: T1.1}] Let us consider the case first where the initial data $\mu_0, \nu_0$ have a finite $q^\text{th}$ moment $\Lambda_q(\mu_0, \nu_0)\le a_3$ for some $a_3\ge 1$. Applying Proposition \ref{prop: discrete scheme}, take $N$-particle empirical measures $\mu^N_0 \in \mathcal{S}_N$ such that $W_p(\mu^N_0, \mu_0)\rightarrow 0$ and such that the $l^\text{th}, q^\text{th}$ moments converge: $\Lambda_l(\mu^N_0)\rightarrow \Lambda_l(\mu_0), \Lambda_q(\mu^N_0)\rightarrow \Lambda_q(\mu_0)$; construct $\nu^N_0$ similarly for $\nu_0$. Using the relaxed triangle inequality, it follows that, for some $C=C(p)$, \begin{equation}\label{eq: limsup distance of initial data}
    \limsup_{N\rightarrow \infty} W_p(\mu^N_0, \nu^N_0)\le C\h W_p(\mu_0, \nu_0).
\end{equation} Let us now take $\mu^{N,K}_t$, $\nu^{N,K'}_t$ be the cutoff Kac processes constructed in Corollary \ref{cor: couple cutoff} started at these initial data; fix $t\ge 0$, and consider \begin{equation}
    \begin{split}
        W_p^2\left(\phi^K_t(\mu_0),\phi^{K'}_t(\nu_0)\right)&\le C\h\EE\bigg[W_p^2\left(\phi^K_t(\mu_0),\mu^N_t\right)+W_p^2\left(\mu^N_t,\nu^{N,K'}_t\right)\\& \hs\hs  +W_p^2\left(\nu^{N,K'}_t,\phi^{K'}_t(\nu_0)\right)\bigg].
    \end{split}
\end{equation} Using Corollary \ref{cor: couple cutoff} to bound the middle term, we have \begin{equation}
    \begin{split}
        W_p^2\left(\phi^K_t(\mu_0),\phi^{K'}_t(\nu_0)\right)&\le C\h\EE\left[W_p^2\left(\phi^K_t(\mu_0),\mu^N_t\right)+W_p^2\left(\nu^{N,K'}_t,\phi^{K'}_t(\nu_0)\right)\right] \\[1ex] &\hs +e^{Ca_1(1+t)}\left(W_p^2\left(\mu^N_0, \nu^N_0\right) +a_2K^{1-1/\nu}\right)\\&\hs +a_3^2CN^{-1/2}.
    \end{split}
\end{equation}Let us now take $N\rightarrow \infty$. Thanks to Lemma \ref{lemma: convergence of cutoff KP}, both terms on the first line converge to $0$, as does the final term. Using (\ref{eq: limsup distance of initial data}), we conclude that \begin{equation}
    W_p^2\left(\phi^K_t(\mu_0),\phi^{K'}_t(\nu_0)\right)\le e^{Ca_1(1+t)}\left(W_p^2\left(\mu^N_0, \nu^N_0\right)+a_2K^{1-1/\nu}\right)
\end{equation} and taking the square root gives the desired result. \bigskip \\ Let us now show how this extends to initial data $\mu_0, \nu_0$ with only $l=p+\gamma+2$ moments as in the statement. In this case, we use Proposition \ref{prop: discrete scheme} again, with $l$ in place of $q$, to construct $\mu^N_0\in \mathcal{S}_N$ such that \begin{equation} W_p(\mu^N_0, \mu_0)\rightarrow 0,\hs  \Lambda_l(\mu^N_0)\rightarrow \Lambda_l(\mu_0)\end{equation} and similarly $\nu^N_0$ for $\nu_0.$ Since $\mu^N_0, \nu^N_0$ are compactly supported, the previous estimate applies so that \begin{equation} \label{eq: T1.1 for approximate initial data}
    W_p\left(\phi^K_t(\mu^N_0),\phi^{K'}_t(\nu^N_0)\right)\le e^{Ca_1(1+t)}\left(W_p\left(\mu^N_0, \nu^N_0\right)+a_2K^{1/2-1/2\nu}\right).
\end{equation} Using Lemma \ref{lemma: W-W stability for cutoff}, \begin{equation}
    W_p\left(\phi^K_t(\mu^N_0),\phi^K_t(\mu_0)\right)\rightarrow 0; \hs  W_p\left(\phi^{K'}_t(\nu^N_0),\phi^{K'}_t(\nu_0)\right)\rightarrow 0.
\end{equation} The same argument as above therefore allows us to take $N\rightarrow \infty$ in (\ref{eq: T1.1 for approximate initial data}), noting that no moments higher than $l^\text{th}$ appear, to conclude that  \begin{equation} \label{eq: T1.1 for nonapproximate initial data}
    W_p^2\left(\phi^K_t(\mu_0),\phi^{K'}_t(\nu_0)\right)\le Ce^{Ca_1(1+t)}\left(W_p\left(\mu_0, \nu_0\right)+a_2K^{1/2-1/2\nu}\right).
\end{equation}    \end{proof} \begin{proof}[Proof of Lemma \ref{lemma: T1.2}] Let us fix $\mu_0$ and consider the space $\mathcal{C}=C([0,\infty), (\mathcal{S},w_1))$, equipped with a metric inducing uniform convergence on compact time intervals; since $(\mathcal{S}, w_1)$ is complete, so is $\mathcal{C}$. Recalling that $w_1\le W_p$, the previous observation shows that $(\phi^K(\mu_0), t\ge 0)_{K\ge 1}$ are Cauchy in $\mathcal{C}$, and hence converge to some process $(\phi_t(\mu_0), t\ge 0)$. \bigskip \\ Next, let us show that $\phi^K_t(\mu_0)\rightarrow \phi_t(\mu_0)$ in $W_p$. For $t=0$, $\phi^K_0(\mu_0)=\mu_0$, and so there is nothing to prove. If $t>0$ then, by point iii) of Proposition \ref{thrm:momentinequalities}, there exists $\lambda_{p+3}=\lambda_{p+3}(t)<\infty$ such that, for all $K\ge 1,$ \begin{equation} \Lambda_{p+3}\left(\phi^K_t(\mu_0)\right)\le \lambda_{p+3}(t).\end{equation} By lower semicontinuity of $\mu\mapsto \Lambda_{p+3}(\mu)$ in $w_1$, the same is true for the limit $\phi_t(\mu_0)$, and using the estimates in Section \ref{sec: metrics},\begin{equation}
    W_p\left(\phi^K_t(\mu_0),\phi_t(\mu_0)\right)\le \Lambda_{p+3}\left(\phi^K_t(\mu_0),\phi_t(\mu_0)\right)\h w_1\left(\phi^K_t(\mu_0),\phi_t(\mu_0)\right)^{\alpha}
\end{equation}for some $\alpha>0$. By construction, the second term on the right-hand side converges to $0$, and the first term is bounded, so the left-hand side converges to $0$ as desired. We now conclude the bound (\ref{eq: convergence of cutoff BE}): if $K>K_0(G,p,d)$, then for all $K'\ge K$, \begin{equation}
    \begin{split}
        W_p\left(\phi^K_t(\mu_0),\phi_t(\mu_0)\right)&\le C\left( W_p\left(\phi^K_t(\mu_0),\phi^{K'}_t(\mu_0)\right)+  W_p\left(\phi^{K'}_t(\mu_0),\phi_t(\mu_0)\right)\right) \\[1ex] & \le C\left(e^{Ca_1(1+t)/2}a_2tK^{1/2-1/2\nu}+W_p\left(\phi^{K'}_t(\mu_0),\phi_t(\mu_0)\right)\right).
    \end{split}
\end{equation}Taking $K'\rightarrow \infty$, the second term on the final line converges to $0$, and the desired bound follows, absorbing the prefactor $C=C(p)$ into the exponent. The bound (\ref{eq: continuity of special solution}) is similar: if $\mu_0, \nu_0$ in $\mathcal{S}^l$ satisfy \begin{equation} \Lambda_{p+\gamma}(\mu_0, \nu_0)\le a_1;\hs \Lambda_l(\mu_0,\nu_0)\le a_2\end{equation} then we bound, for any $K$,\begin{equation} \begin{split} W_p\left(\phi_t(\mu_0),\phi_t(\nu_0)\right)&\le C\big(W_p\left(\phi^K_t(\mu_0), \phi_t(\mu_0)\right)+W_p\left(\phi^K_t(\mu_0), \phi^K_t(\nu_0)\right) \\&\hs +W_p\left(\phi^K_t(\nu_0), \phi_t(\nu_0)\big) \right)\\[1ex]& \le Ce^{Ca_1(1+t)}\left(3a_2tK^{1/2-1/2\nu}+W_p(\mu_0,\nu_0)\right) \end{split} \end{equation} where, in the second line, we have used Lemma \ref{lemma: T1.1} to compare $\phi^K_t(\mu_0), \phi^K_t(\nu_0)$ and used the previous part to estimate the other two terms. Taking $K\rightarrow \infty$, we conclude the desired bound, again up to a new choice of $C$. \bigskip \\ It remains to show that $\phi_t(\mu_0), t\ge 0$ solves the full, noncutoff Boltzmann equation (\ref{BE}). We begin with an analysis of the the Boltzmann collision operator, borrowing from \cite{lu2012measure}. Let us define, for bounded, Lipschitz $f:\RRd\rightarrow \mathbb{R}$, \begin{equation}
    (\Delta_Bf)(v,v_\star):=\int_{\mathbb{S}^{d-1}}(f(v')+f(v'_\star)-f(v)-f(v_\star))B(v-v_\star,d\sigma);
\end{equation}
\begin{equation}
  (\Delta_{B_K}f)(v,v_\star):=\int_{\mathbb{S}^{d-1}}(f(v')+f(v'_\star)-f(v)-f(v_\star))B(v-v_\star,d\sigma)
\end{equation} and observe that \begin{equation}
    \langle f, Q(\mu)\rangle=\int_{\RRd\times\RRd} (\Delta_B f)(v,v_\star)\mu(dv)\mu(dv_\star)
\end{equation}and similarly for $B_K$. It is straightforward to see that each $\Delta_{B_K}f$ is continuous on $\RRd\times\RRd$, and the straightforward estimate $ |v'-v|\le |v-v_\star|\sin\theta$ implies the growth bound \begin{equation} \left|(\Delta_B f)(v,v_\star)\right|\le C(f)|v-v_\star|\int_{\mathbb{S}^{d-1}} \sin \theta B(v-v_\star,\sigma)d\sigma \le C(f)|v-v_\star|^{1+\gamma}  \end{equation} for some constant $C=C(f)$, depending only on the Lipschitz constant of $f$, and similarly for $B_K$. The same argument also shows that \begin{equation}\label{eq: Q vs cutoff Q}
    \left|(\Delta_B f)(v,v_\star)-(\Delta_{B_K}f)(v,v_\star)\right|\le C(f)\epsilon_K|v-v_\star|^{1+\gamma};
\end{equation}\begin{equation}\epsilon_K=\int_{\mathbb{S}^{d-1}}(\sin \theta)1(\theta\le \theta_0(K)) B(u, d\sigma)\rightarrow 0
\end{equation} so that $\Delta_{B_K}f\rightarrow \Delta f$, uniformly on compact subsets of $\RRd\times\RRd$; it therefore follows that $\Delta_B f$ is continuous. \bigskip \\ Equipped with this dual formulation, let us fix $t\ge 0$ and a bounded, Lipschitz $f$. Writing $\mu^K_t:=\phi^K_t(\mu_0), \mu_t:=\phi_t(\mu_0)$, we claim that \begin{equation}
    \langle f, Q_K(\phi^K_t(\mu_0))\rangle \rightarrow \langle f, Q(\phi_t(\mu_0))\rangle.
\end{equation} For all $R\ge 0$, let $\psi_R:\RRd\times\RRd\rightarrow [0,1]$ be a smooth, compactly supported cutoff function, such that $\psi_R(v,v_\star)=1$ on the ball $\{|v|^2+|v_\star|^2\le R\}$. We estimate, uniformly in $K$, \begin{equation}\begin{split}\label{eq: tails of Q}
    &\int_{\RRd\times\RRd} |(\Delta_{B}f)(v,v_\star)(1-\psi_R(v,v_\star))|\mu^K_t(dv)\mu^K_t(dv_\star) \\ & \hs \hs \le C(f)\int_{\RRd\times\RRd}(1+|v|^2+|v_\star|^2)1(|v|^2+|v_\star|^2\ge R)\mu^K_t(dv)\mu^K_t(dv_\star)\\[1ex] & \hs \hs \le C(f)R^{-p}a_2\h\Lambda_l\left(\mu_0\right)\end{split}
\end{equation} where, in the final line, we used the moment hypothesis on $\mu_0$, with $l>p+2$, and the moment propagation result in Proposition \ref{thrm:momentinequalities}; the same argument holds for the limit $\mu_t$. It is elementary to show that the Wasserstein convergence $w_1(\mu^K_t,\mu_t)\rightarrow 0$ implies that, for all compactly supported, continuous $g:\RRd\times\RRd\rightarrow \mathbb{R}$, we have \begin{equation}\label{eq: wasserstein tensor}
    \int_{\RRd\times\RRd}g(v,v_\star)\left(\mu^K_t(dv)\mu^K_t(dv_\star)-\mu_t(dv)\mu_t(\mu_0)(dv_\star)\right)\rightarrow 0
\end{equation} and, in particular, this holds with $g=(\Delta_B f)(v,v_\star)\psi_R$. We now write \begin{equation}
    \begin{split}
       \left| \left\langle f, Q(\mu_t)-Q(\mu^K_t)\right\rangle\right| & \\& \hspace{-2cm} \le \int_{\RRd\times\RRd}|(\Delta_Bf)|(1-\psi_R)(v,v_\star)(\mu^K_t(dv)\mu^K_t(dv_\star)+\mu_t(dv)\mu_t(dv_\star))\\ &   \hspace{-2cm} + \left|\int_{\RRd\times\RRd }(\Delta_B f)\psi_R(v,v_\star)(\mu^K_t(dv)\mu^K_t(dv_\star)-\mu_t(dv)\mu_t(dv_\star))\right|.
    \end{split}
\end{equation} The second term converges to $0$ by (\ref{eq: wasserstein tensor}), so using (\ref{eq: tails of Q}) twice on the first term, \begin{equation}
    \begin{split}
        \limsup_{K\rightarrow \infty}\left|\left\langle f, Q(\mu_t)-Q(\mu^K_t)\right\rangle\right|\le CR^{-p}a_2\Lambda_l(\mu_0)
    \end{split}
\end{equation} and, since $R$ was arbitrary, we have shown that \begin{equation}
    \label{eq: convergence of Q 1} \left\langle f, Q(\mu^K_t)\right\rangle\rightarrow \left\langle f, Q(\mu_t)\right\rangle.
\end{equation}Finally, integrating (\ref{eq: Q vs cutoff Q}), we find \begin{equation}\begin{split}
    \left|\langle f, Q(\mu^K_t)-Q_K(\mu^K_t)\rangle\right| & \le \int_{\RRd\times\RRd}|(\Delta_Bf)-(\Delta_{B_K}f)|(v,v_\star)\mu^K_t(dv)\mu^K_t(dv_\star)\\&  \le C(f)\epsilon_K\int_{\RRd\times\RRd}(1+|v|^2+|v_\star|^2)\mu^K_t(dv)\mu^K_t(dv_\star) \\ & \le C(f)\epsilon_K\rightarrow 0 \end{split}
\end{equation} and, combining with (\ref{eq: convergence of Q 1}), we see that $\langle f, Q(\mu^K_t)\rangle\rightarrow \langle f, Q(\mu_t)\rangle$ as claimed. \bigskip\\ We now conclude. For any $t\ge 0$ and any bounded, Lipschitz $f$, we have \begin{equation}
    \langle f, \mu^K_t\rangle =\langle f, \mu_0\rangle+ \int_0^t\langle f, Q_K(\mu^K_s)\rangle ds.
\end{equation} The integrand $\langle f, Q_K(\mu^K_s)\rangle$ is bounded, uniformly in $s\le t$ and $K\ge 1$, and converges to $\langle f, Q(\mu_s)\rangle$ for all $s$, while the left-hand side converges to $\langle f, \mu_t\rangle$. We therefore take the limit $K\rightarrow \infty$ to conclude that, for all bounded, Lipschitz $f$ and all $t\ge 0$ \begin{equation}
    \langle f,\mu_t\rangle=\langle f, \mu_0\rangle+\int_0^t \langle f, Q(\mu_s)\rangle ds
\end{equation} as desired. \end{proof} Finally, Lemma \ref{lemma: T1.3} follows much the same pattern as above. \begin{proof}[Proof of Lemma \ref{lemma: T1.3}] To extend the maps $\phi_t$, fix $p'>p+2$ and $\mu_0\in \mathcal{S}^{p'}$. Using Proposition \ref{prop: discrete scheme} again, let $\mu^N_0 \in \mathcal{S}_N$ be a sequence of discrete measures such that $W_p(\mu^N_0, \mu_0)\rightarrow 0$ and $\Lambda_{p+\gamma}(\mu^N_0)\rightarrow \Lambda_{p+\gamma}(\mu_0)$; in particular, $\Lambda_{p+\gamma}(\mu^N_0)\le 2a$ for all $N$ large enough. The bound (\ref{eq: continuity of special solution}) obtained in the previous lemma applies to show that, for all such $N$ and all $t\ge 0$,\begin{equation} w_1(\phi_t(\mu^N_0),\phi_t(\mu^{N'}_0))\le W_p\left(\phi_t(\mu^N_0),\phi_t(\mu^{N'}_0)\right)\le e^{2Ca(1+t)}W_p(\mu^N_0, \mu^{N'}_0). \end{equation} The right-hand side converges to $0$ as $N, N'\rightarrow \infty$, which implies that $\phi_t(\mu^N_0)$ converges, uniformly in $w_1$ on compact time intervals, to some limit. If we now define $\phi_t(\mu_0)$ to be this limit, a similar calculation shows that the resulting $\phi_t(\mu_0)$ is independent of the choice of limiting sequence, and the same argument as in Lemma \ref{lemma: T1.2} above shows that $(\phi_t(\mu_0), t\ge 0)$ is again a solution to the noncutoff Boltzmann equation (\ref{BE}). Finally, if $\mu_0, \nu_0$ are two such measures, one applies (\ref{eq: continuity of special solution}) to approximating sequences $\mu^N_0,\nu^N_0$ and passes to the limit $N\rightarrow \infty$ to obtain the same result for $\mu_0, \nu_0$, again up to a new constant $C$ in the exponent. \end{proof}

\section{Proof of Theorem \ref{thrm: Kac Process}}\label{sec: proof_of_poc}
We now prove the Theorem \ref{thrm: Kac Process} concerning the convergence of the full, non-cutoff Kac process to the solution to the Boltzmann equation in the many-particle limit $N\rightarrow \infty$.
\begin{proof}[Proof of Theorem \ref{thrm: Kac Process}] The uniqueness in law follows from Propositions \ref{prop: labelled and unlabelled dynamics}, \ref{prop: wellposedness/uniqueness}, which are discussed in Appendix \ref{sec: pf_wellposedness}. \bigskip \\ For the convergence estimate, let $\mu^N_t, t\ge 0$ be any unlabelled Kac process, and consider the case $\mu_0=\mu^N_0$. Fix $t_\mathrm{fin}$ and $K$ to be chosen later; for this $K$, let $\widetilde{\mu}^N_t, \widetilde{\mu}^{N,K}_t$ be the coupling of noncutoff and cutoff Kac processes, both starting at $\mu^N_0$ given Theorem \ref{thrm: approximate with cutoff}. By uniqueness in law, it is sufficient to prove the estimate with $\widetilde{\mu}^N_t$ in place of $\mu^N_t$. For some constants $C=C(p,q), \alpha=\alpha(p,q)$, we have the following estimates. By Theorem \ref{thrm: approximate with cutoff},\begin{equation} \EE\left[\sup_{t\le t_\mathrm{fin}}W_p\left(\widetilde{\mu}^N_t,\widetilde{\mu}^{N,K}_t\right)\right] \le e^{Ca(1+t_\mathrm{fin})}(K^{1/2-1/2\nu}+N^{-1/4});\end{equation} by Lemma \ref{lemma: convergence of cutoff KP}\begin{equation} \EE\left[\sup_{t\le t_\mathrm{fin}}W_p\left(\widetilde{\mu}^{N,K}_t,\phi^K_t(\mu^N_0)\right)\right]\le \exp\left(CaK(1+t_\mathrm{fin})\right)\h N^{-\alpha}\end{equation} by Lemma \ref{lemma: convergence of cutoff KP}, and by Lemma \ref{lemma: T1.2},\begin{equation} \sup_{t\le t_\mathrm{fin}}\h W_p\left(\phi^K_t(\mu^N_0),\phi_t(\mu^N_0)\right)\le e^{Ca(1+t_\mathrm{fin})}K^{1/2-1/2\nu}. \end{equation} Combining, and keeping the worst terms, we have the estimate \begin{equation} \EE\left[\sup_{t\le t_\mathrm{fin}}\h W_p\left(\widetilde{\mu}^N_t,\phi_t(\mu^N_0)\right)\right]\le e^{Ca(1+t_\mathrm{fin})}K^{1/2-1/2\nu}+e^{CaK(1+t_\mathrm{fin})}\h N^{-\alpha}.\end{equation} We now choose \begin{equation} K=\max\left(1,\frac{1}{2Ca(1+t_\mathrm{fin})}\log(N^{\alpha})\right)\end{equation} to conclude that \begin{equation}\EE\left[\sup_{t\le t_\mathrm{fin}}\h W_p\left(\widetilde{\mu}^N_t,\phi_t(\mu^N_0)\right)\right] \le e^{Ca(1+t_\mathrm{fin})}\h (\log N)^{1/2-1/2\nu}. \end{equation}  Finally, by Theorem \ref{thrm: wellposedness}, we have \begin{equation} \EE\left[\sup_{t\le t_\mathrm{fin}}\h W_p\left(\phi_t(\mu^N_0),\phi_t(\mu_0)\right)\right]\le e^{Ca(1+t_\mathrm{fin})}\h \EE\left[W_p(\mu^N_0, \mu_0)\right]\end{equation} and combining gives the claimed bound. \end{proof}

\section{Main Calculations on the Tanaka Coupling} \label{sec: main calculations}
\subsection{Some Estimates for $G$}\label{subsec: estimates for G}
In preparation for the proofs of Lemma \ref{lemma: big calculation}, we will first record some basic estimates concerning the regularity and integrability of $G$.
\begin{lemma}\label{lemma: estimates for G} \begin{enumerate}[label=\roman{*}.)]\item Let $G$ be as above. Then, for some constants $0<c_1\le c_2<\infty$, we have \begin{equation} c_1(1+z)^{-1/\nu}\le G(z)\le c_2(1+z)^{-1/\nu}.\end{equation} Moreover, $G$ is continuously differentiable, and $c_1, c_2$ above can be chosen such that \begin{equation} \hs c_1(1+z)^{-1-1/\nu}\le G'(z)\le c_2(1+z)^{-1-1/\nu}.  \end{equation} \item We have \begin{equation} \int_0^\infty z\left|\frac{d}{dz}\left(1-\cos G(z))\right)\right|\h dz < \infty. \end{equation} \item There exists a constant $c<\infty$ such that, for all $x,y>0$,  \begin{equation} \label{eq: L2 difference of Gs} \int_0^\infty \left(G\left(\frac{z}{x}\right)-G\left(\frac{z}{y}\right)\right)^2 \h dz \le c\frac{|x-y|^2}{x+y}.\end{equation} \end{enumerate} \end{lemma}  \begin{proof} \begin{enumerate}[label=\roman{*}).]\item  For the first claim, we use the definition of $H$ and (\ref{eq: form of B}) to see that, for some constants $c_1, c_2\in (0,\infty)$ and all $\theta\in (0,\pi/2)$, \begin{equation} c\int_\theta^{\pi/2} x^{-1-\nu}dx \le H(\theta) \le C\int_\theta^{\pi/2} x^{-1-\nu}dx \end{equation} so that \begin{equation}\frac{c_1}{\nu}\left(\theta^{-\nu}-\left(\frac{\pi}{2}\right)^{-\nu}\right) \le H(\theta) \le \frac{c_2}{\nu}\left(\theta^{-\nu}-\left(\frac{\pi}{2}\right)^{-\nu}\right). \end{equation} The first claim now follows, potentially for a new choice of $c_1, c_2$. The differentiability is an immediate consequence of the inverse function theorem. Indeed, we have \begin{equation} G'(z)=\frac{1}{H'(G(z))}=-\frac{1}{b(\cos G(z))} \end{equation} and so the second claim follows from the first, using (\ref{eq: form of B}).    \item We have\begin{equation} z\frac{d}{dz}((1-\cos G(z)))=z\h (\sin G(z))\h G'(z) \end{equation} and so \begin{equation} \left|z\frac{d}{dz}(1-\cos G(z))\right|\le zG(z)|G'(z)|. \end{equation} Using the bounds from the previous part, it follows that the right-hand side is bounded by $c_2(1+z)^{-2/\nu}$ for some $c_2<\infty$, which is integrable because $\nu \in (0,1)$. \item The following is a slight variant of \cite[Lemma 1.1]{fournier2008uniqueness}, and is included here for completeness. Recalling that $G$ is decreasing, and integrating the bound on $G'$ found in part i)., we see that, for all $0\le z\le w$ and some $c<\infty$, we have \begin{equation}0\le G(z)-G(w)\le c\left((1+z)^{-1/\nu}-(1+w)^{-1/\nu}\right). \end{equation} We also recall that, for all $a>b>0$, we have \begin{equation}a^{1/\nu}-b^{1/\nu}\le c\frac{a-b}{a^{1-1/\nu}+b^{1-1/\nu}}. \end{equation} For any $z>0$, $0<y<x$, we apply this bound with $a=(1+z/x)^{-1}, b=(1+z/y)^{-1}$ to obtain \begin{equation}\begin{split}\label{eq: square me} 0\le  G\left(\frac{z}{x}\right)-G\left(\frac{z}{y}\right) &\le c\left((1+z/x)^{-1/\nu}-(1+z/y)^{-1/\nu})\right) \\& \le c\left|\frac{x}{x+z}-\frac{y}{y+z}\right|\left(1+\frac{z}{x}\right)^{1-1/\nu} \\[1ex] & \le c|x-y|(x+z)^{-1/\nu}x^{-1+1/\nu}.\end{split}\end{equation}We square and integrate over $z$, to obtain for all $x>y>0$, \begin{equation} \begin{split} \int_0^\infty \left(G\left(\frac{z}{x}\right)-G\left(\frac{z}{y}\right)\right)^2dz & \le c|x-y|^2 x^{1-2/\nu}x^{-2+2/\nu} \\ & =c\frac{|x-y|^2}{x}\le c\frac{|x-y|^2}{x+y}.\end{split} \end{equation} This concludes the proof of both claimed bounds in the case $x>y>0$; for $y>x$, we reverse the roles of $x\leftrightarrow y$. \end{enumerate} \end{proof}
\subsection{Proof of Lemma \ref{lemma: big calculation}} 
 
We now turn to the proof of Lemma \ref{lemma: big calculation}, which was deferred earlier. In order to avoid unnecessarily unwieldy expressions, we introduce some notation. We define $x=|v-v_\star|, \widetilde{x}=|\widetilde{v}-\widetilde{v}_\star|$, and write $L$ for the cuttoff $L=K\widetilde{x}^\gamma$. We will also write $R$ for $R(v-v_\star,\widetilde{v}-\widetilde{v}_\star)$, and suppress the dependence of $a,\widetilde{a}_K, \mathcal{E}_K$ on $v,\widetilde{v},v_\star,\widetilde{v}_\star$.  Throughout, $c$ will denote a constant which is allowed to depend only on $G, d$, and $C$ will denote a constant which is also allowed to depend on $p$; both are understood to vary from line to line as necessary. \bigskip \\ Our first lemma is the following, which gives us control over the `Pozvner' term, similar to the estimates in Section \ref{sec: moments}. Since this estimate produces the key negative term in Lemma \ref{lemma: big calculation} and is essential for subsequent calculations, it is presented as a separate lemma. \begin{lemma}\label{lemma: pozvner term}  For all $v, v_\star, z$, we have the bound \begin{equation} \label{eq: fp} \begin{split} |v+a|^p &\le |v|^p\left(\frac{1+\cos G(z/x^\gamma)}{2}\right)^{p/2}+|v_\star|^p\left(\frac{\sin G(z/x^\gamma)}{2}\right)^{p/2} \\& \hspace{3cm}+ C \left(|v|^{p-1}|v_\star|+|v||v_\star|^{p-1}\right)\sin G(z/x^\gamma) \\[1ex] &=:f_p(|v|,|v_\star|,z,x).\end{split} \end{equation} \end{lemma} \begin{proof} Let us start from \begin{equation} v+a=v\left(\frac{1+\cos G(z/x^\gamma)}{2}\right)+v_\star\left(\frac{1-\cos G(z/x^\gamma)}{2}\right)+\frac{\sin G(z/x^\gamma)}{2}\Gamma(v-v_\star,\varphi).\end{equation} We now take the norm of both sides, recalling that $|\Gamma(v-v_\star,\varphi)|=|v-v_\star|$: \begin{equation}\begin{split} |v+a|^2&=\left(\frac{1+\cos G(z/x^\gamma)}{2}\right)^2|v|^2 +\left(\frac{1-\cos G(z/x^\gamma)}{2}\right)^2|v_\star|^2 \\& \hs \hs +\left(\frac{\sin G(z/x^\gamma)}{2}\right)^2(|v|^2+|v_\star|^2+2|v||v_\star|) \\& \hs \hs +\left(\frac{\sin G(z/x^\gamma)}{2}\right)\left(\frac{1+\cos G(z/x^\gamma)}{2}\right)v\cdot \Gamma(v-v_\star,\varphi)\\& \hs \hs +\left(\frac{\sin G(z/x^\gamma)}{2}\right)\left(\frac{1-\cos G(z/x^\gamma)}{2}\right)v_\star\cdot \Gamma(v-v_\star,\varphi)\\& \hs \hs +\left(\frac{1-\cos G(z/x^\gamma)}{2}\right)\left(\frac{1+\cos G(z/x^\gamma)}{2}\right)v\cdot v_\star.\end{split}\end{equation} For the third and fourth lines, we use orthogonality to see that $v\cdot \Gamma(v-v_\star,\varphi)=v_\star\cdot \Gamma(v-v_\star,\varphi)$. It follows that \begin{equation}\left|v\cdot \Gamma(v-v_\star,\varphi)\right|\le \min(|v|,|v_\star|)(|v|+|v_\star|)\le 2|v||v_\star|.\end{equation} Using the inequality $1-\cos G(z)\le \sin G(z)$, we now group similar terms to obtain \begin{equation} \begin{split} |v+a|^2 & \le \left(\frac{1+\cos G(z/x^\gamma)}{2}\right)|v|^2+\left(\frac{1-\cos G(z/x^\gamma)}{2}\right)|v_\star|^2+C\sin G(z/x^\gamma)|v||v_\star|\\ & :=h_1+h_2+h_3. \end{split}\end{equation} We now raise both sides to the $(p/2)^\text{th}$ power, recalling the inequality $(x+y)^{p/2}\le x^{p/2}+y^{p/2}+C(xy^{p/2-1}+x^{p/2-1}y)$, valid for all $x, y>0$. It is straightforward to see that the cross terms are dominated by the final term in (\ref{eq: fp}): \begin{equation} h_1^{p/2-1}(h_2+h_3)+h_1(h_2+h_3)^{p/2-1}\le C(|v|^{p-1}|v_\star|+|v||v_\star|^{p-1})\sin G(z/x^\gamma); \end{equation}  \begin{equation} h_2^{p/2-1}h_3+h_2h_3^{p/2-1}\le C(|v|^{p-1}|v_\star|+|v||v_\star|^{p-1})\sin G(z/x^\gamma); \end{equation}
\begin{equation} h_3^{p/2}\le C(|v|^{p-1}|v_\star|+|v||v_\star|^{p-1})\sin G(z/x^\gamma). \end{equation} Using this inequality twice, we thus obtain \begin{equation} \begin{split}|v+a|^p &\le h_1^{p/2}+(h_2+h_3)^{p/2}+ C(|v|^{p-1}|v_\star|+|v||v_\star|^{p-1})\sin G(z/x^\gamma) \\ & \le h_1^{p/2}+h_2^{p/2}+h_3^{p/2}+  C(|v|^{p-1}|v_\star|+|v||v_\star|^{p-1})\sin G(z/x^\gamma)  \\ & \le h_1^{p/2}+h_2^{p/2}+  C(|v|^{p-1}|v_\star|+|v||v_\star|^{p-1})\sin G(z/x^\gamma) \end{split} \end{equation}  which gives the desired bound on substituting the definitions of $h_1, h_2$. \end{proof}  We now break up $\mathcal{E}_K$ as follows. We define \begin{equation} \mathcal{E}^1_K=\int_0^\infty dz \int_{\SSd} d\varphi \left(|v'|^p|v'-\widetilde{v}'_K|^2-|v|^p|v-\widetilde{v}|^2\right); \end{equation} \begin{equation} \begin{split}  &\mathcal{E}^2_K(v,\widetilde{v}, v_\star,\widetilde{v}_\star)=\int_0^\infty dz \int_{\SSd} d\varphi \left(|\widetilde{v}'_K|^p|v'-\widetilde{v}'_K|^2-|\widetilde{v}|^p|v-\widetilde{v}|^2\right)\end{split};\end{equation}\begin{equation} \mathcal{E}^3_K(v,\widetilde{v},v_\star,\widetilde{v}_\star)=\int_0^\infty dz \int_{\SSd}d\varphi \h (|v'-\widetilde{v}'_K|^2-|v-\widetilde{v}|^2). \end{equation} In this way, using the definition of $d_p$, it follows that $\mathcal{E}_K=\mathcal{E}^1_K+\mathcal{E}^2_K+\mathcal{E}^3_K$. It therefore suffices to prove the following estimates. \begin{lemma}\label{lemma: broken up big calculation}For some constants $K_0=K_0(p), c=c(G,d)$ and $C=C(G,d,p)$, and $q=p+2+\gamma$, whenever $K\ge K_0(p)$, we have \begin{equation} \begin{split} \label{eq: E K conclusion} \mathcal{E}_K(v,\widetilde{v},v_\star,\widetilde{v}_\star)& \le \left(c+\left(c-\frac{\lambda_p}{2}\right)|v|^{p+\gamma}+c|\widetilde{v}|^{p+\gamma}\right)|v-\widetilde{v}|^2\\&\hs +\left(c|v_\star|^{p+\gamma}+c|\widetilde{v}_\star|^{p+\gamma}\right)|v_\star-\widetilde{v}_\star|^2 \\&\hs +C\left(|v_\star|^{p+\gamma}+|\widetilde{v}_\star|^{p+\gamma}\right)\left(1+|v|^p+|\widetilde{v}|^p\right)|v-\widetilde{v}|^2\\&\hs +C\left(|v|^{p+\gamma}+|\widetilde{v}|^{p+\gamma}\right)\left(1+|v_\star|^p+|\widetilde{v}_\star|^p\right)|v_\star-\widetilde{v}_\star|^2\\ & \hs +CK^{1-1/\nu}(1+|v|^l+|v_\star|^l+|\widetilde{v}|^l+|\widetilde{v}_\star|^l);\end{split} \end{equation} \begin{equation} \label{eq: widetilde E K} \begin{split} \mathcal{E}^2_K(v,\widetilde{v},v_\star,\widetilde{v}_\star)& \le \left(c+\left(c-\frac{\lambda_p}{2}\right)|\widetilde{v}|^{p+\gamma}+c|{v}|^{p+\gamma}\right)|v-\widetilde{v}|^2\\&\hs +\left(c|v_\star|^{p+\gamma}+c|\widetilde{v}_\star|^{p+\gamma}\right)|v_\star-\widetilde{v}_\star|^2 \\&\hs +C\left(|v_\star|^{p+\gamma}+|\widetilde{v}_\star|^{p+\gamma}\right)\left(1+|v|^p+|\widetilde{v}|^p\right)|v-\widetilde{v}|^2\\&\hs +C\left(|v|^{p+\gamma}+|\widetilde{v}|^{p+\gamma}\right)\left(1+|v_\star|^p+|\widetilde{v}_\star|^p\right)|v_\star-\widetilde{v}_\star|^2\\ & \hs +CK^{1-1/\nu}(1+|v|^l+|v_\star|^l+|\widetilde{v}|^l+|\widetilde{v}_\star|^l) \end{split} \end{equation} and \begin{equation} \begin{split} \mathcal{E}^3_K(v,\widetilde{v},v_\star,\widetilde{v}_\star) & \le c(1+|v|^{\gamma}+|\widetilde{v}|^\gamma+|v_\star|^\gamma+|\widetilde{v}_\star|^\gamma)(|v-\widetilde{v}|^2+|v_\star-\widetilde{v}_\star|^2) \\ & \hs +CK^{1-1/\nu}(1+|v|^l+|\widetilde{v}|^l+|v_\star|^l+|\widetilde{v}_\star|^l). \end{split}  \end{equation} \end{lemma} 
 \begin{proof}[Proof of Lemmas \ref{lemma: big calculation}, \ref{lemma: broken up big calculation}] Let us begin from the bound (\ref{eq: fp}), and define also \begin{equation} f_p^\star(|v|,|v_\star|)= |v|^p+C(|v||v_\star|^{p-1}+|v|^{p-1}|v_\star|)+|v_\star|^p\end{equation} which is an upper bound for $f_p$, uniformly in $z, x$. We therefore find \begin{equation}\begin{split}\mathcal{E}_K^1\le \int_0^\infty dz\int_{\SSd} d\varphi \h (f_p(|v|,|v_\star|,z,x)|v-\widetilde{v}+a-\widetilde{a}|^2-|v|^p|v-\widetilde{v}|^2). \end{split} \end{equation} Let us also introduce \begin{equation} \widehat{a}=a(\widetilde{v},\widetilde{v}_\star,z,R(v-v_\star,\widetilde{v}-\widetilde{v}_\star)\varphi) \end{equation} so that $\widetilde{a}=\widehat{a}1(z\le L)$. We can therefore replace $\widetilde{a}$ by $\widehat{a}$, introducing a further error: \begin{equation}\begin{split} \mathcal{E}^1_K&\le \int_0^\infty dz \int_{\SSd}d\varphi \h (f_p(|v|,|v_\star|,z,x)|v-\widetilde{v}+a-\widehat{a}|^2-|v|^p|v-\widetilde{v}|^2)\\& \hs + \int_L^\infty dz\int_{\SSd}d\varphi \h f_p(|v|,|v_\star|,z,x)\left(|v-\widetilde{v}+a|^2-|v-\widetilde{v}+a-\widehat{a}|^2\right).\end{split}\end{equation} Finally, we expand the squared norm $|v-\widetilde{v}+a-\widehat{a}|^2$ in the first line to obtain the decomposition  \begin{equation} \label{eq: decomposition for E} \begin{split} \mathcal{E}^1_K &\le \T_1 + \T_2+\T_3+\T_4\end{split} \end{equation} where we define \begin{equation}
     \T_1:=\int_0^\infty dz \int_{\SSd}d\varphi \h \left(f_p(|v|,|v_\star|,z,x)-|v|^p\right) |v-\widetilde{v}|^2;
 \end{equation} \begin{equation}
\T_2:=2\int_0^\infty dz\int_{\SSd} d\varphi \h f_p(|v|,|v_\star|,z,x)(v-\widetilde{v})\cdot(a-\widehat{a});
 \end{equation} \begin{equation} \T_3:= \int_0^\infty dz\int_{\SSd} d\varphi \h f^\star_p(|v|,|v_\star|)|a-\widehat{a}|^2;\end{equation} \begin{equation} \T_4:=
     \int_L^\infty dz\int_{\SSd} d\varphi \h f_p(|v|,|v_\star|,z,x) \left(|v+a-\widetilde{v}|^2-|v+a-\widetilde{v}-\widehat{a}|^2\right).
 \end{equation}We will now analyse this bound for $\mathcal{E}^1_K$ in detail, and an equivalent analysis of $\mathcal{E}^2_K, \mathcal{E}^3_K$ will be discussed at the end of the proof. Let us now deal with these terms one by one. 
 
 \paragraph{1. Analysis of $\T_1$. } Recalling the construction of $G$, the moment integral in $\T_1$ can be reparametrised in terms of $\theta$: \begin{equation} \begin{split} \label{eq: start of t1} &\hspace{-1cm}\int_0^\infty (f_p(|v|,|v_\star|,z,x)-|v|^p)dz \\& =-|v-v_\star|^\gamma |v|^p |\mathbb{S}^{d-2}|\int_0^{\pi/2} \beta(\theta)\left(1-\left(\frac{1+\cos \theta}{2}\right)^{p/2} \right) d\theta  \\ &\hs  +|v-v_\star|^\gamma (|v_\star|^p+C(|v||v_\star|^{p-1}+|v|^{p-1}|v_\star|)) |\mathbb{S}^{d-2}|\int_0^{\pi/2} \beta(\theta)\sin(\theta)d\theta\\[1ex]&\le -\lambda_p |v-v_\star|^\gamma |v|^p+ C|v-v_\star|^\gamma(|v_\star|^p+|v||v_\star|^{p-1}+|v|^{p-1}|v_\star|).\end{split} \end{equation} On the negative term, we use the bound $ |v|^\gamma-|v_\star|^\gamma \le |v-v_\star|^\gamma$ and Young's inequality to see that \begin{equation}\begin{split}-|v-v_\star|^\gamma|v|^p& \le -|v|^{p+\gamma} +  |v_\star|^\gamma|v|^p \\ & \le -|v|^{p+\gamma} + \frac{1}{4}|v|^{p+\gamma} +  4^{p/\gamma}|v_\star|^{p+\gamma}. \end{split}\end{equation} For the positive term in (\ref{eq: start of t1}), we use $|v-v_\star|^\gamma \le |v|^\gamma+|v_\star|^\gamma$ to obtain \begin{equation} \begin{split} |v-v_\star|^\gamma(|v_\star|^p+|v||v_\star|^{p-1}+|v|^{p-1}|v_\star|) &\le |v|^{\gamma}|v_\star|^p+|v|^{\gamma+1}|v_\star|^{p-1}+|v|^{p+\gamma-1}|v| \\ & +|v_\star|^{\gamma+p}+|v||v_\star|^{p+\gamma-1}+|v|^{p-1}|v_\star|^{p+\gamma}+|v|^{p-1}|v_\star|^{1+\gamma}.\end{split}  \end{equation} We now use Young's inequality on each term to obtain \begin{equation} \begin{split} C|v-v_\star|^\gamma(|v_\star|^p+|v||v_\star|^{p-1}+|v|^{p-1}|v_\star|) \le \frac{\lambda_p}{4}|v|^{p+\gamma}+C|v_\star|^{p+\gamma}\end{split} \end{equation} Combining, we have shown that \begin{equation} \int_0^\infty (f_p(|v|,|v_\star|,z,x)-|v|^p)dz \le -\frac{\lambda_p}{2}|v|^{p+\gamma}+C |v_\star|^{p+\gamma}\end{equation} and so \begin{equation}\label{eq: T1} \T_1 \le -\frac{\lambda_p}{2}|v|^{p+\gamma}|v-\widetilde{v}|^2+C |v_\star|^{p+\gamma}|v-\widetilde{v}|^2.\end{equation}
 
 \paragraph{2. Analysis of $\T_2$. } We first observe that  \begin{equation} \int_{\SSd}d\varphi \h (a-\widehat{a})=-\frac{1}{2}(1-\cos G(z/x^\gamma))(v-v_\star)+\frac{1}{2}(1-\cos G(z/\widetilde{x}^\gamma))(\widetilde{v}-\widetilde{v}_\star). \end{equation} It therefore follows that \begin{equation}\label{eq: first form of T2} \T_2=(v-\widetilde{v})\cdot\left\{\Phi(\widetilde{x},|v|,|v_\star|,x)(\widetilde{v}-\widetilde{v}_\star)-\Phi(x,|v|,|v_\star|,x)(v-v_\star)\right\}\end{equation} where we define, for any $y,u,v,w>0$, \begin{equation} \begin{split} \Phi(y,u,v,w)&=\int_0^\infty dz \h f_p(u,v,z,w)(1-\cos G(z/y^\gamma))dz \\ & :=\Psi(y^\gamma,u,v,w). \end{split} \end{equation} We differentiate the function thus defined to obtain \begin{equation} \label{eq: differentiate psi}\begin{split} \frac{\partial}{\partial y} \Psi(y,u,v,w) &=\int_0^{\infty}f_p(u,v,z,w)\left(-\frac{z}{y}\right)\frac{d}{d z}\left(1-\cos G\left(\frac{z}{y}\right)\right)dz \\&=\int_0^\infty f_p(u,v,yz,w)\left(z\frac{d}{dz}(1-\cos G(z))\right)dz \end{split} \end{equation} where the final line follows by an integration by substitution $z\mapsto yz$. From the calculations in Lemma \ref{lemma: estimates for G}, we therefore conclude that \begin{equation}\left|\frac{\partial}{\partial y} \Psi(y,u,v,w) \right|\le cf_p^\star(u,v).  \end{equation} Now, using the bound $|x^\gamma-y^\gamma|\le 2|x-y|/(x^{1-\gamma}+y^{1-\gamma})$, we obtain \begin{equation} \label{eq: control of PhiL}
  \left|\Phi(x,|v|,|v_\star|, x)-\Phi(\widetilde{x},|v|,|v_\star|,x)\right| \le \frac{c|x-\widetilde{x}|}{x^{1-\gamma}+\widetilde{x}^{1-\gamma}}f_p^\star(|v|,|v_\star|)\end{equation} and, for all $y>0$, \begin{equation} |\Phi(x, |v|, |v_\star|,y)|\le cy^\gamma f_p^\star(|v|, |v_\star|).\end{equation} We therefore obtain the bound \begin{equation}\label{eq: T2} \begin{split} |\T_2|&\le |v-\widetilde{v}|\bigg\{|v-v_\star-\widetilde{v}+\widetilde{v}_\star||\Phi(x,|v|,|v_\star|,x)+\Phi(\widetilde{x},|v|,|v_\star|,x)|\\ & \hs +(|v-v_\star|+|\widetilde{v}-\widetilde{v}_\star|)\left|\Phi(x,|v|,|v_\star|, x)-\Phi(\widetilde{x},|v|,|v_\star|,x)\right|\bigg\} \\[1ex] &\le c\left(|v-\widetilde{v}|^2+|v_\star-\widetilde{v}_\star|^2\right)(|v|^\gamma+|v_\star|^\gamma+|\widetilde{v}|^\gamma +|\widetilde{v}_\star|^\gamma)f_p^\star(|v|,|v_\star).\end{split} \end{equation} 
\paragraph{3. Analysis of $\T_3$.} We now turn to the term $\T_3$, and begin by noting that \begin{equation}\begin{split}\label{eq: a dot atilde} a\cdot \widehat{a}&=\frac{1}{4}\left(1-\cos G(z/x^\gamma)\right)\left(1-\cos G(z/\widetilde{x}^\gamma)\right)(v-v_\star)\cdot(\widetilde{v}-\widetilde{v}_\star) \\ &-\frac{1}{4} \left(1-\cos G(z/x^\gamma)\right)\sin G(z/\widetilde{x}^\gamma)(v-v_\star)\cdot \Gamma(\widetilde{v}-\widetilde{v}_\star, R\varphi) \\ &-\frac{1}{4} \left(1-\cos G(z/\widetilde{x}^\gamma)\right)\sin G(z/{x}^\gamma)\Gamma(v-v_\star,\varphi)\cdot(\widetilde{v}-\widetilde{v}_\star)\\& +\frac{1}{4}\sin G(z/{x}^\gamma)\sin G(z/\widetilde{x}^\gamma)\Gamma(v-v_\star,\varphi)\cdot\Gamma(\widetilde{v}-\widetilde{v}_\star, R\varphi). \end{split}\end{equation} We now integrate over $\varphi \in \SSd$. Since $\int_{\SSd}\Gamma(u,\varphi)d\varphi=0$ and $R$ preserves the uniform measure $d\varphi$, the middle two lines integrate to $0$. We also recall, from the construction of $R=R(v-v_\star, \widetilde{v}-\widetilde{v}_\star)$ in Lemma \ref{lemma: ATT}, that $\Gamma(v-v_\star,\varphi)\cdot \Gamma(\widetilde{v}-\widetilde{v}_\star,R\varphi)\ge (v-v_\star)\cdot (\widetilde{v}-\widetilde{v}_\star)$, and so integrating (\ref{eq: a dot atilde}) gives \begin{equation} \begin{split} \label{eq: integral a dot atilde}  \int_{\SSd} a\cdot \widehat{a}\h d\varphi & \ge \frac{1}{4}\bigg[\left(1-\cos G(z/x^\gamma)\right)\left(1-\cos G(z/\widetilde{x}^\gamma)\right)\\ & \hs \hs +\sin G(z/{x}^\gamma)\sin G(z/\widetilde{x}^\gamma)\bigg](v-v_\star)\cdot (\widetilde{v}-\widetilde{v}_\star) \\ &=\frac{1}{4}\bigg[(1-\cos G(z/x^\gamma))+(1-\cos G(z/\widetilde{x}^\gamma))\\& \hs \hs -\bigg(1-\cos\big(G(z/x^\gamma)-G(z/\widetilde{x}^\gamma)\big)\bigg)\bigg](v-v_\star)\cdot (\widetilde{v}-\widetilde{v}_\star).\end{split} \end{equation} Similar, elementary calculations show that \begin{equation}\label{eq: norm of a atilde} |a|^2=\frac{1}{2}(1-\cos G(z/x^\gamma))|v-v_\star|^2;\hs |\widehat{a}|^2=\frac{1}{2}(1-\cos G(z/\widetilde{x}^\gamma))|\widetilde{v}-\widetilde{v}_\star|^2. \end{equation} We now observe that\begin{equation} \int_0^\infty (1-\cos G(z/x^\gamma))dz=cx^\gamma \end{equation}and so, from (\ref{eq: integral a dot atilde}, \ref{eq: norm of a atilde}), we obtain \begin{equation} \begin{split} \label{eq: integral of norm a-atilde squared}\int_0^\infty dz\int_{\SSd}d\varphi \h |a-\widehat{a}|^2 & \le \frac{c}{2}\left(x^{2+\gamma}+\widetilde{x}^{2+\gamma}-(v-v_\star)\cdot(\widetilde{v}-\widetilde{v}_\star)(x^\gamma+\widetilde{x}^\gamma)\right) \\ &\hs +\frac{x\widetilde{x}}{4}\int_0^\infty \left(G(z/x^\gamma)-G(z/\widetilde{x}^\gamma)\right)^2dz. \end{split} \end{equation} Recalling that $x^2=(v-v_\star)\cdot (v-v_\star)$ and $\widetilde{x}^2=(\widetilde{v}-\widetilde{v}_\star)\cdot(\widetilde{v}-\widetilde{v}_\star)$, the term in parentheses on the first line rearranges to \begin{equation}\begin{split} &x^{2+\gamma}+\widetilde{x}^{2+\gamma} -(v-v_\star)\cdot(\widetilde{v}-\widetilde{v}_\star)(x^\gamma+\widetilde{x}^\gamma) \\[1ex] & \hs = (v-v_\star)\cdot\left[(v-v_\star)-(\widetilde{v}-\widetilde{v}_\star)\right]x^\gamma \\ & \hs \hs + (\widetilde{v}-\widetilde{v}_\star)\cdot\left[(\widetilde{v}-\widetilde{v}_\star)-(v-v_\star)\right]\widetilde{x}^\gamma\\[1ex]& \hs = \left((v-\widetilde{v})-(v_\star-\widetilde{v}_\star)\right)\cdot\left[(v-v_\star)x^\gamma-(\widetilde{v}-\widetilde{v}_\star)\widetilde{x}^\gamma\right].\end{split} \end{equation} The same estimates as in (\ref{eq: control of PhiL}) now give \begin{equation} \label{eq: bound on t31}\begin{split} &x^{2+\gamma}+\widetilde{x}^{2+\gamma} -(v-v_\star)\cdot(\widetilde{v}-\widetilde{v}_\star)(x^\gamma+\widetilde{x}^\gamma)\\ & \hs \hs \le c\left(|v-\widetilde{v}|^2+|v_\star-\widetilde{v}_\star|^2\right)(|v|^\gamma+|v_\star|^\gamma+|\widetilde{v}|^\gamma+|\widetilde{v}_\star|^\gamma).\end{split} \end{equation} Let us now consider the final line of (\ref{eq: integral of norm a-atilde squared}). By Lemma \ref{lemma: estimates for G}, we have the bound\begin{equation} \int_0^\infty \left(G(z/x^\gamma)-G(z/\widetilde{x}^\gamma)\right)^2dz\le c\frac{|x^\gamma-\widetilde{x}^\gamma|^2}{x^\gamma+\widetilde{x}^\gamma}.\end{equation} We therefore obtain\begin{equation} \label{eq: bound on t32} \begin{split} &x\widetilde{x}\int_0^\infty \left(G(z/x^\gamma)-G(z/\widetilde{x}^\gamma)\right)^2 dz  \le c\frac{\min(x,\widetilde{x})}{\max(x,\widetilde{x})^{1-\gamma}}|x-\widetilde{x}|^2 \\[1ex] & \hs \hs \le c(|v|^\gamma+|v_\star|^\gamma+|\widetilde{v}|^\gamma+|\widetilde{v}_\star|^\gamma)(|v-\widetilde{v}|^2+|v_\star-\widetilde{v}_\star|^2).\end{split} \end{equation} Combining (\ref{eq: integral of norm a-atilde squared}, \ref{eq: bound on t31}, \ref{eq: bound on t32}), we have shown that \begin{equation} \label{eq: T3}  \T_3 \le c(|v|^\gamma+|v_\star|^\gamma+|\widetilde{v}|^\gamma+|\widetilde{v}_\star|^\gamma)(|v-\widetilde{v}|^2+|v_\star-\widetilde{v}_\star|^2)f_p^\star(|v|,|v_\star). \end{equation} \paragraph{4. Analysis of $\T_4$.} The final error term is the term $\mathcal{T}_4$, which corresponds to collisions in the noncutoff system with no corresponding event in the cutoff system. As a result, we anticipate that $\T_4$ will not be bounded in terms of $v-\widetilde{v}, v_\star-\widetilde{v}_\star$, but will be small in the limit $K\rightarrow \infty$. Let us recall that the integration limit $L$ is defined as $L:=K\widetilde{x}^\gamma.$  By expanding out the norms, we bound the integrand, for $z\ge L$, \begin{equation}\begin{split}\left|f_p(|v|,|v_\star|,z,x)(|v+a-\widetilde{v}|^2-|v+a-\widetilde{v}-\widehat{a}|^2)\right| \le cf_p^\star(|v|,|v_\star|)|\widehat{a}|(|v|+|\widetilde{v}|+|v_\star|+|\widetilde{v}_\star|).\end{split}\end{equation} As above, we have \begin{equation} |\widehat{a}|=\sqrt{\frac{1}{2}\left(1-\cos G\left(\frac{z}{\widetilde{x}^\gamma}\right)\right)}|\widetilde{v}-\widetilde{v}_\star| \le \frac{1}{2}G\left(\frac{z}{\widetilde{x}^\gamma}\right)|\widetilde{v}-\widetilde{v}_\star|.\end{equation} We therefore obtain the bound \begin{equation} \T_4 \le cf_p^\star(|v|,|v_\star|)(|v|^2+|v_\star|^2+|\widetilde{v}|^2+|\widetilde{v}_\star|^2)\int_L^\infty G\left(\frac{z}{\widetilde{x}^\gamma}\right)dz.\end{equation}Recalling the definition of $L=K\widetilde{x}^\gamma$, the integral evalues to \begin{equation} \int_L^\infty G\left(\frac{z}{\widetilde{x}^\gamma}\right)dz=\widetilde{x}^\gamma\int_K^\infty G(z)dz \le c\widetilde{x}^\gamma K^{1-1/\nu}.\end{equation} We therefore find \begin{equation} \label{eq: T4}\T_4 \le CK^{1-1/\nu}(|v|^{p+2+\gamma}+|\widetilde{v}|^{p+2+\gamma}+|v_\star|^{p+2+\gamma}+|\widetilde{v}_\star|^{p+2+\gamma})\end{equation}  Recalling that $l:=p+2+\gamma$, this is exactly the error claimed.  
\paragraph{5. Converting into the form desired. } Combining (\ref{eq: T1}, \ref{eq: T2}, \ref{eq: T3}, \ref{eq: T4}), we see that \begin{equation} \begin{split}\label{eq: recombine} \mathcal{E}^1_K&\le \left((c-\frac{\lambda_p}{2})|v|^{p+\gamma}+C|v_\star|^{p+\gamma}\right)|v-\widetilde{v}|^2 \\ & \hs +c(|v|^\gamma+|\widetilde{v}|^\gamma+|v_\star|^\gamma+|\widetilde{v}_\star|^\gamma)f_p^\star(|v|,|v_\star|)(|v-\widetilde{v}|^2+|v_\star-\widetilde{v}_\star|^2) \\[1ex] & \hs +C(|v|^l+|\widetilde{v}|^l+|v_\star|^l+|\widetilde{v}_\star|^l)K^{1-1/\nu}. \end{split}\end{equation} The first and last lines are already in the form desired in the statement of the lemma. Let us now examine the middle term. Using Young on the cross-terms in $f_p^\star$, we see that \begin{equation} f^\star_p(|v|,|v_\star|)\le 2|v|^p+C|v_\star|^p\end{equation} and so \begin{equation} \begin{split}&\hspace{-1cm}(|v|^\gamma+|v_\star|^\gamma+|\widetilde{v}|^\gamma+|\widetilde{v}_\star|^\gamma)f^\star_p(|v|,|v_\star|) \\[1ex]&\le c(|v|^{p+\gamma}+|v|^p|v_\star|^\gamma+|v|^p|\widetilde{v}|^{\gamma}+|v|^p|\widetilde{v}_\star|^{\gamma}) \\ & \hs +C (|v_\star|^{p+\gamma}+|v_\star|^p|v|^\gamma+|v_\star|^p|\widetilde{v}|^\gamma+|v_\star|^p|\widetilde{v}_\star|^{\gamma})\end{split}\end{equation} We now use Young's inequality on all terms appearing in this expression; for the second term, we use Peter-Paul to find \begin{equation} C (|v_\star|^{p+\gamma}+|v_\star|^p|v|^\gamma+|v_\star|^p|\widetilde{v}|^\gamma+|v_\star|^p|\widetilde{v}_\star|^{\gamma}) \le c|v|^{p+\gamma}+c|v_\star|^{p+\gamma}+C(|v_\star|^{p+\gamma}+|\widetilde{v}_\star|^{p+\gamma}).\end{equation} Therefore, \begin{equation} \begin{split}  &(|v|^\gamma+|v_\star|^\gamma+|\widetilde{v}|^\gamma+|\widetilde{v}_\star|^\gamma)f_p^\star(|v|,|v_\star|) \\& \hspace{3cm}\le c|v|^{p+\gamma}+c|\widetilde{v}|^{p+\gamma}+C(1+|v_\star|^{p+\gamma}+|\widetilde{v}_\star|^{p+\gamma})(1+|v|^p+|\widetilde{v}|^p).\end{split} \end{equation} We use this inequality for the term multiplying $|v-\widetilde{v}|^2$ in the second line of (\ref{eq: recombine}), and reverse the roles of $v\leftrightarrow v_\star, \widetilde{v}\leftrightarrow \widetilde{v}_\star$ for the term involving $|v_\star-\widetilde{v}_\star|^2$. Together, we see that \begin{equation}\label{eq: final form of T2} \begin{split} &(|v|^\gamma+|\widetilde{v}|^\gamma+|v_\star|^\gamma+|\widetilde{v}_\star|^\gamma)f_p^\star(|v|,|v_\star|)(|v-\widetilde{v}|^2+|v_\star-\widetilde{v}_\star|^2) \\ & \hs \hs \hs \le c(|v|^{p+\gamma}+|\widetilde{v}|^{p+\gamma})|v-\widetilde{v}|^2 +c(|v_\star|^{p+\gamma}+|\widetilde{v}_\star|^{p+\gamma})|v_\star-\widetilde{v}_\star|^2\\ & \hs \hs \hs +C (1+|{v}_\star|^{p+\gamma}+|\widetilde{v}_\star|^{p+\gamma})(1+|v|^p+|\widetilde{v}|^p)|v-\widetilde{v}|^2 \\ & \hs \hs \hs  + C (1+|{v}|^{p+\gamma}+|\widetilde{v}|^{p+\gamma})(1+|v_\star|^p+|\widetilde{v}_\star|^p)|v_\star-\widetilde{v}_\star|^2 \end{split} \end{equation} which gives the bound desired for $\mathcal{E}^1_K$. 
\paragraph{6. Estimate on $\mathcal{E}^2_K$. } We now turn to the analysis of ${\mathcal{E}}^2_K$, which follows a similar pattern to $\mathcal{E}^1_K$ above. In this case, we use the bound \begin{equation}\label{eq: fpk} |\widetilde{v}+\widetilde{a}_K|^p\le f_{p,L}(|\widetilde{v}|,|\widetilde{v}_\star|,z,\widetilde{x})=\begin{cases} f_p(|\widetilde{v}|,|\widetilde{v}_\star|,z,\widetilde{x}), & z\le L; \\ |\widetilde{v}|^p, &z>L \end{cases}\end{equation} which has the same upper bound $f^\star_p$. We therefore obtain a decomposition equivalent to (\ref{eq: decomposition for E}): \begin{equation}\begin{split} {\mathcal{E}}^2_K  &\le\widetilde{\T}_1+ \widetilde{\T}_2  +\widetilde{\T}_3+\widetilde{\T}_4\end{split} \end{equation} where \begin{equation}
     \widetilde{\T}_1:=\int_0^\infty dz \h(f_{p,L}(|\widetilde{v}|,|\widetilde{v}_\star|,z,x)-|\widetilde{v}|^p)|v-\widetilde{v}|^2;
\end{equation} \begin{equation}
    \widetilde{\T}_2:= 2\int_0^L dz \int_{\SSd}d\varphi \h f_{p,L}(|\widetilde{v}|,|\widetilde{v}_\star|,z,x)(v-\widetilde{v})\cdot(a-\widehat{a});
\end{equation}\begin{equation}
    \widetilde{\T}_3:=\int_0^\infty dz \int_{\SSd}d\varphi \h  f^\star_p(|\widetilde{v}|,|\widetilde{v}_\star|)|a-\widehat{a}|^2;
\end{equation} \begin{equation}
    \widetilde{\T}_4:= \int_L^\infty dz \int_{\SSd}d\varphi \h f_{p,L}(|\widetilde{v}|,|\widetilde{v}_\star|,z,\widetilde{x})\left|2(v-\widetilde{v})\cdot \widehat{a}+|\widehat{a}|^2\right|
\end{equation}The analyses of $\widetilde{\T}_3, \widetilde{\T}_4$ are identical to the arguments above, and we will now discuss the necessary modifications for $\widetilde{\T}_1, \widetilde{\T}_2$. \paragraph{6a. Analysis of $\widetilde{\T}_1$. } Let us begin with $\widetilde{\T}_1$. The same reparametrisation gives \begin{equation} \begin{split} \int_0^\infty (f_{p,L}(|\widetilde{v}|,|\widetilde{v}_\star|z,\widetilde{x})-|\widetilde{v}|^p)dz&\le -|\widetilde{v}-\widetilde{v}_\star|^\gamma |\widetilde{v}|^p |\mathbb{S}^{d-2}|\int_{\theta_0(K)}^{\pi/2}\left(1-\left(\frac{1+\cos \theta}{2}\right)^{p/2}\right)\beta(\theta)d\theta \\ & +C|\widetilde{v}-\widetilde{v}_\star|^\gamma (|\widetilde{v}_\star|^p+|\widetilde{v}|^{p-1}|\widetilde{v}_\star|+|\widetilde{v}||\widetilde{v}_\star|^{p-1}).\end{split} \end{equation} We therefore obtain \begin{equation} \begin{split} & \int_0^\infty (f_{p,L}(|\widetilde{v}|,|\widetilde{v}_\star|z,\widetilde{x})-|\widetilde{v}|^p)dz\\ & \hs \hs \le -|\widetilde{v}-\widetilde{v}_\star|^\gamma |\widetilde{v}|^p \lambda_{p,K}  +|\widetilde{v}-\widetilde{v}_\star|^\gamma (|\widetilde{v}_\star|^p+C(|\widetilde{v}|^{p-1}|\widetilde{v}_\star|+|\widetilde{v}||\widetilde{v}_\star|^{p-1})) \\[1ex] & \hs \hs \le -\lambda_{p,K}|\widetilde{v}|^{p+\gamma}+\lambda_p|\widetilde{v}_\star|^\gamma|\widetilde{v}|^p+C|\widetilde{v}-\widetilde{v}_\star|^\gamma (|\widetilde{v}_\star|^p+|\widetilde{v}|^{p-1}|\widetilde{v}_\star|+|\widetilde{v}||\widetilde{v}_\star|^{p-1})\end{split} \end{equation} where\begin{equation} \lambda_{p,K}:=|\mathbb{S}^{d-2}|\int_{\theta_0(K)}^{\pi/2} \left(1-\left(\frac{1+\cos \theta}{2}\right)^{p/2}\right)\beta(\theta)d\theta \le \lambda_p.\end{equation} We now use Peter-Paul on the positive terms, independently of $K$, to obtain \begin{equation} \lambda_p|\widetilde{v}_\star|^\gamma|\widetilde{v}|^p+C|\widetilde{v}-\widetilde{v}_\star|^\gamma (|\widetilde{v}_\star|^p+|\widetilde{v}|^{p-1}|\widetilde{v}_\star|+|\widetilde{v}||\widetilde{v}_\star|^{p-1}) \le \frac{\lambda_p}{3}|\widetilde{v}|^{p+\gamma}+C|\widetilde{v}_\star|^{p+\gamma}.\end{equation} By monotone convergence, $\lambda_{p,K}\rightarrow \lambda_p$ as $K\rightarrow \infty$ with $p$ fixed; in particular, for some $K_0=K_0(G,p,d)$ and all $K\ge K_0(G,p,d)$, $\lambda_{p,K}\ge \frac{5}{6}\lambda_p$. For such $K$, we have shown that \begin{equation}\widetilde{\T}_1 \le -\frac{\lambda_p}{2}|\widetilde{v}|^{p+\gamma}|v-\widetilde{v}|^2 +C|\widetilde{v}_\star|^{p+\gamma}|v-\widetilde{v}|^2. \end{equation} \paragraph{6b. Analysis of $\widetilde{\T}_2$.} Following the same manipulations as (\ref{eq: first form of T2}), we obtain \begin{equation} \widetilde{\T}_2=(v-\widetilde{v})\cdot\left\{(\Psi_{0L}+\Psi_{L\infty})(\widetilde{x}^\gamma,|\widetilde{v}|,|\widetilde{v}_\star|,\widetilde{x})(\widetilde{v}-\widetilde{v}_\star)-(\Psi_{0L}+\Psi_{L\infty})(\widetilde{x}^\gamma,|\widetilde{v}|,|\widetilde{v}_\star|,\widetilde{x})(v-v_\star)\right\} \end{equation} where we define \begin{equation} \Psi_{0L}(y,u,v,w)=\int_0^L f_p(u,v,z,w)(1-\cos G(z/y))dz \end{equation} and \begin{equation}\Psi_{L\infty}(y,u,v,w)=\int_L^\infty v^p(1-\cos G(z/y))dz.  \end{equation} One then repeats the differentiation (\ref{eq: differentiate psi}) for each part separately, to obtain a bound \begin{equation} \left|\frac{\partial}{\partial y} \Psi_{0L}(y,u,v,w)\right|+\left|\frac{\partial}{\partial y} \Psi_{L\infty}(y,u,v,w)\right| \le cf_p^\star(u,v) \end{equation} and the rest of the argument follows as for $\T_2.$ \paragraph{7. Bound on $\mathcal{E}^3_K$.} Finally, let us mention $\mathcal{E}^3_K$. This term is strictly easier than the two above: there is no term analagous to $\T_1$, and one can omit the moment prefactors in the remaining terms. Alternatively, one may note that $\mathcal{E}^3_K$ is exactly that analysed in \cite[Lemma 3.1]{fournier2016rate}, and the claimed bound is exactly the content of \cite[Lemma 5.1]{fournier2016rate}. \end{proof} 
\subsection{Proof of Lemma \ref{lemma: quadratic pozvner}} We now turn to the proof of the quadratic bound Lemma \ref{lemma: quadratic pozvner}, where we replace the integrand of $\mathcal{E}_K$ with its square. In this case, the integrand is nonnegative, and there is no hope of exploiting cancellations in the way we did above. On the other hand, the statement we seek to prove is much weaker; we ask only for local boundedness of $\mathcal{Q}_K$, rather than being small in a suitable sense when $|v-\widetilde{v}|, |v_\star-\widetilde{v}_\star|$ are small. It will be sufficient to prove the following slightly simpler lemma, which breaks up $\Q_K$ in a similar way to the decomposition $\mathcal{E}_K=\mathcal{E}^1_K+\mathcal{E}^2_K+\mathcal{E}^3_K$ above.

\begin{lemma}\label{lemma: broken up quadratic pozvner}
    Define \begin{equation}
        \Q_K^{1}=\int_0^\infty dz\int_{\SSd} d\varphi \h\left(d_p^2(v',\widetilde{v'}_K)-d_p^2(v,\widetilde{v})\right)^2;
    \end{equation}\begin{equation}
        \Q_K^{2}=\int_0^\infty dz\int_{\SSd} d\varphi \h\left(d_p^2(v'_\star,\widetilde{v'}_{\star K})-d_p^2(v_\star,\widetilde{v}_\star)\right)^2.
    \end{equation} Then the estimate (\ref{eq: conclusion of QP}) holds with either $\Q^{1}_K$ or $\Q^{2}_K$ in place of $\Q_K$.
\end{lemma} Once we have established these estimates, the second point of Lemma \ref{lemma: quadratic pozvner} follows from the easy comparison $\Q^2_K\le 2\Q^{2,1}_K+2\Q^{2,1}_K$.

\begin{proof}[Proof of Lemmas \ref{lemma: quadratic pozvner}] We use the same notation as above, and start from a decomposition similar to (\ref{eq: decomposition for E}): \begin{equation} \begin{split} d_p^2(v',\widetilde{v}_K')^2-d_p^2(v,\widetilde{v})&=(|v'|^p+|\widetilde{v}'_K|^p-|v|^p-|\widetilde{v}|^p)|v-\widetilde{v}|^2 \\ & \hs + (1+|v'|^p+|\widetilde{v}'_K|^p)(2(a-\widehat{a})\cdot(v-\widetilde{v})+|a-\widehat{a}|^2) \\ & \hs +(1+|v'|^p+|\widetilde{v}'_K|^p)(2\widehat{a}\cdot(v+a-\widetilde{v})+|\widehat{a}|^2)1(z\ge L). \end{split} \end{equation} We now square each term, and use the crude bounds $|a|\le |v|+|v_\star|, |\widehat{a}|\le |\widetilde{v}|+|\widetilde{v}_\star|$ to see that \begin{equation} \begin{split} &(d_p^2(v',\widetilde{v}_K')^2-d_p^2(v,\widetilde{v}))^2 \le c(|v'|^p+|\widetilde{v}'_K|^p-|v|^p-|\widetilde{v}|^p)^2|v-\widetilde{v}|^4 \\ & \hs +c(1+|v'|^p+|\widetilde{v}'_K|^p)^2(|v|^2+|\widetilde{v}|^2+|v_\star|^2+|\widetilde{v}|^2)|a-\widehat{a}|^2\\& \hs +c(1+|v'|^p+|\widetilde{v}'_K|^p)^2(|v|^2+|\widetilde{v}|^2+|v_\star|^2+|\widetilde{v}|^2)|\widehat{a}|^21(z\ge L)).\end{split} \end{equation} We can now replace every instance of $|v'|^p\le C(|v|^p+|v_\star|^p)$, and similarly for $\widetilde{v}'_K$, and drop the factor $1(z\ge L)$ in the final term.  In this way, we obtain \begin{equation} \begin{split} \mathcal{Q}^1_K&\le C\left(\T_5+\T_6+\T_7\right); \end{split} \end{equation} where the three terms are \begin{equation}\begin{split}
   \T_5:= \int_0^\infty dz\int_{\SSd}d\varphi \left||v'|^p-|v|^p+|\widetilde{v}_K'|^p-|\widetilde{v}|^p\right|&\\&\hspace{-3cm}\dots\times(1+|v|^{p+4}+|v_\star|^{p+4}+|\widetilde{v}|^{p+4}+|\widetilde{v}_\star|^{p+4});\end{split}
\end{equation}\begin{equation}
    \T_6:=\int_0^\infty dz\int_{\SSd}d\varphi (1+|v|^{2p+2}+|v_\star|^{2p+2}+|\widetilde{v}^{p+2}|+|\widetilde{v}_\star|^{2p+2})|a-\widehat{a}|^2;
\end{equation} \begin{equation}
   \T_7:=\int_0^\infty dz\int_{\SSd}d\varphi (1+|v|^{2p+2}+|v_\star|^{2p+2}+|\widetilde{v}|^{p+2}+|\widetilde{v}_\star|^{2p+2})|\widehat{a}|^2.
\end{equation}Let us now analyse these integrals one by one. The analysis of $\T_5$ is similar to that of $\T_1$, although with an absolute value, and the integrals appearing in $\T_6, \T_7$ can be reduced to the calculations for $\T_3, \T_4$ in the previous proof. 
\paragraph{1. Analysis of $\T_5$.} We start from the observation that, for all $v, w\in \mathbb{R}^d,$ we have \begin{equation}||v|^p-|w|^p|\le C(1+|v|^{p-1}+|w|^{p-1})|v-w|. \end{equation}  It follows that \begin{equation}\begin{split} \left||v'|^p-|v|^p\right|&\le C(1+|v|^{p-1}+|v+a|^{p-1})|a| \\ & \le C(1+|v|^{p-1}|+|v_\star|^{p-1})(|v|+|v_\star|)G\left(\frac{z}{x^\gamma}\right).\end{split}\end{equation} Integrating, we find that \begin{equation}\begin{split}\int_0^\infty dz\int_{\SSd}d\varphi \left||v'|^p-|v|^p\right| &\le C(1+|v|^p+|v_\star|^p)\int_0^\infty G(z/x^\gamma)dz \\& \le C(1+|v|^{p+\gamma}+|v_\star|^{p+\gamma}).\end{split} \end{equation} A similar argument applies for $||\widetilde{v}'_K|^p-|\widetilde{v}|^p|.$ Including the moment prefactors, we obtain \begin{equation}\label{eq: bound on Q1} \T_5\le C(1+|v|^{2p+4+\gamma}+|v_\star|^{2p+4+\gamma}+|\widetilde{v}|^{2p+4+\gamma}+|\widetilde{v}_\star|^{2p+4+\gamma}). \end{equation} \paragraph{2. Analysis of $\T_6$.} For $\T_6$, we note that the moment prefactor is constant over the integral, and that we already analysed $\int_0^\infty dz\int_{\SSd} |a-\widehat{a}|^2$ when analysing $\T_3$ in the previous proof. Absorbing the terms $|v-\widetilde{v}|^2$ and $|v_\star-\widetilde{v}_\star|^2$, the same calculations as above therefore give \begin{equation} \label{eq: bound on Q2} \Q_2\le C(1+|v|^{2p+4+\gamma}+|v_\star|^{2p+4+\gamma}+|\widetilde{v}|^{2p+4+\gamma}+|\widetilde{v}_\star|^{2p+4+\gamma}). \end{equation} \paragraph{3. Analysis of $\T_7$.} As above, the moment prefactor is independent of the integration variables $z,\varphi$, and the problem reduces to estimating $\int_L^\infty \int_{\SSd}|\widehat{a}|^2$, which is analagous to $\T_4$. We recall that \begin{equation} |\widehat{a}|^2=\frac{1}{2}|\widetilde{v}-\widetilde{v}_\star|^2\left(1-\cos G\left(\frac{z}{\widetilde{x}} \right)\right)\le \frac{1}{4}|v-\widetilde{v}|^2 G\left(\frac{z}{\widetilde{x}^\gamma}\right)^2.\end{equation} Therefore, \begin{equation} \int_0^\infty dz\int_{\SSd}d\varphi |\widehat{a}|^2 \le C|v-\widetilde{v}|^2|v-\widetilde{v}|^\gamma \int_0^\infty G(z)^2dz.   \end{equation} The final integral is finite, thanks to the estimates established in Subsection \ref{subsec: estimates for G}, so we conclude \begin{equation}\label{eq: bound on Q3}\T_7 \le C(1+|v|^{2p+4+\gamma}+|v_\star|^{2p+4+\gamma}+|\widetilde{v}|^{2p+4+\gamma}+|\widetilde{v}_\star|^{2p+4+\gamma}). \end{equation} Combining (\ref{eq: bound on Q1}, \ref{eq: bound on Q2}, \ref{eq: bound on Q3}) gives the claimed result. \end{proof}
 
 \begin{appendix}
 	\section{Appendix: Proof of Propositions \ref{prop: labelled and unlabelled dynamics}, \ref{prop: wellposedness/existence}, \ref{prop: wellposedness/uniqueness}} \label{sec: pf_wellposedness} We finally address the well-posedness issues regarding the labelled and unlabelled Kac processes, which have been deferred. We will now prove Propositions \ref{prop: labelled and unlabelled dynamics}, which describes the relationships between the labelled and unlabelled dynamics, and Propositions \ref{prop: wellposedness/existence}, \ref{prop: wellposedness/uniqueness}, which assert a moderate well-posedness for the stochastic differential equation (\ref{eq: SDE form}) and of the martingale problem for the generator (\ref{eq: generator}) of the unlabelled dynamics.\bigskip \\ Our strategy is as follows. The first item of Proposition \ref{prop: labelled and unlabelled dynamics} is elementary, and relies on a consistency between the unlabelled and labelled generators $\mathcal{G}, \mathcal{G}^\mathrm{L}$; for the second item, we carefully state a result of Kurtz \cite{kurtz1998martingale,kurtz2011equivalence} and show how it applies in our case. For Proposition \ref{prop: wellposedness/existence}, we can show existence by standard techniques for martingale problems, using tightness and consistency of the generators; this does not use any result in the paper, and can be read independently of the more delicate esimtates. For uniqueness in Proposition \ref{prop: wellposedness/uniqueness}, we use the coupling and estimates in Section \ref{sec: TK}, which we emphasise do not rely on this result. We do not seek any estimates uniformly in $N$, and we can replace moment estimates with the trivial bound $|V^i_t|\le\sqrt{N}$. For ease of presentation, we will use the estimates we have already developed in this paper, although those from the literature \cite{fournier2016rate} would work equally well.
\bigskip \\ Let us recall some notation which will be needed. We will frequently move between objects defined on the labelled Kac sphere \begin{equation}
    \mathbb{S}_N=\left\{\mathcal{V}^N=(V^1,...V^N)\in (\RRd)^N, \h \sum_{i=1}^N V^i=0, \h \sum_{i=1}^N |V^i|^2=N\right\}
\end{equation} and the unlabelled state space $\mathcal{S}_N$; we recall that $\theta_N:\mathbb{S}_N\rightarrow \mathcal{S}_N$ is the map  \begin{equation} \label{eq: define thetaN}\mathcal{V}^N=(V^1,...,V^N)\mapsto \frac{1}{N}\sum_{i=1}^N\delta_{V^i}.\end{equation} For clarity, we will indicate functions on $\mathbb{S}_N$ with a $\widehat{\cdot}$ to distinguish them from those on $\mathcal{S}_N$. We will equip $\mathcal{S}^N$ with the distance \begin{equation}
    |\mathcal{V}^N-\mathcal{W}^N|:=\sum_{i=1}^N|V^i-W^i|
\end{equation} where the right-hand side is the Euclidean norm on $\mathbb{R}^d$. We will write $W^{1,\infty}(\mathbb{S}_N)$ for the Sobolev space of functions $\widehat{F}:\mathbb{S}_N\rightarrow \mathbb{R}$ which are Lipschitz with respect to this distance, equipped with the norm \begin{equation}
    \|\widehat{F}\|_{W^{1,\infty}(\mathbb{S}_N)}:=\max\left(\sup_{\mathcal{V}^N}|\widehat{F}(\mathcal{V}^N)|, \sup_{\mathcal{V}^N\neq \mathcal{W}^N}\frac{|\widehat{F}(\mathcal{V}^N)-\widehat{F}(\mathcal{V}^N)|}{|\mathcal{V}^N-\mathcal{W}^N|}\right)
\end{equation} and define $W^{1,\infty}(\mathcal{S}_N)$ similarly, equipping $\mathcal{S}_N$ with the Wasserstein$_1$ distance $\mathcal{W}_1$. It is elementary to show that these spaces are separable. Let us also recall, for convenience, the generators of the labelled and unlabelled dynamics, given respectively by \begin{equation}  \label{eq: generator recall}
     (\mathcal{G}^NF)(\mu^N)=N\int_{\mathbb{R}^d\times\mathbb{R}^d\times \mathbb{S}^{d-1}} (F(\mu^{N, v,v_\star,\sigma})-F(\mu^N))B(v-v_\star,\sigma)\mu^N(dv)\mu^N(dv_\star)d\sigma;
 \end{equation}  \begin{equation} \label{eq: labelled generator 3} (\mathcal{G}^\mathrm{L}\widehat{F})(\mathcal{V}^N)=\frac{1}{N}\sum_{i=1}^N\sum_{j=1}^N\int_{\mathbb{S}^{d-1}}\left(\widehat{F}(\mathcal{V}^N_{i,j,\sigma})-\widehat{F}(\mathcal{V}^{N})\right)d\sigma \end{equation} for Lipschitz functions $F\in W^{1,\infty}(\mathcal{S}_N), \widehat{F}\in W^{1,\infty}(\mathbb{S}_N)$ respectively.  With this notation fixed, we turn to the proof of the two propositions.   \begin{proof}[Proof of Proposition \ref{prop: labelled and unlabelled dynamics}] For item i)., observe the following consistency between the unlabelled generator (\ref{eq: generator recall}) and labelled generator (\ref{eq: labelled generator 3}), which follows from the $\text{Sym}(N)$ symmetry of the labelled dynamics: if $F\in W^{1,\infty}(\mathcal{S}_N)$, then $\widehat{F}:=F\circ \theta_N\in W^{1,\infty}(\mathbb{S}_N)$, and \begin{equation} \label{eq: consistency of generators} \mathcal{G}^L\left(F\circ \theta_N\right)=\left(\mathcal{G} F\right)\circ \theta_N. \end{equation} Now, let $\mathcal{V}^N_t$ be a labelled Kac process, for some filtration $(\mathcal{F}_t)_{t\ge 0}$; it follows that $\mathcal{V}^N_t$ solves the martingale problem for (\ref{eq: labelled generator 3}) for the same filtration. Now, let $\mu^N_t=\theta_N(\mathcal{V}^N_t)$ be the associated empirical measures, and fix $F\in W^{1,\infty}(\mathcal{S}_N)$. For $\widehat{F}=F\circ \theta_N$ as above, the consistency (\ref{eq: consistency of generators}) gives \begin{equation}
     \begin{split}
         F(\mu^N_t)-F(\mu^N_0)-\int_0^t (\mathcal{G}F)(\mu^N_s)ds = \widehat{F}(\mathcal{V}^N_t)-\widehat{F}(\mathcal{V}^N_0)-\int_0^t (\mathcal{G}^\mathrm{L}\widehat{F})(\mathcal{V}^N_s)ds.
     \end{split}
 \end{equation} The right-hand side is a martingale by assumption, and hence $\mu^N_t$ solves the martingale problem for (\ref{eq: generator recall}) in the filtration $(\mathcal{F}_t)_{t\ge 0}$, as desired; in particular, $\mu^N_t$ is a Markov process with generator (\ref{eq: generator recall}). \bigskip \\ For item ii), we will use the following result, which generalises the implication needed, due to Kurtz \cite{kurtz1998martingale,kurtz2011equivalence}. Let us first fix some terminology. For a topological space $E$, let us write $\overline{C}(E)$ for the space of bounded, continuous functions on $E$, $B(E)$ for the space of bounded, Borel-measurable functions on $E$, and $\mathcal{P}(E)$ for the space of Borel probability measures. Given another such space $E_0$, a transition function $\alpha$ from $E_0$ to $E$ is a mapping from $E_0\rightarrow \mathcal{P}(E)$ such that, for all Borel sets $A\subset E$, the map $y\mapsto \alpha(y, A)$ is a Borel function on $E_0$; for such $\alpha$ and $f\in B(E)$, define $\alpha f \in B(E_0)$ by \begin{equation} (\alpha f)(y):=\int_E f(z)\alpha(y, dz). \end{equation} We will write $M_E[0,\infty), D_E[0,\infty)$ for the measurable, respectively c\`ad\`ag functions from $[0,\infty)$ to $E$. \bigskip \\ Let us say that a linear operator $\mathcal{A}\subset B(E)\times B(E)$ is seperable if there exists a countable subset $\{f_\beta, \beta\ge 1\}\subset \mathcal{D}(\mathcal{A})$ such that, for all $(f,g)\in \mathcal{A}$, there exists a subsequence $\beta_i\rightarrow \infty$ such that $(f_i,\mathcal{A}f_i)$ are bounded uniformly in $i$, and converge pointwise to $(f,g)$. We say that a linear operator $\mathcal{A}$ is a pregenerator if it is dissipative, and there exists a sequence of functions $q_n: E\rightarrow \mathcal{P}(E), r_n:E\rightarrow [0,\infty)$ such that, for all $f\in \mathcal{D}(\mathcal{A})$, we have the pointwise convergence \begin{equation}
    r_n(x)\int_E(f(y)-f(x))q_n(x,dy)\rightarrow (\mathcal{A}f)(x)\hs \text{ for all }x\in E.
\end{equation}  With these definitions, we can state the following result, which appears as part of \cite[Theorem 1.4]{kurtz2011equivalence} \begin{proposition}\label{prop: equivalence} Let $(E,r), (E_0, r_0)$ be complete, separable metric spaces. Let $\mathcal{A}\subset \overline{C}(E)\times \overline{C}(E)$ be a linear operator which is seperable and a pre-generator, and whose domain $\mathcal{D}(\mathcal{A})$ separates points in $E$. Suppose that $\theta: E\rightarrow E_0$ is Borel measureable, and $\alpha$ is a transition function from $E_0$ to $E$ satisfying the compatibility condition $\alpha(y,\theta^{-1}(y))=1$ for all $y\in E_0$. Let $\mathcal{A}^\theta$ be the linear operator \begin{equation}
    \mathcal{A}^\theta=\left\{\left(\alpha f, \alpha(\mathcal{A}f)\right):f\in \mathcal{D}(\mathcal{A})\right\}\subset B(E_0)\times B(E_0).
\end{equation} Let $\mathcal{L}_0\in \mathcal{P}(E_0)$, and let $\widetilde{\mathcal{L}}_0=\alpha_\#\mathcal{L}_0 \in \mathcal{P}(E)$ be given by \begin{equation} \widetilde{\mathcal{L}}_0(A)=\int_{E_0} \alpha(y,A)\mathcal{L}_0(dy). \end{equation} If $\widetilde{\mu}=(\widetilde{\mu}_t)_{t\ge 0}$ is a solution of the martingale problem for $(\mathcal{A}^\theta, \mathcal{L}_0)$, then there exists a solution $\mathcal{V}$ of the martingale problem for $(\mathcal{A}, \widetilde{\mathcal{L}}_0)$ such that $\widetilde{\mu}$ has the same law on $M_{E_0}[0,\infty)$ as $\mu=\theta\circ \mathcal{V}$. Further, if $\widetilde{\mu}$, and hence $\mu$, has a modification with sample paths in $D_{E_0}[0,\infty)$, then the modified $\widetilde{\mu},\mu$ have the same law on $D_{E_0}[0,\infty).$ \end{proposition} Let us now show how this applies in our case. We will take $E, E_0$ to be the labelled and unlabelled Kac spheres $E=\mathbb{S}_N, E_0=\mathcal{S}_N$ respectively, equipped with the metrics as above. We take $\mathcal{A}$ to be the labelled generator $\mathcal{G}^L$ given by (\ref{eq: labelled generator1}), defined on $F\in W^{1,\infty}(\mathbb{S}_N)$, and let $\theta=\theta_N$ be given by (\ref{eq: define thetaN}). We define $\alpha$ as the average over the preimage \begin{equation}
    \alpha(\mu^N)=\frac{1}{\#\theta_N^{-1}(\mu^N)}\h\sum_{\mathcal{V}^N\in \theta_N^{-1}(\mu^N)}\delta_{\mathcal{V}^N}.
\end{equation} We remark that, if $\mu^N\in \mathcal{S}_N$ and $\mathcal{V}^N\in\theta_N^{-1}(\mu^N)$, then $\alpha(\mu^N)$ can be rewritten \begin{equation}\label{eq: rewrite alpha}
    \alpha(\mu^N)=\frac{1}{N!}\sum_{\pi \in \text{Sym}(N)} \delta_{\mathcal{V}^{N,\pi}}
\end{equation} where $\mathcal{V}^{N,\pi}$ denotes the action of $\pi\in \text{Sym}(N)$ permuting the $N$ components $V^1,..,V^N\in\RRd$ of $\mathcal{V}^N$. It is elementary, if somewhat tedious, to check that with these choices, the linear operator $\mathcal{A}^\theta$ is exactly the unlabelled generator $\mathcal{G}$, defined on $W^{1,\infty}(\mathcal{S}_N)$; the inclusion $\mathcal{G}\subset \mathcal{A}^\theta$ is exactly the statement (\ref{eq: consistency of generators}), and for the other inclusion $\mathcal{A}^\theta\subset \mathcal{G}$, we  use (\ref{eq: rewrite alpha}) to check that, for $\widehat{F}:\mathbb{S}_N\rightarrow \mathbb{R}$ Lipschitz, $\alpha \widehat{F}: \mathcal{S}_N\rightarrow \mathbb{R}$ is Lipschitz, and straightforward calculations show that $\mathcal{G}(\alpha \widehat{F})=\alpha(\mathcal{G}^L \widehat{F})$ as desired. \bigskip \\ To see that $\mathcal{A}=\mathcal{G}^L$ is separable, we note that $W^{1,\infty}(\mathbb{S}_N)$ is separable, and $\mathcal{G}^\mathrm{L}:W^{1,\infty}(\mathbb{S}_N)\rightarrow L^\infty(\mathbb{S}_N)$ is a bounded linear map. Its graph is therefore separable in the stronger topology induced by $\mathcal{G}^\mathrm{L}\subset W^{1,\infty}(\mathbb{S}_N)\times L^\infty(\mathbb{S}_N)$, and so is separable in the topology of bounded pointwise convergence in the definition above.  \bigskip \\ To see that $\mathcal{G}^\mathrm{L}$ is a pregenerator, let us define $\mathcal{G}^\mathrm{L}_K$ to be the cutoff equivalent, replacing $B$ by the cutoff kernel $B_K$ (\ref{eq: cutoff kernel}). It is straightforward to write $\mathcal{G}^\mathrm{L}_K$ in the form desired, and $\mathcal{G}^\mathrm{L}_K\rightarrow \mathcal{G}_K$ in the space of bounded linear maps $\mathcal{B}(W^{1,\infty}(\mathbb{S}_N),L^\infty(\mathbb{S}_N))$. Elementarily, each $\mathcal{G}^\mathrm{L}_K$ is the generator of a cutoff, labelled Kac process, and so generates a semigroup of contraction mappings; by the Lumer-Phillips Theorem, they are therefore dissipative; we can then take a limit to conclude that $\mathcal{G}^\mathrm{L}$ is dissipative, and so is a a pregenerator.  \bigskip \\ We can now apply the conclusion of Proposition \ref{prop: equivalence} above. Let us fix $\mu^N_0 \in \mathcal{S}_N$, and let $(\widetilde{\mu}^N_t)_{t\ge 0}$ be a solution to the martingale problem for the unlabelled generator (\ref{eq: generator}) starting at $\mu^N_0$. The law $\widetilde{\mathcal{L}}_0$ given by Proposition \ref{prop: equivalence} exactly corresponds to picking $\mathcal{V}^N_0\in \theta_N^{-1}(\mu^N_0)$ uniformly at random, as in the statement of the proposition, and by the result quoted above, there exists a solution to the martingale problem for (\ref{eq: labelled generator1}), starting at $\mathcal{V}^N_0$ such that $\widetilde{\mu}^N_t$ has the same law as $\theta_N(\mathcal{V}^N_t)$. $\mathcal{V}^N_t$ is therefore a weak solution to the stochastic differential equation (\ref{eq: SDE form}), and so we have proven the claim of item ii).  \end{proof}  \begin{proof}[Proof of Proposition \ref{prop: wellposedness/existence}] Let us fix $\mathcal{V}^N_0$; for each $K$, $\mathcal{V}^{N,K}_t$ be a solution to (\ref{eq: cutoff SDE}), starting at $\mathcal{V}^N_0$, with cutoff parameter $K$. Since the rates are finite, such processes can be constructed elementarily, and have uniqueness in law. We check tightness via Aldous' criterion; thanks to the energy constraint, each $\mathcal{V}^{N,K}_t$ takes values in $[-N^{1/2},N^{1/2}]^{Nd}$, and for equicontinuity, we estimate \begin{equation} \label{eq: uicontinuity} \int_0^\infty dz\int_{\SSd}d\varphi \h |a(v,v_\star,z,\varphi)| \le C|v-v_\star|^{1+\gamma}\le C(1+|v|^2+|v_\star|^2).\end{equation}  As above, let $\mathcal{G}^\mathrm{L}, \mathcal{G}^\mathrm{L}_K$ be the (noncutoff/cutoff) labelled generators, and fix $\widehat{F}\in W^{1,\infty}(\mathbb{S}_N)$. As mentioned above, is straightforward to show that $\mathcal{G}^L_K\widehat{F}$ are continuous, and converge uniformly to $\mathcal{G}^\mathrm{L}\widehat{F}$; it follows that any subsequential limit point of $\mathcal{V}^{N,K}_t$, as $K\rightarrow \infty$, is a  solution to the martingale problem for (\ref{eq: labelled generator1}), and hence is a weak solution to (\ref{eq: SDE form}). \end{proof} \begin{proof}[Proof of Proposition \ref{prop: wellposedness/uniqueness}] For uniqueness in law, let $\mathcal{V}^{N}_t$ be any solution to (\ref{eq: SDE form}) starting at $\mathcal{V}^N_0$. We now apply Lemma \ref{lemma: Tanaka Coupling of KP}; fix $p>p_0(G,d), K>K_0(G,p,d)$ as in the statement, and $0\le t_1<...<t_m$, we take $b=N^{(p+\gamma)/2}$, so that $T^N_b=T^{N,K}_b=\infty$. The cited lemma now shows that $(\mathcal{V}^N_{t_i})_{i\le m}$ is the limit in probability, of $(\mathcal{V}^{N,K}_{t_i})_{i\le m}$, for cutoff labelled Kac processes $\mathcal{V}^{N,K}_t$ starting at $\mathcal{V}^N_0$, as $K\rightarrow \infty$. Since the law of each $\mathcal{V}^{N,K}_t$ is uniquely determined, the same is true of the $m$-tuple $(\mathcal{V}^N_{t_i})_{i\le m}$. Since $t_i$ were arbitrary, we conclude that the law of $\mathcal{V}^N_t$ is unique, as claimed. \end{proof}
 \end{appendix}

\end{document}